\documentclass[11pt]{amsart}
\usepackage{geometry}                
\usepackage{graphicx,
	amssymb,
	amsmath,
	amsthm,
	bm,
	tikz-cd,
        dsfont,
	epstopdf,
	enumitem,
	mathrsfs,
	mathabx,
	multirow,
	mathtools,
	dsfont,
	xcolor,
	tabu,
	float,
	subfig,
	hyperref}
	\usetikzlibrary{matrix}
\hypersetup{
    colorlinks=true,
    allcolors=blue
}

\usepackage{amsopn, amsthm, amsgen, amscd,amsmath, amssymb}
\usepackage[english]{babel}
\usepackage{fancybox}
\usepackage[all]{xy}
\usepackage{xcolor}
\usepackage{verbatim}

\theoremstyle{definition}
\newtheorem{definition}{Definition}

\newtheorem{example}{Example}

\newtheorem{remark}{Remark}

\theoremstyle{plain}
\newtheorem{theorem}{Theorem}
\newtheorem{corollary}{Corollary}
\newtheorem{proposition}{Proposition}
\newtheorem{lemma}{Lemma}

\newcommand{\lra}{\longrightarrow}

\newcommand{\ie}{{\emph{i.e, }}}
\newcommand{\noi}{{\noindent}}
\newcommand{\su} {\mathbf{U}}
\newcommand{\gl} {\mathbf{GL}}
\newcommand{\Z}{\mathbb{Z}}
\newcommand{\C}{\mathbb{C}}
\newcommand{\R}{\mathbb{R}}
\newcommand{\CP}{{\mathbb P}_{\C}}
\newcommand{\T} {\mathbb{T}}
\newcommand{\vdu}{\mathcal{N}}
\newcommand{\vd} {\mathcal {M}}
\newcommand{\lv} {\mathbf {LVM}}

\newcommand{\lam}{\pmb{\lambda}}
\newcommand{\Lam}{\pmb{\Lambda}}
\newcommand{\ZLam}{Z^{^\C}(\Lam)}

\newcommand{\s}{\mathbb{S}}
\newcommand{\disc}{\mathbb{D}}

\newcommand{\p} {\mathcal{P}}

\newcommand{\te}{\tilde{\mathcal{E}}}

\newcommand{\gato}{\mathop{\pmb{\#}}\limits}

\title[LVM Manifolds]{Intersection of quadrics in $\C^n$,  moment-angle manifolds, complex manifolds and convex polytopes.}
\author{Alberto Verjovsky}
\address{Instituto de Matem\'{a}ticas, Unidad Cuernavaca, Universidad
Nacional Aut\'{o}noma de M\'{e}xico, Av. Universidad S/N, C.P. 62210
Cuernavaca, Morelos, M\'{e}xico}
\email{alberto@matcuer.unam.mx}

\date{}

\begin{document}

\begin{abstract} These are notes for the 
\emph{CIME school on Complex non-K\"ahler geometry}
from July $9^{th}$ to July $13^{th}$ of 2018 in Cetraro, Italy. It is an overview of different properties of a class
of non-K\"ahler compact complex manifolds called LVMB manifolds, obtained as the Hausdorff space of leaves
of systems of commuting complex linear equations in an open set in complex projective space $\CP^{n-1}$.

\end{abstract}

\maketitle
\tableofcontents

\section{Introduction}

The origin of the so-called LVM manifolds is the paper \cite{LdMV} by Santiago L\'opez de Medrano and the author of these notes. There they define and study a new infinite family of compact complex manifolds 
  (a finite number of diffeomorphism classes for each dimension) which, except for a series of cases corresponding to complex tori, are not symplectic.
  The construction is based on the following principle discovered by Andr\'e Heafliger: \\
  
  \noi \emph{If $\mathcal F$ is a holomorphic foliation of complex codimension $m$ on a complex manifold $M$ with $m\leq{n}=dim_\C\,{M}$ 
  and $\Sigma$ is a $C^{\infty}$
  manifold of real dimension $2m$ which is transversal to $\mathcal F$ then $\Sigma$ is a complex manifold. Indeed
it suffices to provide $\Sigma$ with a holomorphic atlas from transversals 
to the plaques of a foliation atlas of $\mathcal F$.}\\

\noi The essential point is that one can obtain {\it non-algebraic} complex manifolds as the space of leaves of holomorphic foliations
of complex \emph{algebraic }manifolds, as long as the space of leaves is Hausdorff. In particular the foliation could be given by a 
holomorphic action of a complex Lie group. In fact the construction in  \cite{LdMV} uses an explicit linear action of $\C$ in $\C^n$ ($n\geq3$) which descends to a projective linear action on complex projective space $\CP^{n-1}$  and there is an open set $\mathcal V\subset\CP^{n-1}$ which is invariant under the action and such that every leaf (orbit) of the action is an immersed copy of $\C$ or $\C^*$; furthermore, the space of leaves of the foliation by orbits of $\mathcal{V}$ is compact and Hausdorff and therefore it is a compact complex manifold. 
In some sense the set 
$\mathcal{V}$ is the union of ``semi-stable orbits''  (or Siegel leaves) of the action in the sense 
of \emph{Geometric Invariant Theory} (GIT) and is the complement of a union of projective subspaces of different dimensions. 
In fact $\mathcal{V}$ is the image under the canonical projection $\C^n\to\CP^{n-1}$ of the set of orbits in $\C^n$ that do not accumulate to the origin (a sort of Kempf-Ness condition). 
In a very pretty paper \cite{CFZ} St\'ephanie Cupit-Foutou
and Dan Zaffran describe how to construct the generalized family of LVMB manifolds from certain Geometric Invariant Theory (GIT) quotients. 

\noi They show that Bosio's generalization parallels exactly the extension obtained by Mumford's GIT to the more general GIT developed by Bia\l{}ynicki-Birula and \'Swiecicka.

\noi The article \cite{LdMV} is a continuation of the foundational papers by J. Girbau, A. Haefliger et D. Sundararaman \cite{GHS}
about the deformations of foliations which are transversally holomorphic. In fact,
  Andr\'e Haefliger used these results to study in
  \cite{Ha} the deformations of Hopf manifolds which are realized as the space of leaves of a foliation. 
  
\noi Another continuation of that work was obtained by Jean-Jacques Loeb and Marcel Nicolau which uses the foliation in order to describe the deformations of the Calabi-Eckmann  manifolds \cite{LN}.
 
 \noi The initial construction in \cite{LdMV} can be extended to the case of projective linear actions of $\C^{m}$ for any positive integer $m$  
 on $\CP^n$ as long as $n>2m$. Then, under two assumptions related to the $n\times{m}$ complex matrix  $\Lam$  of eigenvalues of the linear flows which determine the action one obtains new compact manifolds. These assumptions are that $\Lam$ be admissible \ie it satisfies
the \emph{weak hyperbolicity and Siegel conditions}.
 This was achieved by Laurent Meersseman who studies in detail several aspects of the compact manifolds in \cite{Me}. These compact complex manifolds $N_{\Lam}$ are now known as LVM manifolds.  A very interesting property of these manifolds when $m>1$ is their very rich topology. For instance, any finite abelian group is a summand of the homology group of one of these manifolds. In particular some of the manifolds have arbitrarily large torsion in its homology groups.
In \cite{Bosio}, Fr\'ed\'eric Bosio gives a generalization of the construction of LVM manifolds.  
The idea is to relax the weak hyperbolicity and Siegel conditions $\Lam$ and to look for all the subsets $\mathcal{S}$ of $\C^n$ such that action \eqref{actionLVM} in section \ref{linearactions}
is free and proper. The manifolds that are either LVM manifolds or the generalization by Fr\'ed\'eric Bosio are now known as LVMB manifolds. 
The manifolds $N_{\Lam}$ are obtained as the orbit space of a free action of the circle on an odd-dimensional manifold $M_1(\Lam)$ contained in the sphere $\mathbb{S}^{2n-1}$
which is the intersection of homogeneous quadratic equations and called \emph{moment-angle manifold}. Santiago L\'opez de Medrano has studied deeply these intersection of quadrics in several papers by himself and some collaborators 
\cite{ GLdM, GoLdM, GoLdM2, LdM1, LdM2, LdM4, LdM5} in particular the paper \cite{GLdM} by Samuel  Gitler and Santiago L\'opez de Medrano has been a great advance to understand the topology of moment-angle manifolds.\\

\noi The LVM manifolds are not simplectic (except when $2m+1=n$ when the manifolds are compact complex tori) however under an arbitrarily small deformation
 of $\Lam$ the manifolds $N_{\Lam}$ fiber \`a la Seifert-Orlik over a toric manifold (or orbifold) with fiber a compact complex
 torus. This is due to the following fact: 
 
 \noi The complex manifolds $N_{\Lam}$ of complex dimension $n-m-1$ admit a locally-free holomorphic action of $\C^m$ (recall that $n>2m$); although $N_{\Lam}$ is not K\"ahler, the foliation 
 ${\mathcal G}_{\Lam}$ on $N_{\Lam}$ by the leaves of the action is transversally K\"ahler, in particular  
 ${\mathcal G}_{\Lam}$ is a Riemannian foliation and thus admits a transverse invariant volume form.
 In particular either has a Zariski open set of noncompact leaves or else all are compact complex tori.
 
 \noi If $\Lam$ satisfies a rationality condition called condition  
 {\bf K} in definition \ref{System(K)} then all the leaves of  ${\mathcal G}_{\Lam}$ are compact, in fact they are complex tori
 ${\C^m}/\Gamma$ ($\Gamma\cong\Z^{2m}$) and the quotient is Hausdorff. Hence it is a compact complex manifold (or an orbifold). 
 
 \noi Furthermore, the rationality conditions {\bf K} in definition \ref{System(K)} imply that the transversal K\"ahler form is ``integral'' (a sort of transversal Kodaira embedding condition) which makes this quotient an algebraic manifold or variety with quotient singularities of dimension $n-2m-1$. In fact this quotient admits an action of $(\C^*)^{n-2m-1}$ with a principal
 dense orbit so it is a toric manifold $X(\Delta)$ where $\Delta=\Delta_{\Lam}$ is the corresponding fan which depends on $\Lam$.

\noi  The  reciprocal is true as shown by the author and Laurent Meersseman \cite{MV}: If $X(\Delta)$ is a toric variety with at most singularities which are quotients then there exists an admissible configuration $\Lam$ which satisfies conditions {\bf K} in definition \ref{System(K)} and therefore any toric variety with at most quotient singularities is obtained by the quotient of a LVM manifold
  by a holomorphic locally-free action of a compact complex torus. In this paper one uses Delzant construction over a rational simple convex polytope which is naturally associated to the convex hull $\mathcal{H}(\Lam)$ of the configuration.
When the leaves of  ${\mathcal G}_{\Lam}$ are not compact the leaf space is not Hausdorff and one has a ``noncommutative'' complex manifold in the sense of Alain Connes \cite{Connes, Connes2}.  This happens when the convex polytope $\mathcal{H}_{\Lam}$ is not rational and a convex polytope associated to the foliation 
${\mathcal G}_{\Lam}$ is non-rational. There are important reasons to consider \emph{nonrational} polytopes.
For instance, toric varieties corresponding to simple rational  polytopes are rigid (\ie they cannot be deformed) whereas simple
rational polytopes can be perturbed simply by moving the vertices to non-rational simple polytopes. The problem of associating to a non-rational polytope a geometric space of some kind is an old one and emerges in different subjects, including symplectic geometry, via the convexity theorem and the Delzant construction. In fact it also is connected with the combinatorics of convex polytopes see for instance Stanley \cite{St} where a link between rational simplicial polytopes and the geometry and topology of toric varieties is explained following earlier work of Peter McMullen \cite{McM}
and  Richard P. Stanley \cite{Sta}. There are important reasons to consider \emph{nonrational} simple polytopes and its and give them an interpretation in relation to toric geometry. In this respect the 
article by Elisa Prato \cite{Pr} is the first work that addresses this problem via symplectic geometry and she defines the notion of
\emph{quasifolds}, which is a generalization of the notion of orbifolds and associates to a non-rational simple polytope a quasifold.  In a joint paper \cite{BP2} Elisa Prato and Fiammetta Battaglia generalize the notion of toric variety and associate to each non-rational simplicial polytope a K\"ahler quasifold and compute the Betti numbers (see also \cite{BP2}). The paper by Fiammetta Battaglia and Dan Zaffran
\cite{BaZ} uses also the leaf space of the foliation ${\mathcal G}_{\Lam}$ of the manifolds $N_{\Lam}$ to have either toric orbifolds in the rational case or quasifolds in the non-rational case. Thus the papers \cite{BaZ, BP1, BP2, Pr} are foundational papers in the theory of non-rational polytopes. 

\noi In \cite{KLMV} a different interpretation as \emph{non-commutative toric varieties} is given of the pair 
$(N_{\Lam}, \mathcal{G}_{\Lam})$ in the case $\Lam$ does not satisfy the rational condition ({\bf K}). Non-commutative toric varieties are to toric varieties what non-commutative
tori are to tori and, as such, they can be interpreted in multiple ways: As (noncommutative)
topological spaces, they are $C^*$-algebras associated to dense foliations,
that is to say, deformations of the commutative $C^*$-algebras associated to
tori in the spirit of Alain Connes. 

\noi However, while non-commutative tori correspond to linear foliations (deformations) on
classical tori, non-commutative toric varieties correspond to the holomorphic foliation  ${\mathcal{G}}_{\Lam})$ on 
$N_{\Lam}$. 

\noi The manifolds $N_{\Lam}$ are certain intersections of real quadrics in complex projective
spaces of a very explicit nature. The homotopy type of LVM-manifolds is
described by moment angle complexes.\\

\noi The paper of Fr\'ed\'eric Bosio and Laurent Meerseman \cite{BM} is a beautiful paper with many ideas and interconnection of several branches of mathematics. In fact the title and subject of the present notes is very much inspired on this paper. 

\noi They do a deep study of the properties of of LVM
manifolds and also made significant advances in the study of the topology of the
intersection of $k$ homogeneous quadrics. 

\noi In particular, the question of whether they
are always connected sums of sphere products was considered: they produced new
examples for any $k$ which are so, but also showed how to construct many more
cases where they are not. 

\noi Independently in \cite{DJ} Michael W. Davis and Janusz Januszkiewicz had
introduced new constructions, part of
which essentially coincide with those above, where the main
objective was the study of some important quotients of them
(different from the ones mentioned above) which they called
\emph{toric manifolds} (in contrast with toric \emph{varieties} that are algebraic). These toric manifolds are topological 
analogues of toric varieties in algebraic geometry. They
are even dimensional manifolds with an effective action of an $n$-dimensional compact torus $({\s^1})^{^n}$, there is a kind of ``moment map'' and the
orbit space is a simple convex polytope.
One can do combinatorics on the quotient polytope to obtain information on the manifold above. For example one can compute the Euler characteristic and describe the cohomology ring of the manifold in terms of the polytope.
The paper by Davis and Januszkiewicz originated an important
development through the work of many authors, for which we refer the reader
to the book of Victor M. Buchstaber and Taras E.  Panov  \cite{BP}.  A line of research derived from \cite{DJ} is the paper \cite{BBCG} where a far-reaching generalization is made and a general splitting
formula is derived that provides a very good geometric tool for understanding the
relations among the homology groups of different spaces.

\noi There is a principal circle bundle $p:M_1(\Lam)\to{N_{\Lam}}$   over each manifold $N_{\Lam}$. 
The manifold $M_1(\Lam)$ is a smooth manifold of real odd dimension $2n-2m-1$ called \emph{moment-angle manifold}. The manifold  $M_1(\Lam)$ admits an action of the torus $({\mathbb{S}}^1)^n$. The orbit space of this action is a convex polytope
of dimension $3n-2m-1$ thus there is a ``moment map''. In addition  $M_1(\Lam)$ has in a contact structure and in many cases
is an open book with a very interesting structure \cite{BLdMV}.

\noi In the present notes we present various results and properties of LVMB manifolds:
 
 \noindent I. Complex analytic proprieties  
 
\noindent II. The relation between these manifolds and toric manifolds and orbifolds with quotient singularities.
 
\noindent III. Their topology and geometric structures
 
\noi The main body of the results presented in these notes are in part contained in the papers \cite{BLdMV, BM, LdMV, Me, MV, MV2}.

 \section{Singular holomorphic foliations of $\C^n$ and $\CP^{n-1}$ given by linear holomorphic actions of $\C^m$ on $\C^n$ ($n>2m$)}\label{linearactions}
Let $M$ be a complex manifold of complex dimension $n$ and $0\leq{p}\leq{n}$.
  
\begin{definition}\label{hfoliation} A holomorphic  foliation $\mathcal F$ of complex dimension $p$ (or complex codimension $n-p$) is given 
by a {\em foliated atlas} $(U_\alpha, \Phi_\alpha)_{_{\alpha \in \mathcal{I}}}$ where $U_\alpha$ are open in $M$,
$\left\{U_i\right\}_{i\in\mathcal I}$ is an open covering  of $M$ 
and $\Phi_i:U_i\to {V_i}\subset\C^{n-p}\times\C^{p}=\C^n$ are homeomorphisms 
such that for overlapping pairs $U_i$, $U_j$ the transition functions 
$\Phi_{ij}=\Phi_j{\Phi_i}^{-1}:\Phi_i(U_i\cap{U_j})\to\Phi_j(U_i\cap{U_j})$
are of the form:
\[
\Phi_{ij}(x,y)=(\Phi^1_{ij}(x),\Phi^2_{ij}(x,y))\quad x\in\C^{n-p}, \quad y\in \C^p \tag{\color{red}{$\frak{H}$}}
\]
where $\Phi^1_{ij}$ and $\Phi^2_{ij}$ are holomorphic and $\Phi^1_{ij}$ is a local biholomorphism between open sets of 
$\C^{n-p}$ and $\Phi^2_{ij}$ is a local holomorphic submersion from an open set in $\C^n$ onto an open set of $\C^p$.
\end{definition}

\begin{definition}\label{definition:charts} The atlas $(U_\alpha, \Phi_\alpha)_{_{\alpha \in \mathcal{I}}}$ is called a \emph{holomorphic foliation atlas} and the maps $\Phi_\alpha$
are called \emph{holomorphic flow boxes} or \emph{holomorphic foliation charts}.
The sets of the form $\Phi_\alpha^{-1}(\left\{x\right\}\times{\C^p}),\,x\in\C^{n-p},$ \ie the set of points whose coordinates 
$(X,Y)$ with $X=(x_1,\cdots,x_{n-p})\in\C^{n-p}$ , $Y=(y_1,\cdots,y_p)\in\C^p$ satisfy $X=C$ for some constant vector $C\in\C^{n-p}$ are called {\em plaques}. Condition (\textcolor{red}{$\frak H$})
says that the plaques glue together to form complex submanifolds  
called {\em leaves}, which are immersed in $M$ (not necessarily properly immersed). If  $(U_\alpha, \varphi_\alpha)_{_{\alpha \in \mathcal{I}}}$ is a complex atlas as in definition \ref{hfoliation} the leaves are immersed $p$-dimensional holomorphic submanifolds of $W$.
\end{definition}

The family of biholomorphisms $\{\Phi^1_{ij}\}_{i\in{I}}$ defines a groupoid called the {\it transverse holonomy groupoid}.
It can be used to define noncommutative toric varieties \cite{BaZ, KLMV}.
\bigskip

Let $m$ and $n$ be two positive natural numbers such that $n>2m$.
Let $\Lam:=(\Lambda_1,\cdots,\Lambda_n)$ be an $n$-tuple of vectors in $\C^m$ where
$\Lambda_i=(\lambda_i^1,\cdots,\lambda_i^m)$ for $i=1,\cdots, n$.

\bigskip
To the configuration $(\Lambda_1,\cdots,\Lambda_n)$ we can associate the linear (singular) foliation
of $\C^n$ generated by the $m$ holomorphic linear commuting vector fields ($1\leq{j}\leq{m}$)

\[
 \C^n\ni(z_1,\cdots, z_n) \longmapsto \overset{n}{\underset{i=1}\sum}\,\lambda_i^j {z_i} \frac{\partial}{{\partial}z_i}
\]

\[
\frac{d\mathbf{Z}}{dT}=\begin{bmatrix}
    \lambda_1^j       & 0 &0 & \dots & 0 \\
  0       & \lambda_2^j  & 0 & \dots & 0 \\
    \hdotsfor{5} \\
   0      &0 &  & \dots &\lambda_n^j 
\end{bmatrix}\mathbf{Z},  \quad \quad \ie \quad\quad \frac{d\mathbf{Z}}{dT}=\Lambda_i\mathbf{Z}, 
\tag{\textcolor{blue}{System of linear equations}}
\]  

\bigskip
$$
\mathbf Z=\begin{bmatrix} z_1\\ \vdots	\\	 z_n\end{bmatrix},\quad j=1,\cdots,m,\quad T\in\C
$$

Let us start with the construction of an infinite family of compact complex manifolds.
Let $m$ be a positive integer and $n$  and integer such that $n>2m$.

\begin{definition}\label{admissible}  Let $\Lam=(\Lambda_1,\hdots,\Lambda_{n})$ 
be a configuration of $n$ vectors in ${\mathbb C}^m$. Let 
${\mathcal H}(\Lambda_1,\cdots,\Lambda_n)$
be the convex hull of $(\Lambda_1,\cdots,\Lambda_n)$.

We say that $\Lam$ is {\em admissible} if:
\begin{enumerate}
\item  ({\bf SC}) The Siegel  condition: $0$ belongs to the convex hull $\mathcal H({\Lam}):=\mathcal H(\Lambda_1,\hdots,\Lambda_{n})$ 
of $(\Lambda_1,\hdots,\Lambda_{n})$ in $\mathbb C^m\simeq \mathbb R^{2m}$.

\item  ({\bf WH}) The weak hyperbolicity condition: for every $2m$-tuple of integers $i_1,\cdots, i_{2m}$
such that $1\leq i_1< \cdots < i_{2m}\leq n$ 
we have 
$0\notin\mathcal{H}(\Lambda_{i_1}, \cdots,\Lambda_{i_{2m}})$ 
\end{enumerate}
\end{definition}

This definition can be reformulated geometrically in the following way: the
convex polytope ${\mathcal H}(\Lambda_1,\cdots,\Lambda_n)$ contains $0$, but neither external nor internal
facet of this polytope (that is to say hyperplane passing through $2m$ vertices)
contains $0$. An admissible configuration satisfies the following regularity property

\begin{lemma}
Let $\Lambda_i'=(\Lambda_i,1)\in{\C^{n+1}}$, for $i\in\{1,\cdots,n\}$.
For all set of
integers $J\subset\{1,\cdots,n\}$ such that $0\in{\mathcal H}((\Lambda_j)_{j\in{J}})$ the complex rank of
the matrix whose columns are the vectors $(\Lambda_j)_{j\in{J}}$ is equal to $m+1$, therefore it is of
maximal rank.
\end{lemma}

\begin{figure}

\includegraphics[scale=0.25]{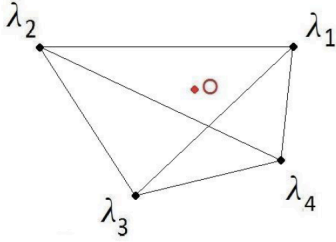}
\caption{Quadrilateral in $\C$}
\end{figure}

One considers the holomorphic (singular) foliation $\mathcal F$ in projective space $\CP^{n-1}$ given by the orbits of the linear action of $\C^m$ on $\C^n$ induced by the linear vector fields (1).

\begin{equation}
\label{actionLVM}
(T,[z])\in\mathbb C^m\times\CP^{n-1}\longmapsto [z_1\cdot\exp\langle \Lambda_1,T\rangle,\hdots, z_n\cdot\exp\langle \Lambda_n,T\rangle]\in\CP^{n-1}\tag{\color{blue}{1}}
\end{equation}

\noi where  $T=(t_1,\cdots,t_m)\in\C^m$, $[z_1,\cdots,z_n]$ are projective coordinates and
 $\langle -,-\rangle$ is inner product $\langle Z, W\rangle=\overset{n}{\underset{i=1}\sum}z_iw_i$  
 
\noi  One can lift this foliation to a foliation $\tilde{\mathcal F}$ in $\mathbb C^n$
given by the linear action
\[
(T,z)\in\mathbb C^m\times\mathbb C^{n}\longmapsto (z_1\cdot\exp\langle \Lambda_1,T\rangle,\hdots, z_n\cdot\exp\langle \Lambda_n,T\rangle)\in\mathbb C^{n}. \label{actionS}\tag{\color{blue}{$\frak{S}$}}
\]
The so-defined foliation is singular, in particular 0 is a singular point.
 The behavior in the neighborhood of 0 determines two different sorts of leaves.
\bigskip

\begin{definition} {\bf (Poincar\'e and Siegel leaves)} Let $L$ be a leaf of the previous foliation. If 0 belongs to the closure
of $L$, we say that $L$ is a {\em Poincar\'e} leaf. In the opposite case, we talk of a {\em Siegel
leaf.}
\end{definition}
\noi If $L$ is a Siegel leaf then the distance from that leaf to the origin is positive and one can show that there exists
a unique point $\mathbf z=(z_1,\cdots,z_n)\in{L}$ which minimizes the distance to the origin and this point satisfies
\[
\sum_{i=1}^n\Lambda_i\vert z_i\vert ^2=0  \tag{2}
\]
This is because the leaf $L_W$ through the point $W=(w_1,\cdots,w_n)$ in the Siegel domain is the Riemann surface
\[
L_W=\left\{(w_1\cdot\exp\langle \Lambda_1,T\rangle,\hdots, w_n\cdot\exp\langle \Lambda_n,T\rangle)\in\mathbb C^{n}\,\,|\,\,
T\in\C^m\right\}
\]    
 and to minimize the (square of the) distance to the  origin we see that Lagrange multipliers imply that the complex line from the origin to a point that minimizes the square of the distance must be orthogonal to the orbit at the point.

\medskip

\noi One has the following dichotomy: 
\begin{enumerate}[label=\roman*]
\item) \, If $0\notin\mathcal H(\Lambda_1,\hdots,\Lambda_{n})$ then every leaf is of Poincar\'e type.
\medskip
\item) \, The set of Siegel leaves is nonempty if and only if $0\in\mathcal H(\Lambda_1,\hdots,\Lambda_{n})$
\end{enumerate}

\begin{figure}
\includegraphics[scale=0.25]{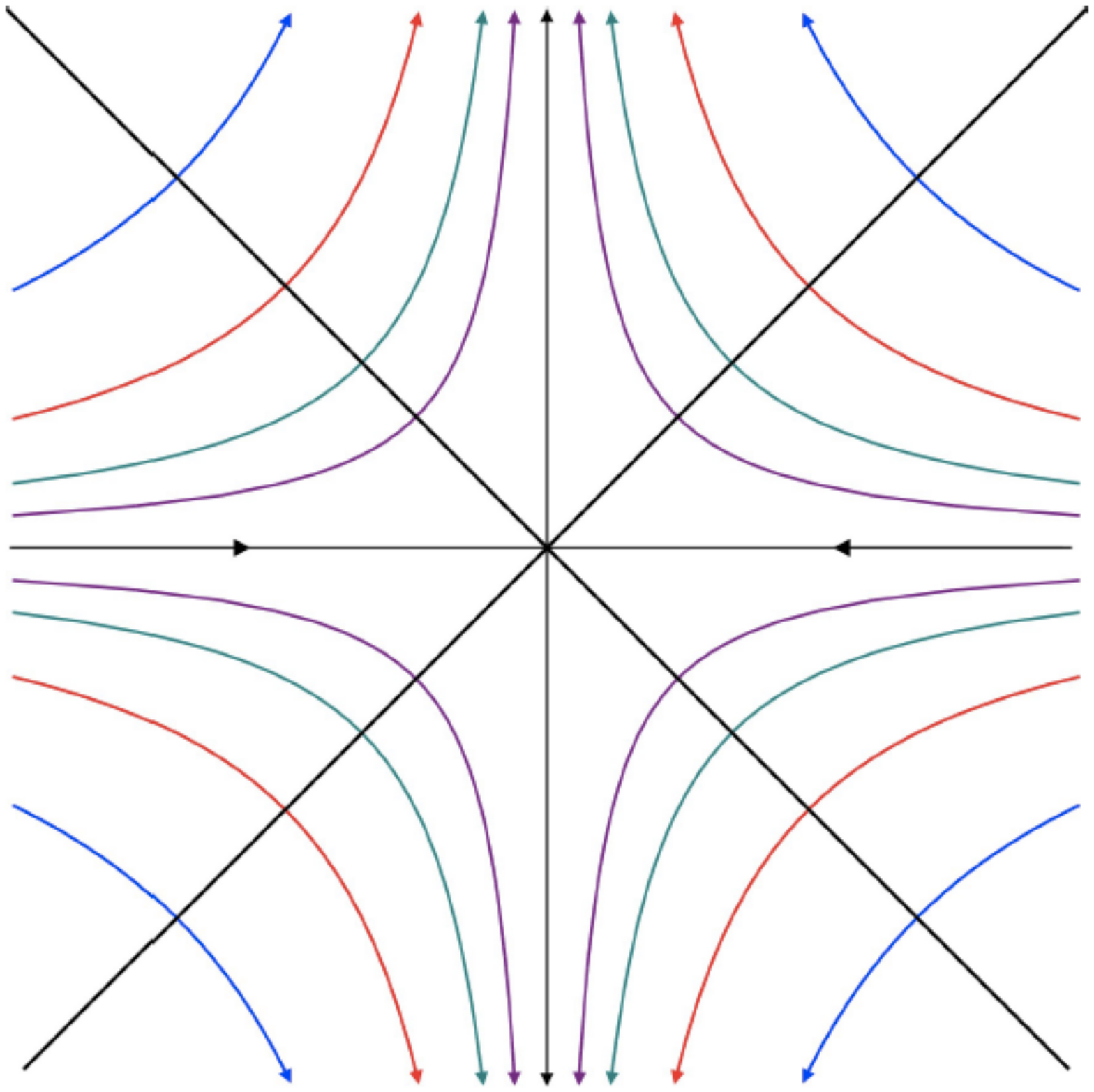}
\includegraphics[scale=0.30]{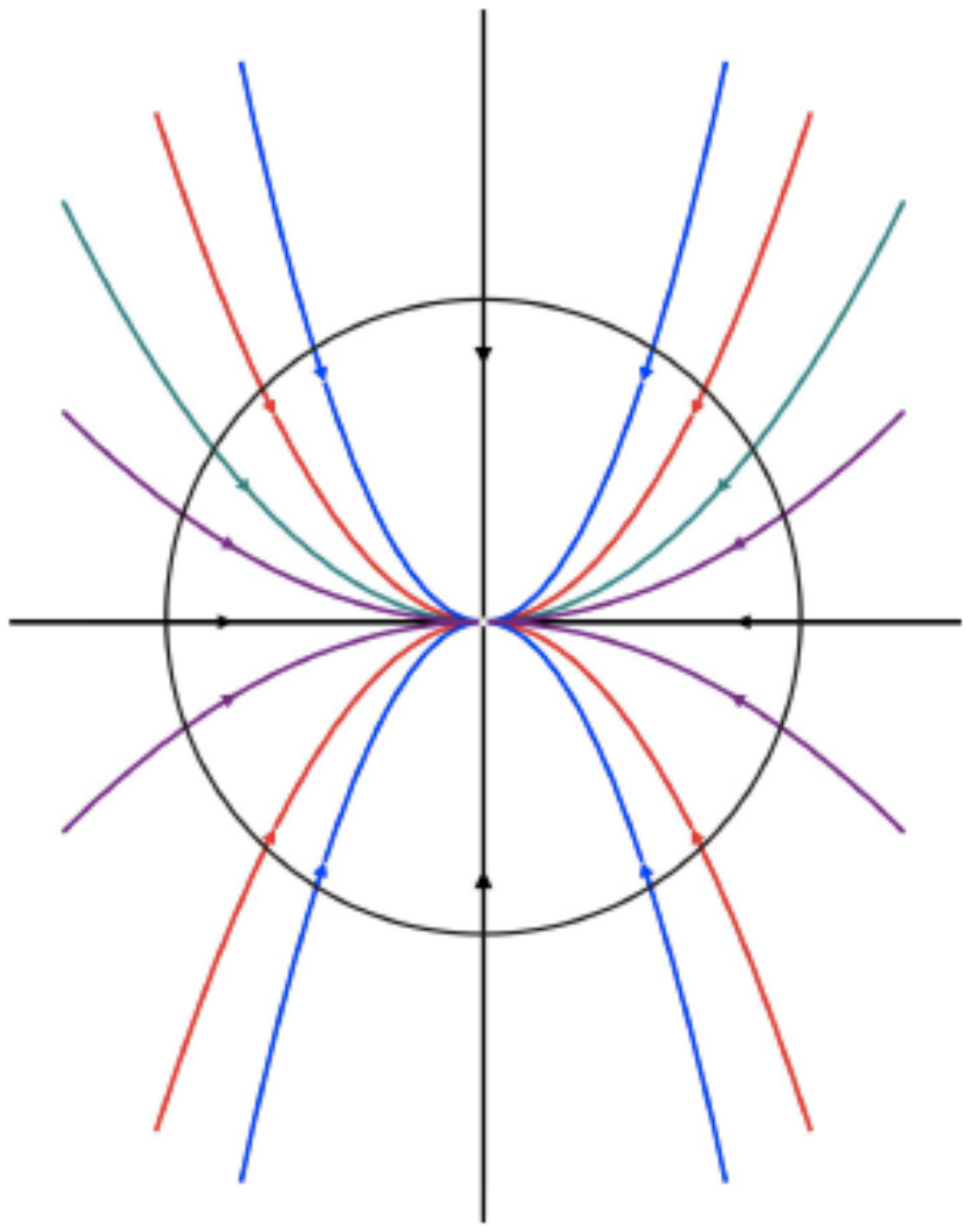}
\caption{Siegel and Poincar\'e leaves}
\end{figure}

\noi For $z=(z_1,\cdots,z_n)\in\C^n$ let $I_z\subset\{1,\cdots,n\}$ defined as follows $I_z=\{j:\,\,z_j\neq0\}$ and
let $\Lambda_{I_{z}}=\{\Lambda_j\,:\,\,j\in{I_z} \}$. One defines:  

\[
\quad\quad\quad\quad\quad\mathcal{S}={\mathcal{S}_{\Lam}}=
\{z\in\mathbb C^n \quad\vert\quad 0\in{\mathcal H}(\Lambda_{I_z}) \}, \hskip.5cm \textcolor{blue}{\mathcal{S}\,\, \rm{as\, the\, complement\,of\,subspaces\, in\, \C^n}}
\label{Scomplement_subspaces}
\tag{\textcolor{blue}{3}}
\]

\begin{definition} We define $\mathcal V={\mathcal V}(\Lam)\subset{\mathbb P}^{n-1}$ to be the image of $\mathcal{S}$ in ${\mathbb P}^{n-1}$ under the canonical projection $\pi:\C^n-\{0\}\to{\mathbb P}^{n-1}$. \label{Siegel-projective}
\end{definition}
\bigskip

\noi Let
\[
\mathcal T={\mathcal T}(\Lam)  =\{z\in\mathbb C^n\quad\vert\quad z\neq0,\,\, \sum_{i=1}^n\Lambda_i\vert z_i\vert ^2=0\,\}
\label{transversal} 
\tag{\textcolor{red}{T}}
\]
then $\mathcal T$ is the set of points that realize the minimum distance in each Siegel leaf. Then $\mathcal T$
meets ever Siegel leaf in exactly one point and it meets each leaf transversally. 
\noi We have that $\overline{\mathcal T}=\mathcal T\cup\{0\}$ is a singular manifold with an isolated singularity at the origin.
 Let

\[
N=N_{\Lam}=\{[z]\in\mathbb P^{n-1}\quad\vert\quad \sum_{i=1}^n\Lambda_i\vert z_i\vert ^2=0\} \label{LVM-definition} 
\quad\quad\quad\quad\quad \quad\quad\quad\quad \textcolor{blue}{\rm{Equations\,\, of\,\, LVM\,\, manifolds}} 
\tag{\large\color{blue}{$\dagger$}}
\]
\noi One can verify that $\mathcal{S}$  is the union of the Siegel leaves and that
 $\mathcal{S}$ is an open set of the form $\mathcal{S}=\C^n-E$  where $E$ is an
analytic set, whose different components correspond to subspaces of $C^n$ where
some coordinates vanish.
 
\noi The leaf space of the foliation restricted to $\mathcal{S}$ , that we call $M$, or $M_{\Lam}$ if we want to emphasize $\Lam$, is
identified with $\mathcal T$. 

\noi Since $\mathcal{S}$ contains $(\C^*)^n$ we see that $\mathcal{S}$ is dense in $\C^n$.

\noi The weak hyperbolicity condition implies that the system of quadrics given by the preceding equations
which define $\mathcal T$ and $N $ have maximal rank in every point

\noi The Siegel condition implies that
$\mathcal T$ and $N $ are nonempty. 
One can show also that  $\tilde{\mathcal F}$ is regular in $\mathcal{S}$ and that $\mathcal T$ 
is a smooth manifold transverse to the restriction
of $\tilde{\mathcal F}$ to  $\mathcal{S}$.
In other words the quotient space of $\tilde{\mathcal F}$ restricted to $\mathcal{S}$ can be identified with $\mathcal T$. 

\noi Therefore by  Haefliger's lemma \ref{Haefliger} below $\mathcal T$ has the structure on a (non-compact) complex manifold
which we call $M$.

\noi Also $N $ can be identified with the quotient space of $\mathcal F$ restricted to $\mathcal V$ 
(definition \ref{Siegel-projective})
and therefore it inherits a complex structure.
Let us denote this complex manifold by $N$.
The complex dimension of $M$ is $n-m$ and of $N$ is $n-m-1$.

\noi The natural projection  $M\to N$, induced by the projection $\pi:\mathbb C^n\setminus\{0\}\to{\mathbb P}_{\C}^{n-1}$, 
is in fact a principal $\mathbb C^*$ fibration. Let $M_1$ denote the total space of the associated circle fibration
It has the same homotopy type as 
$M$ but it has the advantage of being compact. 

\noi Let us observe that  $M_1$ can be identified with the transverse intersection of the cone
$\mathcal T$ (with the vertex at the origin delated) and the unit sphere $\s^{2n-1}$ in $\mathbb C^n$. For this reason we make the following definition 

\begin{definition} \label{moment_angle} Let 
\[
M_1=M_1(\Lam)=\{z=(z_1,\cdots,z_n)\in\mathbb C^n\quad\vert\quad \sum_{i=1}^n\Lambda_i\vert z_i\vert ^2=0,\, \,\,\, \sum_{i=1}^n\vert z_i\vert^2=1\}. \tag{\textcolor{blue}{4}} \label{moment_angle}
\]
Then $M_1(\Lam)$ is called the {\bf moment-angle} manifold corresponding to $\Lam$.
\end{definition}

\begin{remark} 
Let $\Lam$ be an admissible configuration. Then $N_{\Lam}$ 
and $N_{(A\Lam+B)}$
(with $A\in\rm{GL}_m(\C)$ and $B\in\C^m$) 
are biholomorphic (provided that $(A\Lam+B)$ is 
admissible and provided that the corresponding sets $\mathcal S$ are the same).\label{affine_transformation}
\end{remark}

\remark{Remark 1.5} 
The manifold $N$ is naturally equipped with the principal $\C^*$-bundle $\mathcal T\to 
N$.
\endremark
\medskip

\begin{remark}
The natural projection $M_1\to N$ is a $\s^1$-principal bundle. It is
in fact the unit bundle associated to the bundle $\mathcal T\to N$.
\end{remark}
\medskip
 Then, the differentiable
embedding of $N$ into the projective space described  yields an embedding of
fibre bundles
$$
\CD
M_1 @>>> \s^{2n-1} \cr
@VVV @VVV \cr
N @>>> \Bbb CP^{n-1}
\endCD
$$
Let us denote by $\omega$ the pull-back of the Fubini-Study K\"ahlerian form
by this embedding. The form $\omega$ is thus a closed real two-form on $N$ which represents the Euler class of the bundle
$M_1\to N$.

\begin{definition}
We call $\omega$ the {\it canonical Euler form} of the bundle $M_1\to N$.\label{Euler_class}
\end{definition}

\begin{definition}\label{indispensable}
Let $1\leq i\leq n$. We say that $\Lambda_i$ (or more briefly $i$) is {\it indispensable} if $(\Lambda_j)_{j\in \{i\}^c}$ is not admissible. Let 
$I\subset\{1,\hdots,n\}$. We say that $(\Lambda_i)_{i\in I}$ (or more briefly $I$) is {\it removable} if $(\Lambda_j)_{j\in I^c}$ is still admissible.
\end{definition}

\begin{remark}
 Let 
$I\subset\{1,\hdots,n\}$ of cardinal $p$. If $I$ is removable, then the configuration $(\Lambda_i)_{i\in I^c}$ gives rise to a holomorphic LVM submanifold of
$N(\Lambda_1,\hdots,\Lambda_n)$ of codimension $p$.
\end{remark}

\textcolor{blue}{\begin{remark} We write ${\mathcal S}_{\Lam}$, $N_{\Lam}$, $M_1(\Lam)$ etc., if we want to emphasize the configuration $\Lam$.
However many times we omit $\Lam$ if it is clearly understood and no confusion is possible.
\end{remark}}

\noi Another characterization of $\mathcal{S}$ is the following:
\[
\mathcal{S}=\{z\in{\mathbb C}^n\,\,|\,\,0\,\,\text{is not in the closure of the leaf of}\,\, \tilde{\mathcal F}\,\,\text{through}\,\,z \}  \tag{5}
\]
in other words $\mathcal{S}$ is the union of the Siegel Leaves and it open and invariant under the action of ${\mathbb C}^m$

\begin{remark} The space of Siegel leaves $\mathcal{S}$ has the same homotopy type 
as $M$ and therefore also as $M_1$.
\end{remark}

\begin{remark} The linear holomorphic action of $(\C^*)^m$ commutes with the diagonal action (by diagonal matrices) hence
$(\C^*)^n$ acts on $M$.
\end{remark}

\noi The open set $\mathcal{S}$ is a deleted complex cone in $\C^n$:  i.e. if $Z\in\mathcal{S}$ then $\lambda{Z}\in\mathcal{S}$ for all 
$\lambda \in \C^*$. Therefore $\mathcal V=\pi(\mathcal{S})$ (definition \ref{Siegel-projective})
is an open set of ${\mathbb P}_{\C}^{n-1}$.
Then $\pi(\mathcal T)$ is a smooth manifold of dimension equal to the codimension of $\mathcal F$ and transversal to the leaves.
By the following observation by Andr\'e Haefliger it is a complex manifold:
\begin{lemma}\label{Haefliger} {\bf A. Haefliger.} Let $\mathfrak M$ be a complex manifold of complex dimension $n\geq2$ and $\mathcal F$ a holomorphic foliation of $\mathfrak M$ of
codimension $m\geq1$ with $n\geq{m}$. Let
$\mathfrak N\subset{\mathfrak M}$ be a smooth manifold of real dimension $2m$ which is transversal to the leaves of $\mathcal F$. Then $\mathfrak N$ is in a natural way a complex manifold of complex dimension $m$.
\end{lemma}
\begin{proof} In fact if $V\subset\mathfrak M$ is an open subset of $\mathfrak N$ which is contained the domain $U_\alpha$ of the foliation chart 
$\Phi_\alpha:U_\alpha\to\C^{m}\times\C^{m-n}$ of $\mathcal F$ then if $\hat\Phi_\alpha={\Phi_\alpha}_\restriction{V}$ is the restriction of $\Phi_\alpha$ to $V$ and $\pi_1:\C^m\times\C^{n-m}$ is projection onto the first factor then $\Psi_\alpha=\pi_1\circ\hat\Phi_\alpha$ is a holomorphic coordinate chart of $\mathfrak M$. Condition (2) in definition \ref{hfoliation} implies that the coordinate changes are holomorphic. 
\end{proof}

\begin{remark} Bogomolov has conjectured that every compact complex manifold $W$ can be obtained by this process for a singular holomorphic foliation of projective space and $W$ transversal to the foliation outside of the singularities. More precisely he asks:
\emph{can one embed every compact complex manifold as a $C^\infty$ smooth
subvariety that is transverse to an algebraic foliation on a complex projective algebraic variety?}

\medskip
In this respect, Jean-Pierre Demailly, Herv\'e Gaussier \cite{DeG}
have shown an embedding theorem for compact almost complex
manifolds into complex algebraic varieties. They show that every almost complex structure can
be realized by the transverse structure to an algebraic distribution on an affine algebraic variety,
namely an algebraic subbundle of the tangent bundle.
\end{remark}

\begin{definition} \label{LVM} ({\bf LVM manifolds}) If $\Lam$ is an admissible configuration the manifold $N=N_{\Lam}$ given by formula \ref{LVM-definition} 
above is called a LVM manifold corresponding to $\Lam$. It is a compact complex manifold and 
$dim_\C\,N_{\Lam}=n-m-1$

\end{definition}   

\section{Examples}
 
 \subsection{Elliptic curves} \label{elliptic curves} In $\mathbb C$ consider a non-degenerate triangle with vertices  $\lambda_1$, $\lambda_2$ and $\lambda_3$. Suppose that the origin is in the interior of this triangle. 
 Then the open set of Siegel leaves ${\mathcal S}\subset {\mathbb C}^3$ is the complement of 
 the three coordinate hyperplanes $z_1=0$, $z_2=0$ and
 $z_3=0$. 
 
 The set in $\mathcal T\subset{\mathbb C}^3-\{0\}$ given by the equation:
 \[
  \lambda_1|z_1|^2+ \lambda_2|z_2|^2+ \lambda_3|z_3|^2=0     \tag{6}  
 \]
is the transversal as in formula \ref{transversal} above and it meets every leaf in  $\mathcal S$ in exactly one point. So that the space of leaves in  $\mathcal S$ can be identified
with the set given by equation (6).

The set $\mathcal T$ is a complex cone with the origin deleted so that if $Z\in{M}$ also
$cZ\in{\mathcal T}$ for al $c\in{\mathbb C}^*$. 

   We see that $N=:N_{(\lambda_1,\lambda_s,\lambda_3)}$ is the projectivization of $\mathcal T$ and therefore $N$ can be identified  is the set of points
satisfying the following two equations:

 \[
   \left\{
                \begin{array}{ll}
                \lambda_1|z_1|^2+ \lambda_2|z_2|^2+ \lambda_3|z_3|^2=0
\\
                                 \\
                |z_1|^2+|z_2|^2+|z_3|^2=1                   \end{array}
              \right.   
  \]

modulo the natural action of the circle given by 
\[
(z_1,z_2,z_3)\mapsto (\mu{z_1},\mu{z_2},\mu{z_3}),\,\, |\mu|=1,\,\, (z_1,z_2,z_3)\in{N}.   
\]

\noi Hence one has a free action of ${\mathbb C}^*$ and
the quotient $N:=M/{\mathbb C}^*$, then a complex, compact manifold of dimension one.
In fact $N$ is an elliptic curve. In the cases where $M$ is not simply connected (\ie when $k=3$ and
$d=n_1=1$), the complex structure on $N$ can be described in terms
of the defining parameters by identifying it with previous
descriptions of these known manifolds, for instance
when $n=k=3$ the manifold $N$ is
diffeomorphic to the torus $\s^1\times \s^1$. To identify the
corresponding complex structure, observe that in this case ${\mathcal{S}}=(\C^*)^3$. The mapping $exp:\C^3 \to {\mathcal{S}}=(\C^*)^3$ given by
$exp(\zeta_1,\zeta_2,\zeta_3)=(e^{\zeta_1},e^{\zeta_2},e^{\zeta_3})$
can be used to identify $N(\lambda_1,\lambda_2,\lambda_3)$ with the
quotient of $\C$ by the lattice generated by $\lambda_3-\lambda_2$
and $\lambda_1-\lambda_2$. So we have that \vskip10pt \noindent {\it
$N_{(\lambda_1,\lambda_2,\lambda_3)}$ is biholomorphically equivalent
to the elliptic curve with modulus
$\frac{\lambda_3-\lambda_2}{\lambda_1-\lambda_2}$.} \vskip10pt
\noi Observe that in this case we obtain all complex structures on the
torus. By choosing adequately the order of the $\lambda_i$ we obtain
a mapping from the Siegel domain to the Siegel upper half-plane in
$\C$. Therefore  {\it Any elliptic curve is obtained this way.}

\subsection{Compact complex tori}\label{tori} ({\bf i}) If $n=2m+1$, the convex hull $\{\Lambda_i\}_{i\in\{1,\cdots,2n+1\}}$ 
is a simplex in $\mathbb C^m\simeq \mathbb R^{2m}$. 

\noi In fact if one removes one the $\Lambda$'s then  $0$ is not in the complex hull of the remaining.
In other words $\mathcal{S}$ is equal to $(\mathbb C^*)^n$
and one can show that $N$ is a complex torus.

\bigskip
\begin{remark}Every compact complex torus is obtained by this process.
In particular, if $n=3$ and $m=1$ we obtain every elliptic curve.
\end{remark}

\subsection{Hopf manifolds} \label{Hopf-manifolds} ({\bf ii}) If $m=1$ let us define for $n\geq 4$
$$
\Lambda_1=1 \qquad \Lambda_2=i \qquad \Lambda_3=\hdots=\Lambda_n=-1-i\ .
$$
It is easy to verify that under these conditions
 $\mathcal{S}$ is equal to $(\mathbb C^*)^2\times \mathbb C^{n-2}\setminus\{0\}$. 
Consider the two real equations that are used to define
 $\mathcal T$:
 \[
   \left\{
                \begin{array}{ll}
                 |z_1|^2=|z_3|^2+\hdots+ |z_n|^2
\\
                                 \\
                  |z_2 |^2=|z_3|^2+\hdots+ |z_n|^2.\\
                   \end{array}
              \right.   
  \]

\noi If we fix the modules of $z_1$ and $z_2$ (by the definition of $\mathcal{S}$ they cannot be 0) 
the above equations imply that these modules are equal and that
 $(z_3,\hdots ,z_n)$ belong to a sphere $\s^{2n-5}$. 
 Therefore these equations define a manifold which is diffeomorphic to
 $\s^{2n-5}\times\mathbb \s^1\times \s^1\times\mathbb R^+_*$. 
The manifold $M_1$ obtained as the intersection of  $\mathcal T$ and the unit sphere of $\mathbb C^n$ is diffeomorphic to
 $\s^{2n-5}\times\mathbb S^1\times \s^1$
and $N$ is diffeomorphic to $\s^1\times\s^{2n-5}$. In particular for $n=4$, on has all the linear Hopf surfaces.

\subsection{Calabi-Eckmann manifolds} \label{Calabi-Eckmann} {\bf (iii)} Let $m=1$, $\,n=5$ and
$$
\Lambda_1=1 \qquad \Lambda_2=\Lambda_3=i \qquad \Lambda_4=\Lambda_5=-1-i\ .
$$

\noi An argument similar to the previous one
shows that $N$ is diffeomorphic to $\s^3\times \s^3$. 
One obtains an example of Calabi-Eckmann of non Ka\"hler manifolds.

\begin{remark} In general one obtains complex structures in products of odd dimensional spheres $\s^{2r+1}\times\s^{2l+1}$ like in the classical Calabi-Eckmann manifolds. In fact: \emph{Every Calabi-Eckmann manifold is obtained by this process}.
\end{remark}

\subsection{Connected sums}\label{Connected sums1} {\bf (iv)} S. L\'opez de Medrano has shown that for the pentagon in the picture below $M_1$ is diffeomorphic to
 the connected sum of five copies of $\mathbb S^3\times \mathbb S^4$.  The complex manifold $N$ is the quotient of this connected sum under a non-trivial action of $\s^1$.

\begin{figure}[h]  
\begin{center}
\includegraphics[height=5cm]{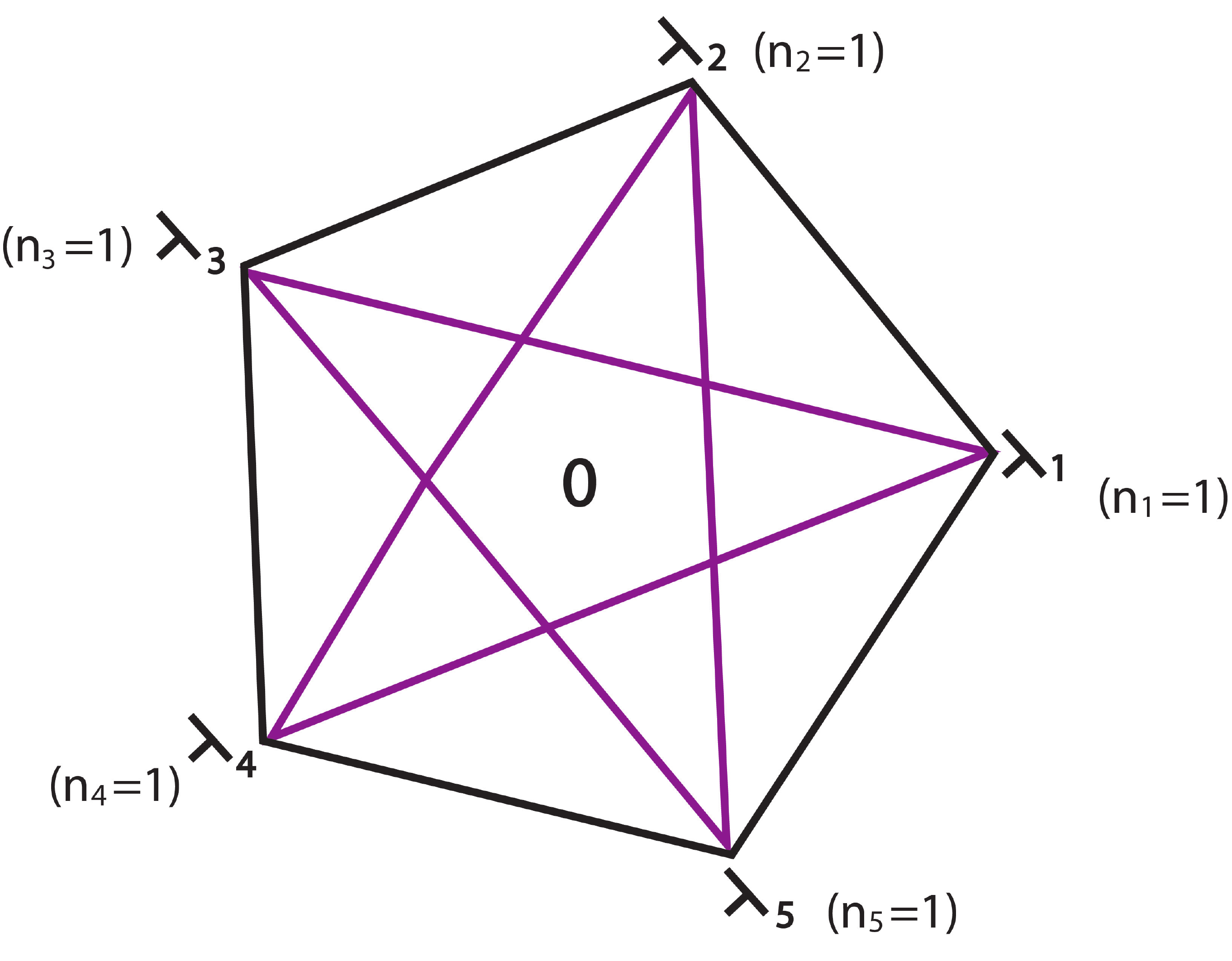}
\end{center}
\caption{\sl Pentagon in $\C$. The number $n_i$ is the multiplicity of $\lambda_i$} 
\label{Figure 1}
\end{figure}

When $m=1$ it can be assumed $\Lam$ is one of the following normal forms: Take $n=n_1+\dots+n_{2\ell+1}$ a partition of $n$ into an odd number of positive integers. Consider the configuration consisting of the vertices of a regular polygon with $(2\ell+1)$ vertices, where the $i$-th vertex in the cyclic order appears with multiplicity $n_{i}$. \\

\noi The topology of $M_1$ and $N$ can be completely described in terms of the numbers $d_i=n_i+\dots+n_{i+\ell-1}$, i.e., the sums of $\ell$ consecutive $n_i$ in the cyclic order of the partition:

\noi  For $\ell=1$: $\quad M_1=\s^{2n_1-1}\times\s^{2n_2-1}\times\s^{2n_3-1}$. For $\ell>1$: $\quad M_1=\#_{j=1}^{2\ell+1}\left(\s^{2d_i-1}\times\s^{2n-2d_i-2}\right)$. 
\noi See theorem \ref{polygon_connected_sum} below.

\begin{center}
\begin{figure}[h]
\centerline{\includegraphics[height=3.5cm]{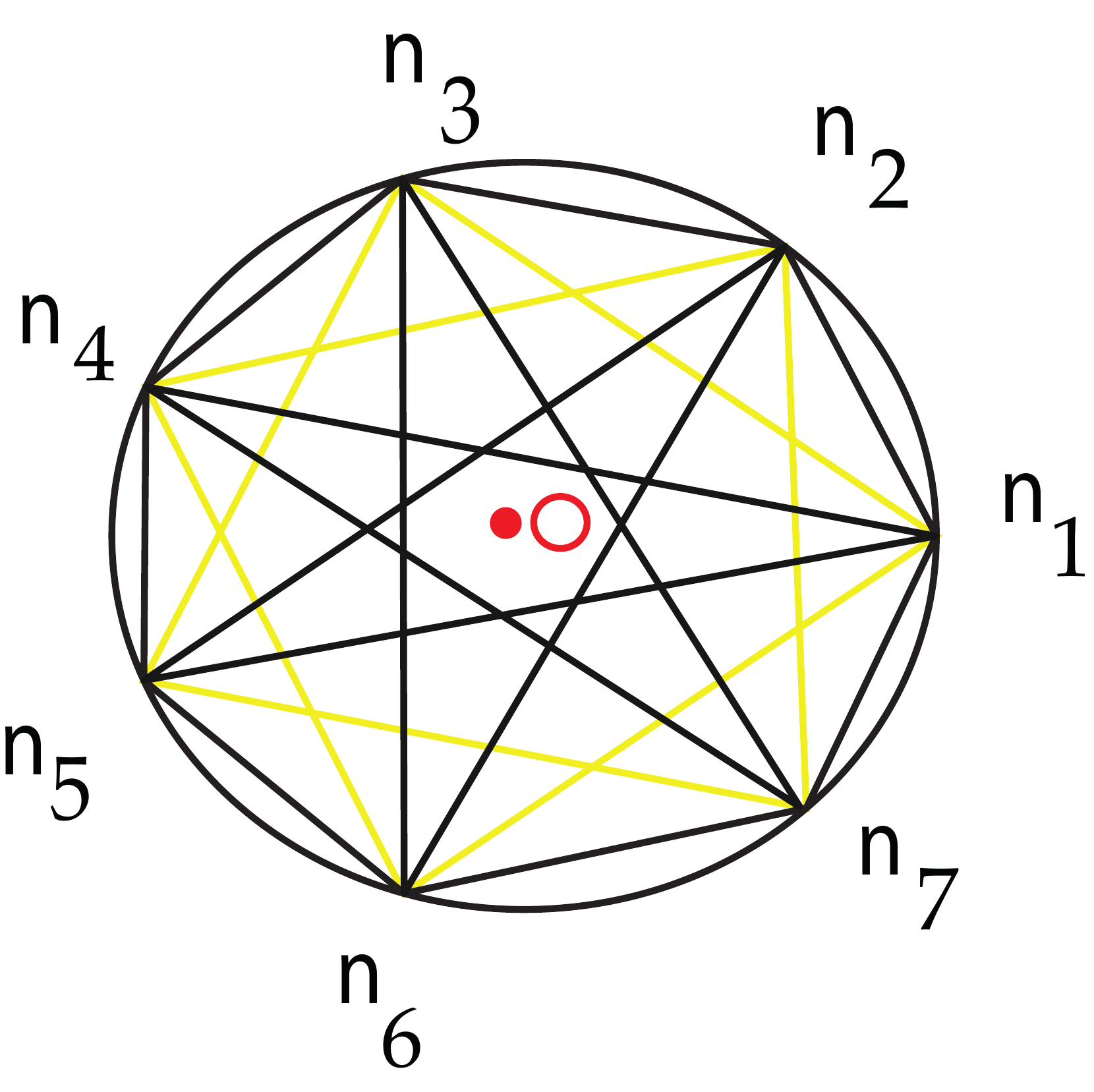}}
\centerline{\includegraphics[height=3.5cm]{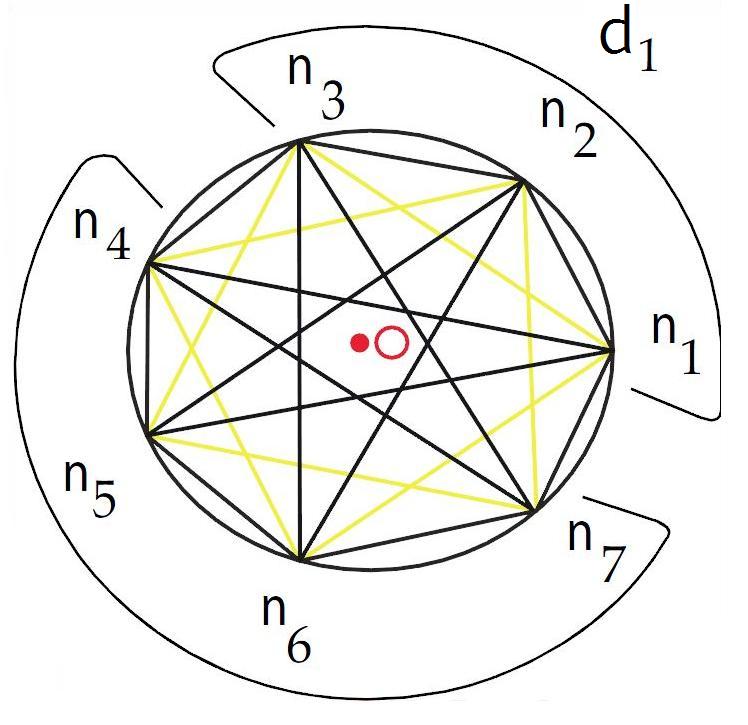}}
\caption{\sl Polygon, the number $n_i$ is the multiplicity of $\lambda_i$} 
\end{figure}
\end{center}  
\vspace{-0.3in}

\medskip

To describe the topology of $N$ we will use the following known
facts about the topology of $M_1$: First observe that the smooth
topological type of $M_1$ (as well as that of $N$) does not change
if we vary continuously the parameters $\Lam$ as long as we do
not violate condition (WH) in definition \ref{admissible} in the process. It is shown in \cite{LdM2}
that the parameters $\Lam$ can always be so deformed until they
occupy the vertices of a regular $k$-gon in the unit circle, where
$k=2l+1$ is an odd integer, every vertex being occupied by one or
more of the $\lambda_i$. 

\noi Therefore the topology of $M_1$ (and that
of $N$ also) is totally described by this final configuration, which
can be specified by the multiplicities of those vertices, that is,
by the partition

$$ n=n_1+\dots+n_k.$$

Observe that different partitions give different open sets ${\mathcal{S}}$ and therefore also different reduced 
deformation spaces. It is
clear that if we permute cyclically the numbers $n_i$ we obtain
again the same manifolds and deformation spaces, but it follows from
the next result that the cyclic order is relevant for their
description.

It is shown in \cite{LdM2} that the topology of $M_1$ is given as
follows: Let $d_i=n_i+n_{i+1}+\dots+n_{i+l-1}$ for $i=1,\dots,k$
(the subscripts being taken modulo $k$). Let also
$$ d=min\{d_1,\dots,d_k\}. $$

\noi These numbers determine the topology of $M_1$:

\newpage

\begin{theorem}\label{polygon_connected_sum}
\noindent (1) If $k=1$ then $M_1=\emptyset.$

\noindent (2) If $k=3$ then $M_1=S^{2n_1-1}\times S^{2n_2-1}\times
S^{2n_3-1}.$

\noindent (3) If $k=2l+1>3$ then $M_1$ is diffeomorphic to the
connected sum of the manifolds $S^{2d_i-1}\times S^{2n-2d_i-2}$,
$i=1,\dots,k$: $\quad M_1=\s^{2n_1-1}\times\s^{2n_2-1}\times\s^{2n_3-1}$. For $\ell>1$: $\quad M_1=\#_{j=1}^{2\ell+1}\left(\s^{2d_i-1}\times\s^{2n-2d_i-2}\right)$.
\end{theorem}
The proof of parts (1) and (2) is quite direct, while the
proof of part (3) is long and complicated \cite{LdM2}. In what follows
we shall only use the fact that the integral homology groups of
$M_1$ coincide with those of the above described connected sum and
the fact that $M_1$ is $(2d-2)$-connected. The homology calculations
(and part (2) of Theorem 1) were first obtained by C. T. C. Wall
(\cite{Wa}). Thus our results will be independent of \cite{LdM2} and will
provide a simplified proof of some of the cases of Theorem 1.
\vskip10pt

\subsection{Some examples of LVM.}

In all the other cases (i.e., when $M_1$ is simply connected) we
obtain new complex structures on manifolds. An intermediate
situation is given by the cases $k=3$, with $n_1=2$, $n_2$ and $n_3$
even, where one can show, using the fact that each $\C^{n_i}$ can be
considered as a quaternionic vector space, that $N$ is diffeomorphic
to $\CP^1 \times \s^{2n_2-1}\times \s^{2n_3-1}$. It is easy to see
that in some cases $N$ can be identified with the product of $\CP^1$ 
with one of the Loeb-Nicolau complex structures on
$\s^{2n_2-1}\times \s^{2n_3-1}$. But in other cases there is no simple
way to stablish such an identification, and it is plausible that
these give new complex structures.

When $k=3$, $n_1>2$ we definitely get a manifold which is not a
product, but a twisted fibration over $\CP^{n_1-1}$. In fact, $N$
clearly fibers over $\CP^{n_1-1}$ with fiber $\s^{2n_2-1}\times
\s^{2n_3-1}$. This fibration does have a section (recall that we are
assuming that $n_1=d$ is not bigger than the other $n_i$) which is
homotopic to the map $\CP^{n_1-1}\to N$ constructed in the Lemma in
section 3. But, by the observation and the end of that section, the
normal bundle of $\CP^{n_1-1}$ in $N$ is stably equivalent to the
normal bundle of $\CP^{n_1-1}$ in $\CP^{n-1}$. By computing the
Pontryagin classes of this bundle one shows that it is not trivial.
 We therefore have:
\vskip10pt \noindent {\bf Theorem 3.} {\it When $3\le n_1\le n_2\le
n_3$ there is a non-trivial $(\s^{2n_2-1}\times
\s^{2n_3-1})$-fibration over $\CP^{n_1-1}$  with an
($n-2$)-dimensional space of complex structures.} \vskip10pt

When $k>3$ we get new complex structures on manifolds. We will give
the complete description of the underlying real smooth manifold only
in the case where all $n_i=1$ (so $n=k=2l+1$), where the
computations and arguments are simpler. To do this we can assume as
before that the $\lambda_i$ are the $n$-th roots of unity:
$\lambda_i=\rho^i$, $\rho$ a primitive root.

In that case $M_1$ is a parallelizable $(2n-3)$-manifold with
homology in the middle dimensions only, where it is free of rank
$n$:
\[H_{n-2}(M_1)=H_{n-1}(M_1)=\Z^n\]

It follows from the Gysin sequence of the fibration $M_1\to N$ (and
from the order of its Euler class found in section 3) that $N$ has
homology only in dimensions $2i$, $i=1,\dots, n-2$ where it is free
of rank $1$, and in dimension $2l-1$ where it is free of rank $2l$.

On the other hand, $M_1$ is the boundary of a manifold $Q$
constructed as follows:

Let
$$
Z=\{z\in\C^n\quad\vert\quad\Sigma \Re(\lambda_i)\,z_i\bar z_i=0,\,\, \Sigma z_i \bar
z_i=1\}.
$$
$Z$ is diffeomorphic to $S^{2l-1}\times S^{2l+1}$ (since the
defining quadratic form has index $2l$) and is the union of two
manifolds with boundary

$$
Q^\pm= \{z\in\C^n\quad\vert\quad\Sigma \Re(\lambda_i)z_i\bar z_i=0,\,\, \pm\Sigma\,
\Im(\lambda_i)z_i\bar z_i\ge 0,\,\,\Sigma z_i \bar z_i=1\}
$$
\noindent whose intersection is $M_1$.

The involution of $\C^n$ which interchanges the coordinates $z_i$
and $z_{n-i}$ preserves $Z$ and $M_1$, and interchanges $Q^+$ with
$Q^-$. Therefore these two are diffeomorphic and $M_1$ is an equator
of $Z$.

Let $Q=Q^+$. It follows now easily from the Mayer-Vietoris sequence
of the triple $(S^{2l-1}\times S^{2l+1},Q,Q^-)$ that $H_i(Q)=0$ for
$i\ne 2l-1, 2l$, in which case it is free of rank $l+1$ and $l$,
respectively, and that $H_i(M_1)\to H_i(Q) $ is always surjective.
$Q$ is also simply connected by Van Kampen's Theorem. The Hurewicz
and Whitehead Theorems now show that all homology classes in $Q$ can
be represented by spheres which for dimensional reasons can be
assumed to be embedded in $M$ by Whitney's imbedding theorem. (This
is enough to show, using the h-cobordism theorem, that $M_1$ is a
connected sum, as described in Theorem 1. in \cite{LdM2} and the
argument used below. It is shown in \cite{LdM2} that these facts are true
in general, by a detailed description of all homology classes in
$M_1$).

The $\s^1$ scalar action leaves $Q$ invariant, so the quotient
$R=Q/\s^1$ is a compact manifold with boundary $\partial R=N$. Now
the fibration $Q\to R$ again embeds in a diagram like 
\ref{diagram-embedding} in lemma
 It follows now from the cohomology Gysin
sequence of the fibration $Q\to R$ that $H_{2i}(R)=\Z$, $i=0,\dots,
l-1$ and $H_{2l-1}(R)=\Z ^l$, all other homology groups being
trivial.

Now we can embed, by lemma \ref{diagram-embedding}, $\CP^{l-1}$ in $R$
representing all even dimensional homology classes, and $l$ disjoint
$(2l-1)$-spheres with trivial normal bundle representing the
generators of the corresponding homology group of $R$ (since all
these classes come from $Q$ and are therefore spherical, and their
normal bundles are again stably equivalent to the trivial normal
bundle of $\s^{2l-1}$ in $\CP^{n-1}$). Taking a tubular
neighborhood of these manifolds and joining them by tubes we get a
manifold with boundary $R^\prime$ whose inclusion in $R$ induces
isomorphisms in homology groups. It follows from the h-cobordism
theorem (\cite{Mi}) that $N=\partial R$ is diffeomorphic to $\partial
R^\prime$ which is a connected sum of simple manifolds. These are
$l$ copies of  $\s^{2l-1}\times \s^{2l-1}$ and the boundary of the
tubular neighborhood of $\CP^{l-1}$ in $R$. By the remark at the
end of lemma \label{diagram-embedding} we know that the normal bundle of this inclusion is
stably equivalent to the normal bundle of $\CP^{l-1}$ in 
$\CP^{2l}$. We have therefore proved the following

\vskip10pt \noindent {\bf Theorem 4.} {\it For every $l>1$ there is
a $(2l-1)$-dimensional space of complex structures on the connected
sum of $\CP^{l-1}\bar\times \s^{2l}$ and $l$ copies of
$\s^{2l-1}\times \s^{2l-1}$, where $\CP^{l-1}\bar\times \s^{2l}$
denotes the total space of the $S^{2l}$-bundle over $\CP^{l-1}$
stably equivalent to the spherical normal bundle of $\CP^{l-1}$ in
$\CP^{2l}$.} 

\newpage

\noi Observe that for $l=2$ we get a manifold
which is close, but not equal, to the one constructed by Kato and Yamada \cite{Kato},
 where the first summand is a product.
Both manifolds had been considered before, from the point of view of
group actions, by Goldstein and Lininger (see \cite{GL}).

In general, these complex structures are very symmetric, in the
sense that we can still find holomorphic actions of large groups on
them (see \cite{LdM1}). In particular, there is an action of the complex,
noncompact, ($n-2$)-torus $(\C^*)^{n-2}$ on them with a dense orbit.
In this sense, our manifolds behave as toric varieties. \vskip10pt

\section{For $m=1$ and $n>3$ the manifolds $N$ are not symplectic.}\label{nonsymplectic}
\begin{theorem} For $n>3$, the manifold
$N=N_{\Lam}$ is a compact, complex manifold that does not admit a
symplectic structure.
\end{theorem}

\begin{proof} In fact it follows from the classification given by theorem \ref{polygon_connected_sum} that the manifold depends on the polygon of $k$ vertices  and for $k=1$ the manifold $M_1$ is empty and $M_1$ is a nontrivial circle bundle over $N$. In general we
have that $M_1$ lies in the sphere $\s^{2n-1}$ and that $N$ sits
inside the complex projective space $\CP^{n-1}$  (but not as a holomorphic submanifold), so we have an
inclusion of $\s^1$-bundles:

\begin{center}
\[
\begin{tikzcd}
M_1\arrow[hookrightarrow]{r}{}\arrow{d}[swap]{\pi_1} & \s^{2n-1}\arrow{d}{\pi_2} \\
N\arrow[hookrightarrow]{r}[swap]{} & \CP^{n-1} \\
\end{tikzcd}
\]
\end{center}

\noi where $\pi_1$ and $\pi_2$ are the restrictions of the canonical map $\pi:\C^n\setminus\{0\}\to{\mathbb P}^{n-1}$
to $M_1$ and $\s^{2n-1}$ respectively. 

\medskip
We will prove first that the inclusion of $N$ can be deformed down in $\CP^{n-1}$
into a projective subspace of low dimension $d - 1$, but not lower. We will prove first the following:

\begin{lemma}\label{diagram-embedding} The above inclusion of
$S^1$-bundles embeds homotopically in the following sequence of
bundle maps:

\[
\begin{array}{ccccccc} \s^{2d-1}&\to& M_1 &\to& \s^{2d-1}&\to&\s^{2n-1}\\
              \Big\downarrow &\ & \Big\downarrow &\ &\Big\downarrow &\ & \Big\downarrow \\
              \CP^{d-1}&\to& N& \to &\CP^{d-1} &\to& \CP^{n-1}
  \end{array} \tag{\textcolor{blue}{Diagram}}
\]

\noindent where the composition of the bottom arrows is homotopic to
the natural inclusion.
\end{lemma}

\newpage

\noindent {\bf Proof of lemma 3.} 
If we put $d$ coordinates $z_i=0$ we obtain a new manifold
$M_1(\Lam^\prime)$ where $\Lam^\prime$ is a configuration of
eigenvalues that is concentrated in $l+1$ consecutive vertices of
the regular $(2l+1)$-gon. This configuration being in the Poincar\'e
domain, it follows that the above manifold is empty.

\noi This means that the original $M_1(\Lam)$ does not intersect a
linear subspace of $C^n$ of codimension $d$ and that correspondingly
$N$ does not intersect an $d$-codimensional projective subspace of
$\CP^{n-1}$. Then the inclusion of $N$ in $\CP^{n-1}$ can be
deformed into a complementary projective subspace of dimension
$d-1$, which gives the middle bundle map.

\noi Now, $M_1$ being $(2d-2)$-connected (by theorem \ref{polygon_connected_sum}), it follows that $M_1\to N$ is a
universal $S^1$-bundle for spaces of dimension less than $2d-1$ (see
\cite{S} page 19) and therefore the Hopf bundle over $\CP^{d-1}$ admits
a classifying map into it, which gives the first map in the bottom
row. The composition of the bottom maps also classifies this Hopf
bundle and is therefore homotopic to the natural inclusion, so the
Lemma is proved.
\medskip
\noi From the description of $M_1$ it follows that $M_1$ is simply connected, except
for the cases $k=3$, $d=n_1=1$. In these cases the $S^1$-action on
$M_1=S^1\times S^{2n_2-1}\times S^{2n_3-1}$ can be concentrated on
the first factor, and therefore $N$ is diffeomorphic to
$S^{2n_2-1}\times S^{2n_3-1}$. Unless $n_2=n_3=1$ we have that
$H^2(N)=0$ and $N$ is not symplectic.

\noi In all the other cases we have that $d>1$ and $M_1$ is
$2$-connected. From the cohomology Gysin sequence of the fibration
$M_1\to N$ it follows that $H^2(N)=\Z $ generated by the Euler class
$e$. However, it follows from the Lemma that
$$ e^{d-1}\ne 0 $$
$$ e^{d}= 0 $$

\noi so this class does not go up to the top cohomology group
$H^{2n-4}(N,\Z)$, and it follows again that $N$ is not symplectic, and
Theorem 2 is proved. \end{proof}

\noi Nevertheless, observe that $N$ is a real algebraic submanifold of
$\CP^{n-1}$ since it is the regular zero set of the (non
holomorphic) function $g:\CP^{n-1} \to \R ^2$ defined by

$$
g([z_1,\dots, z_n]) = \frac{\Sigma \lambda_i z_i\bar z_i}{\Sigma z_i
\bar z_i}
$$

\noi This implies that the normal bundle of $N$ in $\CP^{n-1}$ is
trivial. Observe also that the map $\CP^{d-1}\to N$ in the Lemma
is homotopic to an embedding, whose normal bundle is then stably
equivalent to the normal bundle of $\CP^{d-1}$ in $\CP^{n-1}$. 
\vskip10pt
\subsection{Compact complex tori are the only K\"ahler LVM manifolds}

\vskip10pt

\noi Let $k$ denote the number of indispensable points (remember definition \ref{indispensable} above). By
Carath\'eodory's theorem  $k\leq2m+1$ the maximum is attained only when $n=2m+1$. One has:

\begin{lemma}
\begin{enumerate}
\item  $\mathcal{S}=(\mathbb C^*)^{k}\times (\mathbb C^{n-k}\setminus A)$ with $A$ 
an analytic set of codimension at least two in every point. 
\item This decomposition descends to a decomposition 
 $M_1=(\mathbb S^1)^k\times M_0$ where $M_0$ is a real compact manifold which is $2$-connected. 
\end{enumerate}\label{indispensable-lemma}
\end{lemma}

\noi {\bf Sketch of the proof.}
Let $\mathcal{S}=\mathbb C^n \setminus E$, where $E$ is a union of subspaces (see \ref{Scomplement_subspaces})
$$
E=\{z\in\mathbb C^n \quad\vert\quad 0\notin{\mathcal H}(\Lambda_{I_z}) \}.
$$
The components of codimension one are given by indices corresponding to indispensable points in the configuration.
This proves the first part.
Since $A$ is of complex codimension at least 2 in every point
 $(\mathbb C^{n-k}\setminus A)$ is 2-connected, hence $M_0$ is 2-connected, since they have the same homotopy type.
\qed

\noi In examples (ii) and (iii) one obtains compact complex manifolds which are not symplectic because
the second de Rham cohomology group is trivial. This is in fact a general property of the manifolds we obtain:

\begin{theorem} Let $\Lam$ be an admissible configuration  as in definition \ref{admissible} and $N_{\Lam}$ the corresponding compact complex manifold.
The the following are equivalent:
\medskip
\begin{enumerate}

\item ${\mathcal H}(\Lam)$ is a simplex
\item $N_{\Lam}$ is symplectic.

\item $N_{\Lam}$ is K\"ahler.

\item $N_{\Lam}$ is a complex torus.

\item  $n=2m+1$.
\end{enumerate}
\end{theorem}

\noi {\bf Sketch of the proof.}
It is easy to prove the equivalence of
(3) and (4): If $N$ is a complex torus, one must have $\mathcal S=(\mathbb C^*)^n$ hence all the
 $\Lambda_i$ must be indispensable and in this case the convex hull must be a simplex and 
 $n=2m+1$. If  the convex hull is a simplex then as in example (i) $N$ is a compact complex torus.
 
 \medskip

\noi The most difficult part is that (1) implies (4). One proves that by contradiction. Suppose 
 $n>2m+1$.  As in the examples one must study the de Rham cohomology of $N$ and to prove that it is incompatible with the existence of a symplectic form.
 
\noi We consider two cases:

\medskip
 \noi$\underline{\text{1st case}}$ There exists indispensable points. From here one can deduce that the fibration
 $M\to N$  is trivial. Hence the decomposition
  $M_1=(\mathbb S^1)^k\times M_0$ of the previous lemma gives a decomposition 
   $N=(\mathbb S^1)^{k-1}\times M_0$.  
   
   \noi In other words if $N$ has a symplectic structure it must be supported
   by $(\mathbb S^1)^{k-1}$. The maximal power of this symplectic form must be a volume form in $N$ 
   but that is only possible only if $k-1$ is equal to the real dimension of $N$, i.e. to $2n-2m-2$.
   since $k\leq2m+1$ and $n>2m+1$
    
\medskip

 \noindent $\underline{\text{second case:}}$ If there are not indispensable points then $M$ is 2-connected 
and the fibration  $M\to N$ is not topologically trivial. Therefore the second de Rham cohomology group of $N$
is generated by the Euler class of that fibration. Analyzing carefully this fibration one shows that the Euler class is trivial. Therefore this class is not symplectic, the proof is similar to that of theorem \ref{nonsymplectic} \qed

\section{Meromorphic functions on the manifold $N_{\Lam}$}

Many analytic properties of LVM manifolds are related to the arithmetic properties of the configuration $\Lam$. One nice example of this fact is given by the following

\begin{theorem} \label{algebraic-dimension} {\bf (\cite{Me} Theorem 4)}
Let $N$ be a LVM manifold without indispensable point
Then the algebraic dimension of $N_{\Lam}$ is equal to the dimension over  $\mathbb Q$ of the  $\mathbb Q$-vector space of rational solutions of the system \ref{System(S)}:

 \[
   \left\{
                \begin{array}{ll}
                 \overset{n}{\underset{i=1}\sum}s_i\Lambda_i=0\\
                                 \\
                  \overset{n}{\underset{i=1}\sum} s_i=0\\
                   \end{array}
              \right.   \tag{\textcolor{red}{S}}
  \]
\end{theorem}

The idea of the proof is very simple. If $f$ is a meromorphic function on $N$, then it can be lifted 
to a meromorphic function $\tilde f$ of $\mathcal{S}$ which is constant along the leaves of
$\tilde{\mathcal F}$. Since we have assumed that there is not an indispensable point lemma \ref{indispensable} implies that $\mathcal{S}$ is obtained from $\mathbb C^n$ by removing an analytic subspace
of codimension at least two at every point. Therefore $\tilde f$ can be extended to all $\mathbb C^n$ by Levi's extension theorem (see \cite{BHPV}  p.26).  Furthermore  $\tilde f$ must be invariant by the action given in (2). In particular  $\tilde f$
must be invariant by the standard action of  $\mathbb C^*$ on $\mathbb C^n\setminus\{0\}$, and descends to $\CP^{n-1}$. Therefore $\tilde f$ is a rational function. 
We can show that the fact that  $\tilde f$ is constant along the leaves of $\mathcal F$ implies that an algebraic basis
of these rational functions is given by the monomials

\[z_1^{s_1}\cdot\hdots\cdot z_n ^{s_n},\] 
where $(s_1,\hdots,s_n)$  is a rational basis of the vector space of solutions of   
system (S).

\begin{example}
Let $n=5$ et $m=1$, and:          

$$\Lambda _1=1 \qquad \Lambda _2=i \qquad \Lambda _3=-1-i \qquad \Lambda _4=\dfrac{3}{2}i+1 \qquad
\Lambda _5=-i-\dfrac{1}{2}$$

One verifies immediately that there are not indispensable points. 
The complex dimension of $N$ is 3 and its algebraic dimension is according to the preceding theorem 2. Indeed

\[
f(z)=\dfrac{z_1^5z_2^5z_3^2}{z_4^6z_5^6},\quad\quad  g(z)=\dfrac{z_1z_2^2}{z_3z_4^2}
\]

\noindent 
are meromorphic functions which are algebraically independent on $N$ and in addition every meromorphic
function on $N$ depends algebraically on $f$ and $g$.
\end{example}

\medskip
Recall that a connected Moishezon manifold $M$ is a compact complex manifold such that the 
field of meromorphic functions  has transcendence degree equal the complex dimension of the manifold.
   
When theorem \ref{algebraic-dimension} applies the algebraic dimension of $N$ is at most $n-2m-1$ 
therefore the dimension is strictly inferior to its dimension $n-m-1$. In other words: if there are not
indispensable points $N$ is not  Moishezon. This happens if and only if $\Lam$ is a simplex.
Hence we have the following:

\begin{theorem}\label{Moishezon} {\bf(\cite{Me} Theorem 3)}
The following are equivalent:

\medskip
\noindent (i) $N$ is Moishezon.

\noindent (ii) $N$ is projective.

\noindent (iii) $N$ is a complex projective torus.
\end{theorem}

\noi {\bf \noi{\bf Sketch of the proof.} } We follow the proof given by Fr\'ed\'eric Bosio in \cite{Bosio} p.1276-1277. If  $I$ is a subset of $\{1,\hdots,n\}$ such that $0$ is in the convex envelope of $(\Lambda_i)_{i\in I}$, then
the restriction of action (2) to the complex vector subspace of $\mathbb C^n$ given by the equations 
$$
z_j=0\qquad\text{pour } j\not\in I
$$
defines also a LVM manifold that we denote $N_I$. Then this is a complex submanifold of $N$. One can verify that
if $n>2m+1$, \ie there are points that can be eliminated, one can find always submanifolds which have indispensable and in facr we can fine such a submanifold with odd first Betti number. But if $N$ is Moishezon then all of its complex subvarieties are also Moishezon and therefore must have first Betti number even. \qed
 
\begin{remark}
Exactly this last argument implies that $N$ is not K\"ahler if $n>2m+1$.
\end{remark}

\section{Deformation theory}\label{deformation}

\subsection{Small deformations} \label{small-deformations} We will state without a proof a theorem of stability of LVM manifolds under small deformations. Let $\Lam$ be an admissible configuration and $N$ the associated LVM manifold. For $\epsilon>0$, let $(\Lam_t)_{-\epsilon <t<\epsilon}$ be a small smooth perturbation of $\Lam$ (\ie a smooth function from $(-\epsilon,\epsilon)$ to
$(\C^m)^n$ such that  $\Lam_0=\Lam=(\Lambda_1,\hdots,\Lambda_n)$).

\medskip
Since the Siegel and weak hyperbolicity conditions are open in $(\C^m)^n$, if $\epsilon$ is
sufficiently small all the configurations $(\Lam^t)$ are admissibles. The manifold
$(\epsilon,\epsilon)$, $\underset{t\in(\epsilon,\epsilon)}\bigcup{N_t}\subset \CP^{n-1}\times \mathbb R$
admits an obvious submersion over $(-\epsilon,\epsilon)$ with compact fibers. 
Ehresmann' lemma implies that all the $N _t$ are diffeomorphic, however they are not necessarily biholomorphic
it is enough, for instance, to start with a configuration $\Lam$ which verifies la condition ({\bf K}) in definition \ref{System(K)}, and to 
perturb it  in $(\mathbb C^m)^n$)in order to obtain $(\Lam)^t$ which verify (H) in \ref{System(K)}. This way one obtains a non-trivial
family of de LVM manifold $N_{\Lam}$ parametrized by the interval $(-\epsilon,\epsilon)$.

\medskip

On the other hand, if $\Lam$ et $\Lam'$ are two admissible configurations such that  $\Lam'$ is obtained from $\Lam'$ 
by a complex affine transformation of $\mathbb C^m$, \ie there exists a complex affine
transformation $A$ of $\mathbb C^m$ such that $\Lambda '_i=A(\Lambda_i)$ for all $i$, one sees immediately since  
$A(\mathcal{S}_{\Lam})={\mathcal{S}}'$ and $A$ sends a Siegel leaf of the system corresponding to $\Lam$ to a leaf
corresponding to $\Lam'$

\begin{definition}
Let $\Lam$ be an admissible configuration and $N_{\Lam}$ the corresponding LVM manifold.
One calls \emph{space of parameters} of $N_{\Lam}$ the set of equivalence classes on an open connected neighborhood of
$\Lam$ in $(\mathbb C^m)^n$ consisting of equivalence classes of admissible configurations under the equivalence $\cong$
given by $\Lam\cong\Lam'$ if and only if there exists a complex affine transformation $A$ such that $A(\Lam)=A(\Lam')$. 
\end{definition}

The weak hyperbolicity condition implies that  $\Lam$ affinely generates  the space 
 $\mathbb C^m$  (\cite{MV}, Lemma 1.1). Up to renumbering the vectors on can assume that 
 $(\Lambda_1,\hdots,\Lambda_{m+1})$ are affinely independent.  Given a sufficiently small open connected set 
 of configurations in $(\mathbb C^m)^n$ containing  $\Lam$, one sees that every element in that open set can be transformed
 in a unique way to a configuration where the first
 $m+1$ vectors coincide with those of$\Lam$.  Therefore:

\begin{lemma} 
Let $D$ be an space of parameters for $N_{\Lam}$. Then $D$ can be identified with an open connected subset of
$(\mathbb C^m)^{n-m-1}$.
\end{lemma}
 
 Under these conditions on can construct a holomorphic family $\mathcal D$ of deformations of
 $N_{\Lam}$ parametrized by $D$. It is enough to consider the quotient of
 $\mathcal{S}\times D$ under the action in formula \ref{actionLVM}
 with parameters in $D$.

\begin{theorem} {\bf (\cite{Me} Theorem 11)}
Let $D$ be an space of parameters of the LVM manifold $N_{\Lam}$ 
corresponding to the configuration $\Lam$. Let $\mathcal D\to D$ be the associated family of deformation. Then

\medskip

\noindent (i) If $\mathcal{S}$ is at least 3-connected, the family $\mathcal D$ 
is a versal family of deformations of $N_{\Lam}$.

\noindent (ii) If  $\mathcal{S}$ is at least 4-connected and $\Lambda_i\not =\Lambda_j$ if $i\not =j$, the family $\mathcal D$ is universal
\end{theorem}

Hence under rather restrictive conditions we have that all the small deformations of $N_{\Lam}$ are obtained by just perturbing the configuration $\Lam$. However this is not the general case: the Hopf surfaces don't admit a universal family.

\subsection{Rigidity and Versality for $m=1$.} We consider the configuration corresponding to the regular polygon with 
$=2l+1$ vertices (see section \ref{Connected sums1}). Let $n=n_1+\dots+n_k$ be an ordered
partition of $n$ with $d\ge 4$. Let $\Lam=(\lambda_1,\hdots,\lambda_k)$, $\lambda_i\in\C$ be the admissible configuration where the multiplicity of $\lambda_i$  is $n_i$.

Recall that the complex structure on $N(\Lam)$ does not vary
within the affine equivalence class of $\Lam$. We show now that
the converse is true in most of the cases. These include in
particular all cases with $k>5$. It is plausible that the result is
true in general.

\vskip10pt \noindent {\bf Theorem 5.} {\it Let $n=n_1+\dots+n_k$ be
an ordered partition of $n$ with $d\ne 2$. Then any two collections
of eigenvalues corresponding to this partition give holomorphically
equivalent manifolds $N$ if, and only if, they are affinely  
equivalent.} 

\begin{proof}The sufficiency of the condition was observed above. For the
necessity, if $d=1$ we are in the Calabi-Eckmann case, and this was
shown by Loeb and Nicolau (\cite{LN}, proposition 12). For $d>2$ we
follow their argument:

Let $V={\mathcal{S}}/\C^*$ which is an open subset of $\CP^{n-1}$. Then
the complement of $V$ in $\CP^{n-1}$ is a union of projective
subspaces whose smallest codimension is $d$. By the results of
Scheja \cite{Sc} we have that

$$H^i(V, {\mathcal O}) = H^i(\CP^{n-1}, {\mathcal O}) \ for \ i\le d-2$$
where ${\mathcal O}_X$ denotes the sheaf of holomorphic functions on a
manifold. The second cohomology groups were computed by Serre and
are $\C$ in dimension 0 and trivial otherwise (see e.g. \cite{GH}
p.118).

Now, let ${\mathcal O}^{inv}$ be the kernel of the map ${\mathcal O} \to
{\mathcal O}$ given by the Lie derivative along the vector field $\xi$
which generates the $\C$ action on $V$, so we have an exact sequence
of sheaves:

\[
0 \lra \mathcal{O}^{inv} \rightarrow {\mathcal O} \overset{L_{\xi}}{\lra}{\mathcal{O}} \lra 0
\]

The associated cohomology exact sequence shows that, for $d\ge3$,\break 
$H^1(V, {\mathcal O}^{inv})= \C$, but this group can be
identified with $H^1(N,{\mathcal O})$. Therefore this group is also $\C$
and since it classifies the principal $\C$-bundles over $N$, any two
non-trivial principal $\C$-bundles over $N$ differ by a scalar
factor.

Let $N_1,N_2$ be two such manifolds which are holomorphically
equivalent and consider a biholomorphism $\phi: N_1 \to N_2$. Over
each $N_i$ there is a principal $\C$-bundle $V_i \to N_i$, where the
total space $V_i$ is in both cases $V$, but is foliated in two
different ways by the projectivized leaves of each system. We have
to lift $\phi$ to an equivalence of the principal $\C$-bundles
$V_i$, which amounts to finding an equivalence between $V_1$ and
$\phi^*V_2$. Now $V_1$ and $\phi^*V_2$ are non-trivial $\C$-bundles
(otherwise they would have sections, $N_i$ would embed
holomorphically in $\CP^{n-1}$ and would be a K\"ahler manifold,
recall \cite{We}, p.182). By the previous computation these differ by a
scalar factor and there is an equivalence between $V_1$ and $V_2$
preserving the leaves of the foliations. By Hartog's Theorem this
equivalence extends to one of $\CP^{n-1}$ into itself which must
then necessarily be linear since the group of biholomorphisms of 
$\CP^{n-1}$ is the corresponding projective linear group. But then it
follows easily that the corresponding eigenvalues must be affinely
equivalent, and Theorem 5 is proved. \end{proof}

Theorem 5 says that when $d\ne 2$ the reduced deformation space of
$N$ injects into its universal deformation space. For $d=1$ the
question of whether the reduced deformation space is universal or
not depends on the existence of resonances among the $\lambda_i$
(see \cite{Ha}, \cite{LN}). For $d\ge 4$ the situation is simpler and only
depends on the condition that all the $\lambda_i$ be different:
\vskip10pt

\noindent {\bf Theorem 6.} {\it Let $n=n_1+\dots+n_k$ be an ordered
partition of $n$ with $d\ge 4$. Let $\Lambda$ be a collection of
eigenvalues corresponding to this partition and assume that all
$\lambda_i$ are different. Then the corresponding reduced
deformation space of $N(\Lambda)$ is universal.} \vskip10pt
\noindent {\bf Proof:}  Following again \cite{LN} we consider
the exact sequences of sheaves over $V$:

$$0 \to \Theta^{inv} \to \Theta  \overset{L_\xi}{\lra} \Theta \to 0$$

$$0 \to {\mathcal O}^{inv} \xi \to \Theta^{inv} \to \Theta_b \to 0$$

\noindent where $\Theta$ denotes the sheaf of holomorphic vector
fields on a manifold and $\Theta^{inv}$ and $\Theta_b$ are defined
by these sequences. Now again by Scheja \cite{Sc} we have

$$H^i(V, \Theta) = H^i(\CP^{n-1}, \Theta) \ for \ i\le d-2$$

$H^0(\CP^{n-1}, \Theta)$ is the space of holomorphic global vector
fields on $\CP^{n-1}$ (all of which are linear) and can be
identified with the space of $n\times n$ matrices modulo the scalar
ones. For $i>0$, $H^i(\CP^{n-1}, \Theta)=0$.

The first sequence above gives a cohomology exact sequence for 
$d\geq4$:

$$0 \to H^0(\Theta^{inv}) \to H^0(\Theta) \overset{L_{\xi}}{\lra} H^0(\Theta)
\to H^1(\Theta^{inv}) \to 0.$$

Since $\xi$ corresponds to the diagonal matrix with entries
${\lambda_i}$ and these are different, the kernel and cokernel of
$L_\xi$ can be identified with the space of diagonal matrices modulo
the scalar ones, so $H^1(\Theta^{inv})$ is a space of dimension
$n-1$. The class of $\xi$ in this vector space is non-zero.

From the exact sequences of sheaves we have the diagram: {\small
$$\begin{matrix}
H^0(\Theta_b) \to & H^1({\mathcal O}^{inv}) &\to& H^1(\Theta^{inv}) &
\to H^1(\Theta_b)\to 0\cr
         &\uparrow{\cong} && \uparrow{} &\cr
        & H^0({\mathcal O})&\to& H^0(\Theta) &
\end{matrix}
$$
\noindent where the two middle horizontal maps are induced by
multiplication by $\xi$. Since the lower one is injective by the
above remark, it follows that so is the upper one and that
$H^1(\Theta_b)$ is of dimension $n-2$.

Now it is easy to see that $H^i(\Theta_b)$ is isomorphic to
$H^i(N,\Theta)$. It follows that $H^1(N,\Theta)$ is of dimension
$n-2$ and is the tangent space to the universal deformation space of
$N$. Since we have shown that the reduced deformation space is
smooth, has dimension $n-2$ and injects into this universal space,
it follows that it is itself a universal deformation space and
Theorem 6 is proved. \qed

\noi Observe that in Theorem 4 for $l\ge 4$ the space of complex
structures is the universal deformation space for any of its
members.

\def\mapup#1{\Bigg\uparrow \rlap{$\vcenter{\hbox{$\gatoriptstyle#1$}}$}}

\subsection{Global deformation theory of LVM manifolds} Here the deformation theory of equivariant LVM manifolds is explained and then together with the reconstruction theorem we conclude that this implies the existence of the moduli stack of torics.

Let $\Lambda$ be an admissible configuration. We want to describe the set $\mathcal M_\Lambda$ of $G$-biholomorphism classes of LVM manifolds $N_{\Lambda'}$ 
such that $\mathcal{S}_{\Lambda'}$ is equal to $\mathcal{S}_{\Lambda}$ up to a permutation of coordinates in $\mathbb C^n$.

We assume that $\Lambda$ 
satisfies \eqref{firstmcondition} and 
\begin{equation}
 \label{km}
 \Lambda_i\text{ is indispensable}\iff i\leq k
\end{equation}
that is, the $k$ indispensable points are the first $k$ vectors of the configuration. In the same way, every class $[N_{\Lambda'}]$ of $\mathcal M_\Lambda$ can be represented by a 
configuration $\Lambda'$ satisfying \eqref{firstmcondition}, \eqref{km} and

\begin{equation}
\label{S=}
\mathcal{S}:=\mathcal{S}_\Lambda=\mathcal{S}_{\Lambda'}.
\end{equation}

\begin{remark}

Condition \eqref{S=} is equivalent to $K_{\Lambda}$ being combinatorially equivalent to $K_{\Lambda'}$ with same numbering \eqref{Knumbering}. Observe that 
because of our convention \eqref{km},
having the same numbering implies having the same number of indispensable points.
\end{remark}

Now, observe that, because of \eqref{firstmcondition}, there exists an affine transformation $T$ of $\mathbb C^m$ sending $\Lambda$ onto a configuration 
(which we still denote by $\Lambda$) whose first $m+1$ vectors satisfies
\begin{equation}
\label{LambdaAffineNorm}
\Lambda_1=ie_1,\, \Lambda_2-\Lambda_1=e_1,\, \hdots,\,\Lambda_{m+1}-\Lambda_1=e_m,\,\,\,
\rm{where} \,\,(e_1,\hdots,e_m)\, \rm{is\, the\, canonical \, basis \,of}\,\, \C^m.
\end{equation}

It is straightforward to check that this does not change $N_{\Lambda}$ up to $G$-biholomor\-phism. In the same way, each class of  $\mathcal M_\Lambda$ can be 
represented by an element $\Lambda'$ satisfying \eqref{firstmcondition}, \eqref{km}, \eqref{S=} and \eqref{LambdaAffineNorm}. We call $\mathcal{S}$-{\it normalized 
configuration} such a configuration.

Let $\mathcal T_\Lambda$ be the set of $\mathcal{S}$-normalized configurations. This is an open and connected set in $(\mathbb C^m)^{n-m-1}$.
\vspace{5pt}\\
Assume now that $N_{\Lambda'}$ is $G$-biholomorphic to $N_\Lambda$. Then, $G_\Lambda$ and $G_{\Lambda'}$ as subgroups of $\text{Aut}(N_\Lambda)$, respectively 
$\text{Aut}(N_{\Lambda'})$ are isomorphic Lie groups. Hence, their universal cover are isomorphic as Lie groups, that is, using the presentation given in 
Proposition \ref{Glattice}, there exists a matrix $M$ in $\text{GL}_{n-m-1}(\mathbb C)$ which sends the lattice of $G_\Lambda$ bijectively onto that of 
$G_{\Lambda'}$. Using notations \eqref{Alattice} and \eqref{Blattice}, this means that there exists a matrix $P$ in $\text{SL}_{n-1}(\mathbb Z)$ such that 
\begin{equation}
\label{isoAB}
M(Id, B_\Lambda A_\Lambda^{-1})=(Id,B_{\Lambda'}A_{\Lambda'}^{-1})P.
\end{equation}
Decomposing $P$ as
\begin{equation}
\label{Pdec}
P=
\begin{pmatrix}
 P_1 &P_2\cr
 Q_1 &Q_2
\end{pmatrix}
\end{equation}
with $P_1$ a square matrix of size $n-m-1$ and $Q_2$ a square matrix of size $m$, we obtain
\begin{equation}
\label{equa1}
MB_\Lambda A_\Lambda^{-1}=(P_1+B_{\Lambda'}A_{\Lambda'}^{-1}Q_1)B_\Lambda A_\Lambda^{-1}=P_2+B_{\Lambda'}A_{\Lambda'}^{-1}Q_2.
\end{equation}

Because of \eqref{LambdaAffineNorm}, this means that 

\begin{equation}
\label{isoLambda}
^t \kern-2ptB_\Lambda=(\,^t\kern-2pt P_2+\,^t\kern-2pt Q_2\,^t\kern-2pt B_{\Lambda'})(\,^t\kern-2pt P_1+\,^t\kern-2pt Q_1\,^t\kern-2pt B_{\Lambda'})^{-1}
\end{equation}

that is
\begin{proposition}
\label{iso}
 Let $\Lambda$ and $\Lambda'$ be two $\mathcal{S}$-normalized configurations. Then $N_\Lambda$ and $N_{\Lambda'}$ are $G$-biholomorphic if and only if $\Lambda$ and $\Lambda'$ satisfies \eqref{isoLambda}.
\end{proposition}

Thus, $\mathcal M_\Lambda$ is the quotient of $\mathcal T_\Lambda$ by the action of $\text{SL}_{n-1}(\mathbb Z)$ described in \eqref{isoLambda}. We claim

\begin{proposition}
\label{orbifold}
If the number $k$ of indispensable points is less than $m+1$, then the moduli space $\mathcal M(X)$ is an orbifold.
\end{proposition}

{\bf Proof}
From the previous description, it is enough to prove that the stabilizers of action \eqref{isoLambda} are finite. Let $f$ be a $G$-biholomorphism of $N_\Lambda$. Set
\begin{equation}
\label{S1}
\mathcal{S}_1=\{w\in\mathbb C^{n-m-1}\quad\vert\quad (1,\hdots,1,w)\in\mathcal{S}\}.
\end{equation}
Observe that \eqref{S1} is a covering of the quotient $N_1$ of $\mathcal{S}\cap \{z_1\cdots z_{m+1}\not= 0\}$ by the action \eqref{actionLVM}. Indeed, we have a commutative diagram
\begin{equation}
\label{incl}
\begin{CD}
(\mathbb C^*)^{n-m-1}@>>> \mathcal{S}_1 @>>> \mathcal{S}\\
@VVV @VVV @VVV\\
G_\Lambda @>>> N_1 @>>> N_\Lambda
\end{CD}
\end{equation}
where the horizontal maps are inclusions and the first two vertical ones are coverings.

Then, up to composing with a permutation of $\mathbb C^n$, we may assume that $f$ sends 
$N_1$
onto itself. 
Because of assumption \eqref{km}, the set \eqref{S1} is a $2$-connected open subset of $\mathbb C^{n-m-1}$, hence the restriction of $f$ to $N_1$, say $f_1$, lifts to a biholomorphic map $F_1$ of $\eqref{S1}$. More precisely, $\mathcal{S}_1$ is equal to $\mathbb C^{n-m-1}$ minus a finite union of codimension $2$ vector subspaces, hence by Hartogs, $F_1$ extends as a biholomorphism of $\mathbb C^{n-m-1}$.

On the other hand, the restriction of $f$ to $G_\Lambda$ preserves $G_\Lambda$ and lifts as a biholomorphism $\tilde F$ of its universal covering $\mathbb C^{n-m-1}$. And we have a commutative diagram
\begin{equation}
\label{lifts} 
\begin{CD}
\mathbb C^{n-m-1}@>\exp (2i\pi -)>>(\mathbb C^*)^{n-m-1}\\
@V \tilde F VV @VVF_1 V\\
\mathbb C^{n-m-1}@>\exp (2i\pi -)>> (\mathbb C^*)^{n-m-1}
\end{CD}
\end{equation}
But, since the linear map $\tilde F=M$ must preserve the abelian subgroup of Proposition \ref{Glattice}, using \eqref{LambdaAffineNorm} and \eqref{isoAB}, we have
\begin{equation}
\label{equ3}
\tilde F(z+e_i)=\tilde F(z)+P_1e_i:=\tilde F(z)+\sum_{j=1}^{n-m-1}a_{ij}e_j
\end{equation}
that is $Q_1$ is equal to $0$. But through \eqref{lifts}, this implies that 
\begin{equation}
\label{equ4}
F_1(w)=\Big (w_1^{a_{1j}}\cdots w_{n-m-1}^{a_{n-m-1j}}\Big )_{j=1}^{n-m-1}
\end{equation}

Now, recall that $F_1$ is a biholomorphism of the whole $\mathbb C^{n-m-1}$, so must send a coordinate hyperplane onto another one without ramifying. This shows that $P_1=(a_{ij})$ is a matrix of permutation. Hence every stabilizer is a subgroup of the group of permutations with $n-m-1$ elements, so is finite.

\begin{example}
\label{tori}
{\bf Tori.} Let $n=2m+1$, then there are $2m+1$ indispensable points, $\mathcal{S}$ is $(\mathbb C^*)^n$ and $N$ is a compact complex torus of dimension $m$ \cite[Theorem 1]{Me}. The associate polytope $K$ is reduced to a 
point and $N=G$. The moduli space $\mathcal M$ is equal to the moduli space of compact complex tori of dimension $m$, which is not an orbifold for $m>1$.
\end{example}

{\bf Example}
\label{Hopf}
{\bf Hopf surfaces.} Let $n=4$ and $m=1$, then there are two indispensable points and $\mathcal{S}$ is $(\mathbb C^*)^2\times\mathbb C^2\setminus\{(0,0)\}$.
A $\mathcal{S}$-admissible configuration is given by a couple complex numbers $(\lambda_3,\lambda_4)$ belonging to
\begin{equation}
 \label{zoneHopf}
 \{z\in\mathbb C\quad\vert\quad \Re z <0\text{ and } \Re z<\Im z\}.
\end{equation}
The manifold $N_\Lambda$ is equal to the diagonal Hopf surface obtained by taking the quotient of $\mathbb C^2\setminus\{(0,0)\}$ by the group generated by
\begin{equation}
 \label{eqHopf}
 (z,w)\longmapsto (\exp {2i\pi(\lambda_3-\lambda_1)}\cdot z, (\exp {2i\pi(\lambda_4-\lambda_1)}\cdot w)
\end{equation}

Two points $(\lambda_3,\lambda_4)$ and $(\lambda'_3,\lambda'_4)$ with coordinates in \eqref{zoneHopf} are equivalent if and only if their difference is in the 
lattice $\Z\oplus \Z$ or if the difference of $(\lambda_3,\lambda_4)$ by the switched $(\lambda'_4,\lambda'_3)$ is in this lattice. The isotropy group of 
a point is $\mathbb Z_2$ for the diagonal $\lambda_3=\lambda_4$ and is zero elsewhere. The moduli space is an orbifold.

Observe that not all Hopf surfaces are obtained as LVM-manifolds, but only the linear diagonal ones. Now, they coincide with the set of Hopf surfaces that are 
equivariant compactifications of $(\mathbb C^*)^2$.

\vskip10pt \noindent (b) Generalized Hopf manifolds. \vskip10pt

When $n_1=n_2=1$ the manifold $N$ is diffeomorphic to $\s^1\times \s^{2n_3-1}$. 
Here the mapping $exp:\C^2\times (\C^{n_3}\backslash 0)
\to {\mathcal{S}}=(\C^*)^2 \times (\C^{n_3}\backslash 0)$ given by
$exp(\zeta_1,\zeta_2,\zeta)=(e^{\zeta_1},e^{\zeta_2},\zeta)$ can be
used to identify $N$ with the quotient of $\C^{n_3}\backslash 0$ by
the action of $\Z$ defined by the multipliers

$$\alpha_i=\exp \left( 2\pi i \frac{\lambda_{2+i}-\lambda_2}{\lambda_1-\lambda_2}\right),
\: i=1,\dots,n_3.$$

In this case we obtain all complex structures on $\s^1\times \s^{2n_3-1}$ having $\C^n\backslash 0$ as universal cover when there
is no resonance among the $\alpha_i$. But in the resonant case we do
not obtain all such complex structures since we do not obtain the
non-linear resonant cases of Haefliger. It is clear that, in order
to obtain the latter, one must look at the resonant non-linear
versions of equation (1).

\vskip10pt \noindent (c) Generalized Calabi-Eckmann manifolds.
\vskip10pt

When $n_1=1$ and $n_2$, $n_3$ are both greater than $1$ we have seen
that the manifold $N$ is diffeomorphic to $\s^{2n_2-1}\times\s^{2n_3-1}$. Here the mapping 
$exp:\C \times (\C^{n_2}\backslash
0)\times (\C^{n_3}\backslash 0) \to {\mathcal{S}}=\C^*\times
(\C^{n_2}\backslash 0)\times(\C^{n_3}\backslash 0)$ given by
$exp(\zeta,\zeta_1,\zeta_2)=(e^\zeta,\zeta_1,\zeta_2)$ can be used
to identify $N$ with the quotient of $(\C^{n_2}\backslash
0)\times(\C^{n_3}\backslash 0)$ by the action of $\C$ defined by the
linear differential equation with eigenvalues $\lambda_i^\prime=2\pi
i ( \lambda_i-\lambda_1), \, i=2,\dots,n$. This is exactly the
construction of the Loeb-Nicolau complex structure corresponding to
a linear system of equations of Poincar\'e type \cite{LN}.

Observe that in their construction only the quotients of the
eigenvalues of the system are relevant for the definition of the
complex structure on $N$, so once again only the quotients
$\frac{\lambda_i-\lambda_k}{\lambda_j-\lambda_k}$ of our original
eigenvalues count.

Again we obtain all their examples of complex structures on
$\s^{2n_2-1}\s^{2n_3-1}$ when there is no resonance among the
$\lambda_i^\prime$. But, once more, in the resonant case we do not
obtain all their complex structures since we do not obtain the
non-linear resonant examples.

Observe that in all the cases considered in this section only the
quotients $\frac{\lambda_i-\lambda_k}{\lambda_j-\lambda_k}$ are
relevant in the description of the complex structure of $N$ (in
accordance with the observation made in section 2 that affinely
equivalent configurations of eigenvalues with the same ${\mathcal{S}}$
give the same complex structure) and that they are actually {\it
moduli} of that complex structure.

\section{LVM manifolds as equivariant compactifications}
\label{LVMBec}

Theorem  \ref{algebraic-dimension} has a deeper explanation related to the structure of $N_{\Lam}$ and the arithmetic properties
of $\Lam$. In fact  $\mathcal{S}$ contains always $(\mathbb C^*)^n$  as an open and dense subset {\it invariant under the foliation} $\tilde{\mathcal{F}}$. If we pass to the quotient under the action of \eqref{actionLVM} one obtains that $N_{\Lam}$
has as an open subset $G_{\Lam}$, which is the quotient of $(\mathbb C^*)^n$ by $\tilde{\mathcal{F}}$. Since $\tilde{\mathcal{F}}$ is defined by the action \eqref{actionLVM} and this action commutes with the group structure of the multiplicative group de $(\mathbb C^*)^n$, it follows that $G_{\Lam}$ itself is a connected commutative complex Lie group. 
In other words:

\begin{theorem} \label{GLam}\rm{({\bf \cite{LM}})} $N_{\Lam}$ is the equivariant compactification of a complex commutative Lie group $G_{\Lam}$.
\end{theorem}
\begin{remark} \textcolor{blue}{In some sense, this theorem is the principal reason of the interconnection between LVMB manifolds, toric varieties,
convex polytopes and moment-angle manifolds}
\end{remark}

\begin{definition} A  connected complex Lie group $G$ is called \emph{Cousin group} (or toroidal group in \cite{K}) if any holomorphic function on it is constant \cite{AK}.
\end{definition}

\begin{proposition} Cousin groups are commutative. Moreover, they are quotients of a complex vector space $\mathbb V$ by a
discrete additive subgroup of $\mathbb V$ \cite{AK}.
\end{proposition}
\begin{proposition} Any commutative connected complex Lie group $G$ can be written in a unique way as a product
$G=C\times\C^l\times({\C^*})^r$ where $C$ is a Cousin group ($l,r\geq0$). 
\end{proposition}
\begin{proposition} A commutative complex Lie group is Cousin if and only if it does not have nontrivial characters.
\end{proposition}
Observe that  $(\mathbb C^*)^n$ acts by multiplication on the space of Siegel leaves $\mathcal{S}_{\Lam}$ with an open and dense orbit, making it a toric variety. This action commutes with projectivization and with \eqref{actionLVM}, making of $N_\Lambda$ an equivariant compactification of an abelian Lie group, say $G_{\Lam}$. A straightforward computation shows the following \cite[p.27]{MeThesis}

\begin{proposition}
\label{Glattice}
Assume that 
\begin{equation}
\label{firstmcondition}
\text{\rm rank}_{\mathbb C}
\begin{pmatrix}
\Lambda_1 &\hdots &\Lambda_{m+1}\cr
1 &\hdots &1
\end{pmatrix}
=m+1.
\end{equation}
Then $G_{\Lam}$ is isomorphic to the quotient of $\mathbb C^{n-m-1}$ by the $\mathbb Z^{n-1}$ abelian subgroup generated by
$(Id, B_\Lambda A_\Lambda^{-1})$ where
\begin{equation}
\label{Alattice}
A_\Lambda=^t\kern-4pt(\Lambda_2-\Lambda_1,\hdots,\Lambda_{m+1}-\Lambda_1)
\end{equation}
and
\begin{equation}
\label{Blattice}
B_\Lambda= ^t\kern-4pt(\Lambda_{m+2}-\Lambda_1,\hdots,\Lambda_{n-1}-\Lambda_1).
\end{equation}
\end{proposition}

\begin{remark}
It is easy to prove that 
\begin{equation*}
\text{\rm rank}_{\mathbb C}
\begin{pmatrix}
\Lambda_1 &\hdots &\Lambda_n\cr
1 &\hdots &1
\end{pmatrix}
=m+1.
\end{equation*}
(cf. {\rm \cite[Lemma 1.1]{MV}} in the LVM case). Hence, up to a permutation, condition \eqref{firstmcondition} is always fulfilled.
\end{remark}
We say that $N_\Lambda$ and $N_{\Lambda'}$ are {\it $G$-biholomorphic} if they are $(G_\Lambda,G_{\Lambda'})$-equivariantly biholomorphic.

\begin{remark}
When $\mathcal{S}$ is $(\mathbb C^{*})^{n}$, one has
 $N=G$ is a compact complex Lie group and therefore a compact complex torus.
 This is a direct proof of the example presented in \ref{tori} 
 \end{remark}

\medskip
The structure of groups like $G_{\Lam}$ is well-known \cite{Mo}. Since we know that the dimension of $G_{\Lam}$ is equal to that of $N$
we obtain that $G_{\Lam}$ is the quotient of $(\mathbb C^*)^{n-m-1}$ by a discrete multiplicative subgroup $\Gamma$. 
The group $G_{\Lam}$ is sometimes called a semi torus \ie there exists a short equivariant exact sequence 
\[
0\longrightarrow (\mathbb C^*)^{n-m-1}\longrightarrow G_{\Lam} \longrightarrow T \longrightarrow 0
\]
where $T$ is an appropriate  compact complex torus of dimension $n-m-1$.

\medskip

Furthermore, the group $G_{\Lam}$ is isomorphic to $(\mathbb C^*)^a\times C$, where $a\geq 0$
and $C$ is a Cousin group.
Compact Cousin groups are just complex tori. However there are non-compact Cousin group,
for instance If $C=(\mathbb C^*)^{n-m-a-1}/\Gamma_0$, and $\Gamma_0$ is a ``sufficiently generic'' discrete subgroup in order to
have that any holomorphic function on $(\mathbb C^*)^{n-m-a-1}$ which is invariant under $\Gamma_0$  must be constant
then any holomorphic function on the quotient is constant.

In our case, any holomorphic function on $G$ extends to a meromorphic function on $N_{\Lam}$. 
Then theorem 2 shows is the following:

\begin{proposition}
If $N_{\Lam}$ does not have an indispensable point, then the algebraic dimension of $N_{\Lam}$ 
is equal to the dimension $a$ of the factor $\mathbb C^*$ in the associated decomposition
$G=(\mathbb C^*)^a\times C$.
\end{proposition}

Hence we obtain the following.

\begin{corollary} {\bf (\cite{Me} Proposition IV.1.)}
Let $N_{\Lam}$ be an LVM  manifold which is the equivariant compactification of the connected complex abelian Lie group $G_{\Lam}$. Suppose $N_{\Lam}$ is without indispensable points. Then one has an equivalence:

\medskip
\noindent (i) $N_{\Lam}$ does not have non-constant meromorphic functions

\noindent (ii) $G_{\Lam}$ is a Cousin group, \ie every holomorphic function on $G_{\Lam}$ is constant

\noindent (iii) System \ref{System(S)} has no solution in the rationals.
\end{corollary}

\section{Toric varieties and Generalized Calabi-Eckmann fibrations}

Let $\Lam=(\Lambda_1,\cdots,\Lambda_n)$ be a configuration which is admissible \ie it satisfies both the Siegel and weak hyperbolicity conditions as before.

Consider the system of equations:

\[
   \left\{
                \begin{array}{ll}
                 \overset{n}{\underset{i=1}\sum}s_i\Lambda_i=0\\
                                 \\
                  \overset{n}{\underset{i=1}\sum} s_i=0\\
                   \end{array}
              \right.   \tag{\textcolor{red}S} \label{System(S)}
  \]

\begin{definition} We say that the configuration $\Lam$ satisfies condition ({\bf K}) if
the dimension over $\mathbb Q$ of the vector space of rational solutions of the system \ref{System(S)}
\label{System(K)}
above is maximal,
in other words is of dimension $n-2m-1$.
\end{definition}
    Observe that any linear diagonal holomorphic vector field
$$
\xi=\sum_{i=1}^n\alpha_iz_i\dfrac{\partial}{\partial z_i}
$$
on $\C^n$ projects onto a holomorphic vector field on $N$. In particular, let
$$
\Lambda_i=(\lambda_i^1,\hdots, \lambda_i^m)
\eqno 1\leq i\leq n
$$
and define the $m$ commuting vector fields on $\mathcal S$
\[
\eta_i(z)=\left\langle \text{Re }\Lambda_i\sum_{j=1}^nz_j\dfrac{\partial}{\partial z_j}\right\rangle =\sum_{j=1}^n \Re (\lambda_j^i)z_j\dfrac{\partial}{\partial z_j} \quad\quad 1\leq i\leq m \tag{\textcolor{blue}{VF}}
\]
for $i$ between $1$ and $m$. The composition of the (holomorphic) flows of these vector fields gives an action of $\C^m$
on $N$. The following Theorem is proven in \cite{Me}, as a generalization of a result of J.J. Loeb and M. Nicolau \cite{LN2}.

\begin{theorem} {\bf (\cite{Me}, Theorem 7)}
The projection onto $N$ of the vector fields $(\eta_1,\hdots,\eta_m)$ gives 
on $N$ a regular holomorphic
foliation $\mathcal G$ of dimension $m$. Moreover, the foliation $\mathcal G$ is transversely K\"ahlerian with respect to $\omega$, 
the canonical Euler form of the bundle $M_1\to N$ (see definition \ref{Euler_class}).
\end{theorem}

\noi Recall that transversely K\"ahlerian means that

\noindent (i) $\mathcal G$ is the kernel of $\omega$.

\noindent (ii) $\omega$ is closed and real.

\noindent (iii) The quadratic form $h(-,-)=\omega(J-,-)+i\omega(-,-)$ (where $J$ 
denotes the almost complex structure of $N$) defines a hermitian metric on the normal bundle 
to the foliation.
\medskip

We note that this Theorem gives non-trivial examples of transversely K\"ahlerian foliations on compact 
complex manifolds.

The aim of this Section is to study the quotient space of $N$ by $\mathcal G$. Observe that it can be obtained as the quotient space of $\mathcal S$
by the action induced by the action \ref{actionS}.

Going back to the abelian group $G_{\Lam}$ of theorem \ref{GLam}, we see that its Lie algebra is generated by the linear vector fields $z_i\partial/\partial z_i$ for $i=1,\hdots, n$. Due to the quotient by \eqref{actionS}, they only generate a vector space of dimension $n-m-1$ as needed. Amongst these $n-m-1$ linearly independent vector fields, we can find $m$ of them which extend to $\mathcal{S}_{\Lam}$ without zeros and which generates a locally free action of $\mathbb C^m$ onto $\mathcal{S}_{\Lam}$. For example, we can take the commuting vector fields in formula {\textcolor{blue}{VF}} above.

\begin{definition} ({\bf Canonical foliation}).  We denote by $\mathcal G={\mathcal G}_{\Lam}$ the foliation induced by this action. 
This is the \emph{canonical foliation} of $N_{\Lam}$
\label{canonical-foliation}
\end{definition}

It is easy to check that $\mathcal G$ is independent of the choice of vector fields. Indeed, changing the vector fields just means changing the parametrization of $\mathcal G$, that is changing the $\mathbb C^m$-action by taking a different basis of $\C^m$.

Pull back the Fubini-Study form of $\mathbb P^{n-1}$ to the embedding \eqref{Nsmooth}. This is the {\it canonical Euler form} $\omega$ of $\Lambda$, as defined in definition \ref{Euler_class}. It is a representative of the Euler class of a particular 
$\s^1$-bundle associated to $N _{\Lam}$, hence the name. Then $\mathcal G$ is transversely K\"ahler with transverse K\"ahler form $\omega$. For our purposes, we will not focus on $\omega$ but on the ray $\mathbb R^{>0}\omega$ it generates Recall that $\Lam$ fulfills  condition ({\bf K}) if \ref{systemS} admits a basis of solutions with integer coordinates; and that $\Lam$ fulfills condition (H) if \ref{System(S)} does not admit any solution with integer coordinates. If condition {\bf K} in definition \ref{System(K)} is fulfilled, then $\mathcal G$ is a foliation by compact complex tori and the quotient space is a projective toric orbifold, see \cite{MV} which contains a thorough study of this case.

We just note here that, even if condition ({\bf K}) in definition \ref{System(K)}  is not satisfied, the foliation $\mathcal G$ has some compact orbits. Indeed, let $I$ be a vertex of $K_\Lambda$. Then, by \eqref{Kcombinatoire}, $0$ belongs to $\mathcal H(\Lam_{I^c})$, so by \cite[Lemma 1.1]{MV},
\begin{equation}
\label{secondmcondition}
\text{\rm rank}_{\mathbb C}
\begin{pmatrix}
\Lambda_{i^c_1} &\hdots &\Lambda_{i^c_{2m+1}}\cr
1 &\hdots &1
\end{pmatrix}
=m+1.
\end{equation}
Hence, up to performing a permutation, we may assume at the same time \eqref{firstmcondition} and 
\begin{equation}
\label{inter}
I\cap \{1,\hdots, m+1\}=\emptyset.
\end{equation}

\begin{definition}\label{definition-toric} An $n$-dimensional toric variety $W$ (possibly singular) is an algebraic variety with an open and dense subset biholomorphic to $({\C^*})^n$ such that the  natural action of  $({\C^*})^n$ extends to a holomorphic action on all af $W$.
In other words: a toric variety of complex dimension $n$ is an algebraic variety which is an equivariant compactification
of the abelian algebraic torus  $({\C^*})^n$.
\end{definition}
    
 \noi We have
\begin{proposition}
\label{verticesorbits}
For each vertex $I$ of $K_{\Lam}$, the corresponding submanifold $N_I$ is a compact complex torus of dimension $m$ and is a leaf of $\mathcal F$. Moreover, assume that $\Lambda$ satisfies \eqref{firstmcondition} and \eqref{inter}. Then, letting $B_I$ denote the matrix obtained from \eqref{Blattice} by erasing the rows $\Lambda_i-\Lambda_1$ for $i\in I$, the torus $N_I$ is isomorphic to the torus of lattice $(Id,B_IA_{\Lam}^{-1})$. 
\end{proposition}

The following theorem is the fundamental connection between toric varieties with at most quotient singularities
(\ie quasi-regular varieties) and LVM manifolds.
\newpage
  
\begin{theorem} {\bf (\cite{MV} Theorem A)}
\label{theorem-generalized-CE}
Let $N$ be one of our manifolds corresponding to a configuration which satisfies condition \rm({\bf K}) in definition \ref{System(K)}.
Then $N$ is a Seifert-Orlik fibration in complex tori of dimension $m$ over a quasi-regular, projective, toric variety of dimension $n-2m-1$. More precisely:
Let $\Lam$ be an admissible configuration satisfying condition \rm{({\bf K})} Then
\begin{enumerate}
\item The leaves of the foliation $\mathcal G$ of $N_{\Lam}$ are compact complex tori of 
dimension $m$.

\item The quotient space of $N_{\Lam}$ by $\mathcal G$ is a projective toric 
variety of dimension
$n-2m-1$. We denote it by $X(\Delta)$, where $\Delta$ is the corresponding fan.

\item The toric variety $X(\Delta)$ comes equipped with an equivariant orbifold structure.

\item The natural projection $\pi : N\to X(\Delta)$ is a holomorphic 
principal Seifert bundle, with compact complex tori of dimension $m$ as fibers.

\item The transversely K\"ahlerian form $\omega$ of $N$ projects onto a 
K\"ahlerian 
(singular at the singular locus of $X(\Delta)$ as a variety) form $\tilde \omega$ of $X(\Delta)$.
\end{enumerate}
Moreover, condition \rm{({\bf K})} is optimal with respect to these properties in the 
sense that the foliation $\mathcal G$ of a configuration which does not
satisfy it has non-compact 
leaves, so item 1 is not 
verified.
\end{theorem}

\section{Idea of the proof of theorem \ref{theorem-generalized-CE}}
\subsection{Toy example}\label{xemple-introductif}

We start with an example that shows the close relationship between LVM manifolds and Calabi-Eckmann manifolds

\begin{example}\label{Hopf-surfaces}
Consider the admissible configuration given by
$$
\Lambda_1=\Lambda_2=1 \qquad \Lambda_3=i \qquad \Lambda_4=-1-i.
$$

We have seen in \ref{Hopf-manifolds} that the manifold corresponding $N$ is diffeomorphic to $\s^1\times\s^3$, \ie $N$ is a primary Hopf surface primaire.
One knows  (\cite{BHPV} Chapter V, Proposition 8.18) that such a surface contains either exactly two elliptic curves or else
it is an elliptic fibration over $\CP^1$. In addition these two cases are distinguished by their algebraic dimensions: in the first case the algebraic dimension is 0 and in the second it is 1.

\medskip

Although theorem  \ref{algebraic-dimension} does not apply directly here since the configuration has two indispensable points
one still has that the algebraic dimension of $N$ is {\it  greater or equal to} $a$, the number of rational solutions
which are $\mathbb Q$-linearly independent of system \ref{System(S)}. Here $a=1$, the solutions of system \ref{System(S)}
are generated by
 
\[
s_1=1 \qquad s_2=-1 \qquad s_3=0\qquad s_4=0
\] 

It follows that the algebraic dimension of $N$ is equal to one and that $N$ fibers over $\CP^1$ with fiber an elliptic curve. 

Let us now consider $\Lambda_j=a_j+ib_j$  and the following action of  $\C^2$ in $\CP^3$:
\[
((t_1,t_2), [z]) \in (\C^*)^2\times\CP^3\longmapsto [z_i\cdot t_1^{a_i}\cdot t_2^{b_i}]_{i=1}^{i=4}=[z_1t_1,z_2t_1,z_3t_2,z_4/(t_1t_2)]\in\CP^3  \tag{$\bullet$}
\]
Let us restrict this action to $\mathcal V$ (the projection to $\CP^3$ of the open set of Siegel leaves of the system 
definition \ref{Siegel-projective}
)
\[
\mathcal V=\{[z]\in\CP^3 \quad|     \quad (z_1,\hdots,z_4)\in (\mathbb C^2\setminus \{(0,0)\})\times (\mathbb C^*)^2\} \
\]
The projection
$$
[z_1,\hdots,z_4]\in \mathcal V\longmapsto [z_1,z_2]\in\CP^1
$$
is invariant invariante under the action $\bullet$ , hence the quotient of $\mathcal V$  under the action can be identified
with $\CP^1$. 

Consider now the action \ref{actionLVM} of
$\mathbb C$ on $\mathcal V$ (definition \ref{Siegel-projective}). 
This action commutes with the action of
of $(\mathbb C^*)^2$,
hence it respects the projection.  In fact, the inclusion
$$
T\in\mathbb C\longmapsto (t_1=\exp T,t_2=\exp iT)\in (\mathbb C^*)^2 
$$
intertwines the two actions:           $N$ is given by the action ($\bullet$) of
 $(\mathbb C^*)^2$ restricted to couples $(t_1,t_2)$ on its image. 
 So one has the commutative diagram:
\[
\begin{array}{ccc} \mathcal{V} &\overset{Id}\longrightarrow& \mathcal{V}\\
              \Big\downarrow &\ & \Big\downarrow\\
              N &\overset{p}\longrightarrow& \CP^{n-1}
  \end{array}
\]

On the other hand the fibers of the projection in the righthand side are biholomorphic to
$(\mathbb C^*)^2$, and the fibers in the lefthand side are biholomorphic to
 $\mathbb C$. The the fibers of
 $p$  are given by the quotient of $(\mathbb C^*)^2$ 
$\mathbb C$ where $\C$ acts on $(\mathbb C^*)^2$ by the inclusion defined above.
A direct calculation shows that the fibers are elliptic curves isomorphic to the quotient
of $\mathbb C^*$ by the group generated by the homothety 
$z\to \exp {2\pi}\cdot z$. In this way we obtain the elliptic fibration of $N$ over $\CP^1$.
\end{example}

Everything in the preceding example can be generalized to the case of any manifold $N=N_{\Lam}$ where $\Lam$ verifies
condition ({\bf K}) in definition \ref{System(K)} .  
Let
$$
\Lambda_{j}=a_{j}+ib_{j}
\leqno 1\leq j\leq n
$$
and consider the action of $\mathbb C^{2m}$ onr $\CP^{n-1}$ given by the formula:
\[
(R,S, [z])\in\mathbb C^m\times\mathbb C^m \times \CP^{n-1}\longmapsto [z_j\cdot\exp{\langle a_j, R\rangle}\cdot\exp{\langle b_j, S\rangle}]_{j=1}^n\in\CP^{n-1}
\]
Since $N$ verifies condition ({\bf K}) in definition \ref{System(K)}, up to replacing the configuration 

$\Lam$ by $A(\Lam)$ where $A$ is an appropriate {\it real}  affine transformation of $\mathbb R^{2m}\simeq\mathbb C^{m}$ one can assume that the real and imaginary parts of each $\Lambda_{j}$ are vectors belonging to the lattice $\mathbb Z^{m}$ (\cite{MV} Lemma 2.4):

\medskip

This means that the preceding action of $\mathbb C^{2m}$ on $\CP^{n-1}$ is equivalent to an algebraic action of $(\mathbb C^{*})^{2m}$ on $\CP^{n-1}$

\[
(t,s,[z])\in (\mathbb C^{*})^{m}\times (\mathbb C^{*})^{m}\times\CP^{n-1}\longmapsto [z_{1}\cdot t^{a_{1}}\cdot s^{b_{1}},\hdots ,
z_{n}\cdot t^{a_{n}}\cdot s^{b_{n}}]\in\CP^{n-1}  \tag{8} \label{algebraic-action}
\]

\noi where $a_{j}$ and $b_{j}$ belong to $\mathbb Z^{m}$. Here $t^{a_{j}}$ (respectively $s^{b_{j}}$) means $t_{1}^{a_{j}^{1}}\cdot \hdots\cdot t_{m}^{a_{j}^{m}}$ (respectively
$s_{1}^{b_{j}^{1}}\cdot \hdots\cdot s_{m}^{b_{j}^{m}}$).

\noi When one restricts the action (8) to $\mathcal V$ (\ref{Siegel-projective}), one can show that one obtains as quotient a projective toric variety
 $X$. In fact this procedure is precisely the construction of toric varieties as {\bf GIT}  (Geometric Invariant Theory) quotients
 by de David Cox  in \cite{Co}. One simply verifies that the open set
 $\mathcal V\subset\CP^{n-1}$ corresponds to the semi-stable points for the natural linearization of $\mathbb C^{n}\to \CP^{n-1}$
 (\cite{MV}, Lemma 2.12).
 Since in our case the quotient is a geometric quotient (the orbit space which is Hausdorff) and not a quotient where one identifies instead the closure of the orbits, one deduces that the semi-stable points are in fact stable and, via 
\cite{Co}, that the quotient is a projective quasi-regular toric variety \ie it possesses at worst quotient singularities,

\medskip
Let $i$ be the inclusion:  $ T\in\mathbb C^{m} \overset{i}{\longmapsto} (\exp T, \exp iT)\in (\mathbb C^{*})^{m}\times (\mathbb C^{*})^{m}$

Like in the toy example one can restrict the action (8) to the pairs in $(\mathbb C^{*})^{m}\times (\mathbb C^{*})^{m}$ 
that are in the image of $i$. This way one obtains an {\it algebraic} action of
$\mathbb C^{m}$ on $\mathcal V$. This action is precisely the action of formula (2). One obtains the same type of commutative diagram as in the toy example:

\begin{center}
\[
\begin{tikzcd}
\mathcal V\arrow{r}{Id}\arrow{d}[swap]{} & \mathcal V\arrow{d} \\
N\arrow{r}[swap]{p} & X \\
\end{tikzcd} \tag{\textcolor{blue}{CE}}\label{diagramCE}
\]
\end{center}

A calculation shows that the fibers of $p:N_{\Lam}\to{X}$ are compact complex tori of complex dimension $m$. This is equivalent
to showing that every isotropy group  under the action is a lattice isomorphic to $\Z^{2m}$. In fact this lattice can be explicitly calculated here is where one uses the rationality condition ({\bf K}) in definition \ref{System(K)}.
The lattice is constant in an open an dense set (in the Hausdorff or Zariski topology)
but it could have special fibers that are finite quotients of the typical fibre (all the fibers are isogenous). 
In other words: the projection $p$ corresponds to the quotient (\ie the orbit space) of $N$ by the holomorphic action
of a compact complex torus of complex dimension $m$ acting with finite isotropy groups. This implies that $X$ has the structure 
structure of an orbifold such that $p:N\to{X}$ is a Seifert-Orlik fibration \cite{Or}.

\begin{remark}
The pre-image under  $p:N\to{X}$ of a singular point of $X$ is necessarily a special fiber. 
However there could be above regular points special fibers. In fact, the locus on the base $X$ having special fibers could be of codimension one but the singular locus of $X$ as a normal projective toric manifold must have codimension at least two.
The reason of this difference is that $X$ is an orbifold in addition to being toric and this structure could have ``fake''
codimension one singularities. For instance in the examples \ref{Hopf-manifolds} of 
if one replaces $\Lambda_{2}=1$ by $\Lambda_{2}=p$, one constructs a Hopf surface with an elliptic fibratio
over $\CP^{1}$ having two singular points of orbifold type and one or two special fibers.
\end{remark}

\noi Theorem \ref{theorem-generalized-CE} has the following corollary:

\begin{corollary}
 Let $N$ satisfy the conditions of theorem  \ref{theorem-generalized-CE}. Then the algebraic reduction of $N$
 is a quasi-regular, projective, toric variety of dimension $n-2m-1$.
 \end{corollary}

 As a particular case of the previous theorem one recovers the elliptic fibrations  used by E. Calabi et B. Eckmann to provide the product of spheres $\s^{2p-1}\times\s^{2q-1}$ (for $p>1$ et $q>1$)
with a complex structure. This generalization is given by the following

\begin{definition}
A generalized Calabi-Eckmann fibration is the fibration obtained by the previous theorem.
\end{definition}

\noi Since we know, fixing $m$ and $n$, that the set of configurations satisfying condition ({\bf K}) in definition \ref{System(K)} is dense in the space of admissible configurations
on obtains:

\begin{corollary} Every manifold $N$ corresponding to an admissible configuration is a small deformation of a generalized Calabi-Eckmann fibration
\end{corollary}

\begin{remark}
All the fibers are isogeneous complex tori.
\end{remark}
\medskip

In the case where $X(\Delta)=\CP^p\times \CP^q$ with its manifold structure, 
then the fibration is the one in
elliptic curves
$$
\s^{2p+1}\times \s^{2q+1}\to\Bbb CP^p\times \Bbb CP^q
$$
\noindent where $\s^{2p+1}\times \s^{2q+1}$ is endowed with a Calabi-Eckmann complex 
structure
(see \cite{CE}). This explains the following definition

\definition{Definition 2.11}
We call such a Seifert bundle $N\to X(\Delta)$ a generalized Calabi-Eckmann 
fibration over 
$X(\Delta)$.
\enddefinition

\begin{corollary}
Let $\Lam$ be an admissible configuration satisfying condition  \rm{({\bf K})}.
Let $z\in \mathcal S$ and let 
$$
J_z=\{i_1,\hdots,i_p\}=\{1\leq i\leq n\quad\vert\quad z_i\not =0\}
$$

Then, 
\begin{enumerate}
\item The lattice in $\C^m$ of the orbit through $z$ is $2\pi L^*_{{(J_z)}^c}$.

\item The orbit through $z$ is an exceptional orbit if and only if $L_0\subsetneq L_{{(J_z)}^c}$. 
In this case, it is a finite unramified
quotient of the generic orbit of degree the index of $L_0$ in $L_{{(J_z)}^c}$. 
\end{enumerate}
\end{corollary}
  
The construction in the preceding section is completely reversible. Let $X$ be a quasi-regular toric projective variety 
(\ie if it has singularities they are quotient singularities) then the construction by David Cox in\cite{Co} permits to realize
$X$ as the quotient of an open set $X_{ss}$ by a linear algebraic action of  $(\mathbb C^*)^p$ 
like the one given in formula (\ref{algebraic-action}) above. One can arrange this action in order to have $p$ even and set
$m=p/2$. Then one defines the configuration by the formulae:

\[
\label{realization-Lambda}
\Lam=\{\Lambda_{j}=a_{j}+ib_{j} \quad |  \quad\quad   1\leq j\leq n\}\tag{\textcolor{blue}{Realization of $\Lam$}}
\]
\noi where the natural numbers $a_j$    $b_j$ are the weights (like in formula  (\ref{algebraic-action}) ) of the algebraic action of $(\mathbb C^*)^{2m}=(\mathbb C^*)^m\times (\mathbb C^*)^m$ and one induces an action of
$\mathbb C^m$ on $X_{ss}$ via the inclusion $i$ defined above. To achieve one uses the following technical lemma
found in {\bf(\cite{MV}, Lemmas 2.12 et 4.9.)}:

\begin{lemma} {\bf(\cite{MV}, Lemmas 2.12 et 4.9.)}
With the previous definition the configuration
$\Lam$ is admissible and satisfies condition {\rm ({\bf K})} in definition \ref{System(K)} . 
In addition the open set ${\mathcal V}(\Lam)$ (see definition \ref{Siegel-projective}) is equal to  $X_{ss}$.
\end{lemma}

One can show without difficulty that the variety $X$ obtained by Cox construction is the generalized Calabi-Eckmann fibration associated to $N_{\Lam}$ 
via the commutative diagram \ref{diagramCE}.

\noi Therefore one the following theorem which is the reciprocal of theorem \ref{theorem-generalized-CE}:

\begin{theorem}
 Let $X$  be a projective, quasi-regular, toric variety. Then there exists $m>0$ and a manifold $N$ corresponding to an admissible configuration which
 admits a generalized Calabi-Eckmann over $X$ and whose fibres are  complex torii of complex dimension $m$.

Furthermore, if $X$ is nonsingular (smooth), one can choose $m$ and $N$ such that the fibration is a holomorphic principal fibration.
\end{theorem}

\begin{remark}\textcolor{blue}{The previous theorem motivated a possible definition of {non-commutative toric varieties} and its deformations
(usual toric varieties are rigid). See \cite{BaZ, KLMV}.}
\end{remark}

\begin{example} {\bf(\cite{MV}, Proposition I.)} ({\bf Hirzebruch surfaces})
Let $a\in\mathbb N$. Then the manifold $N_{\Lam}$ associated to the admissible configuration
$\Lam=\{\Lambda_1,\Lambda_2,\Lambda_3,\Lambda_4,\Lambda_5\}$ with
$$
\Lambda_1=\Lambda_4=1 \qquad\Lambda_2=i 
\qquad \Lambda_3=(2a^2+3a)+i(2a+1) \qquad \Lambda_5=-2(a+1)-2i
$$
is diffeomorphic to $\s^3\times\s^3$ and it is a principal fibration in elliptic curves over the
Hirzebruch surface $\mathbb F_a$.
\end{example}

\medskip

The existence of such manifold was noticed by H. Maeda dans \cite{Ma}.

\subsection{Examples}
 In this subection, we shall use the following facts (which are proven in \cite{LdMV}). Let $\Lam=(\lambda_1,\hdots,\lambda_5)$ be an admissible
configuration with $m=1$ and $n=5$. Then, the classification of $N_{\Lam}$ {\it up to diffeomorphism} is completely determined by $k$, the
number of indispensable points. We have

\[
k=0 \iff N \text{ is the quotient of }\gato(5)(\s^3\times\s^4)\text{ by a non-trivial action of }\s^1
\]
\[ 
k=1 \iff N \text{ is diffeomorphic to } \s^3\times \s^3
\]
\[
k=2 \iff N \text{ is diffeomorphic to }(\s^5\times\s^1)
\]
where $\gato(5)\s^3\times \s^4$ denotes the connected sum of five copies of $\s^3\times \s^4$. In all of the following examples, we shall give
the fans in $\Bbb R^2$ with canonical basis $(e_1,e_2)$ and lattice $\Z^2$, or in $\R$ with canonical basis $e_1$ and lattice 
$\Z$. 

\medskip
\begin{remark}\noi \textcolor{blue}{In the examples that follow we also use the very technical fact, 
proven in \cite{MV}, that given the fan $\Delta$ of a toric variety with quotient singularities one can recover the configuration 
$\Lam$ satisfying condition ({\bf K}) of definition \ref{System(K)}. 
In particular one can recover from the fan the number of equations $m$ and the dimension $n$ to have an admissible action of $C^m$ on $C^n$.}.
\end{remark}

\begin{example}

Consider the complete fan $\Delta$ 
generated
by
$$
w_1=e_1 \qquad w_2=e_2 \qquad w_3=-e_1-e_2 
$$
of the complex projective space $\Bbb CP^2$. There is a unique class of K\"ahler classes 
(in the sense of Theorem G in \cite{MV}), which is 
that of the (Chern class) of the
anti-canonical divisor. Up to scaling
and up to translation, the polytope $Q$ is defined as
$$
Q=\{u\in\R^2\quad\vert\quad \langle w_1,u\rangle \geq -1,\ \langle w_2,u\rangle \geq -1,\ \langle w_3,u\rangle \geq -1\}\ .
$$

By the methods of \cite{MV} (\ref{realization-Lambda}) we can recover the configuration:
$$
\lambda_1=\lambda_2=\lambda_3=1 \qquad \lambda_4=-3+i \qquad \lambda_5=-i\ .
$$
It is easy to check that this configuration has two indispensable points ($\lambda_4$ and $\lambda_5$) and thus the manifold $N$
is diffeomorphic to $\s^5\times\s^1$. We obtain finally the well-known (holomorphic) fibration $\s^5\times\s^1\to \CP^2$ 
and the pre-symplectic form
of $N$ scaled by $5$ projects onto the (Chern class) of the anti-canonical divisor.

Notice that we may easily compute the modulus of the fiber in this case. From Corollary B in \cite{MV}, we obtain the lattice and we obtain: 
$$
L_0^*=\dfrac{1}{5}\text{Vect}_{\Z} (-4+i,1+i)
$$
and this modulus is 
equal to
$$
\dfrac{-4+i}{1+i}=\dfrac{-3}{2}+\dfrac{5}{2}i\ .
$$
It is known \cite{CE} (see also \cite{LN}) that, for any choice of a modulus $\tau$, there exists a complex structure on $\s^5\times \s^1$ such that it fibers 
in elliptic curves of modulus $\tau$
over the complex projective space. 

Fix $\tau=\alpha+i\beta$ with $\beta > 0$. A straightforward computation shows that the admissible configuration
\[
\begin{pmatrix}
    \Re\lambda'_i  \\
  \\
    
   \Im \lambda'_i    
  \end{pmatrix}
\bigskip=
\begin{pmatrix}
-\beta &\beta \\
1+\alpha &4-\alpha
\end{pmatrix}.
\begin{pmatrix}
\Re\lambda_i\\ 
\Im \lambda_i
\end{pmatrix}
\eqno 1\leq i\leq 5
\]

\noi determines a complex threefold diffeomorphic to $\s^5\times\s^1$ which fibers over $\CP^2$ with fiber an elliptic curve of 
modulus $\tau$.
\end{example}  

\begin{example}
Consider the complete fan $\Delta$ generated by
$$
w_1=e_1 \qquad w_2=e_2 \qquad w_3=-e_1 \qquad w_4=-e_2
$$
of $\Bbb CP^1\times\Bbb CP^1$. The K\"ahler classes are
$$
D_{\alpha,\beta}=\alpha(D_1+D_3)+\beta(D_2+D_4)
\eqno \alpha>0,\ \beta>0
$$
corresponding to the rectangles
$$
Q_{\alpha,\beta}=\{(u_1,u_2)\in\Bbb R^2 \quad\vert\quad -\alpha\leq u_1\leq \alpha,\ -\beta\leq u_2\leq \beta\}\ .
$$
We get the corresponding configuration:		
$$
\lambda_1=\lambda_3=1\qquad \lambda_2=\lambda_4=i \qquad \lambda_5=-2\alpha-2i\beta .
$$

\noi The corresponding manifold $N_{\alpha ,\beta}$ is diffeomorphic to $\s^3\times \s^3$ and we find the Calabi-Eckmann fibration 
$\s^3\times \s^3\to \CP^1\times \CP^1$. The pre-symplectic form of $N_{\alpha ,\beta}$ projects onto a representative of
the class of
$D_{\alpha,\beta}$ (up to scaling). 

Fix $\alpha$ and $\beta$. As in Example 5, for every choice of $\tau\in\C$ with 
Im $\tau>0$, there exists a matrix $A$ of $\rm{GL}_2(\R)$ such that the product of the previous configuration by $A$ determines
a LVM manifold diffeomorphic to $\s^3\times \s^3$ which fibers over the product of projective lines with an elliptic curve of modulus 
$\tau$ as fiber (compare with \cite{CE}).

Alternatively, we may start from 
$$
\lambda_1=\lambda_3=1\qquad \lambda_2=\lambda_4=i \qquad \lambda_5=-1-i .
$$
which is an admissible configuration such that the class of $\tilde\omega$ is $D_{1,1}$ (up to scaling) and perform a translation on $(\lambda_1,\hdots,\lambda_5)$
to have another K\"ahler ray associated to $\tilde\omega$ (see Remarks 4.11 and 4.12 in \cite{MV}).

More precisely, assume that $\alpha <1$ and $\beta <1$ and let 
$$
b=\dfrac{1-2\alpha}{2\alpha+2\beta+1}+i\dfrac{1-2\beta}{2\alpha+2\beta+1}
$$
then the class of $\tilde\omega$ of $N((\lambda_1,\hdots,\lambda_5)+b)$ is $D_{\alpha,\beta}$.
\end{example}

\begin{example}
Consider the fan of $\CP^1$:
$$
w_1=e_1 \qquad w_2=-e_1
$$

There is a unique K\"ahler ray, that of $D=D_1+D_2$. We choose $p_1=p$ and $p_2=q$ for $p$ and $q$ strictly positive integers, that is we want to recover all
possible codimension one equivariant orbifold singularities on $\CP^1$. We take $n=4$ and $m=1$ and choose

One can show that that configuration is:

$$
\lambda_1=\dfrac{-3}{2p} +i\dfrac{1-2q}{2p} \qquad \lambda_2=\dfrac{-1}{2q}-i\dfrac{2p+1}{2q} \qquad \lambda_3=1 \qquad \lambda_4=i\ .
$$

It is easy to check that the corresponding manifold $N_{p,q}$ is diffeomorphic to $\s^3\times \s^1$ . It is the total space
of a generalized Calabi-Eckmann fibration over $\CP^1$ with at most two singular points of order $p$ and $q$.

Notice that if $p$ and $q$ are not coprime, then the orbifold structure cannot be obtained as a weighted projective space, i.e. cannot be obtained as quotient of 
$\C^2\setminus \{(0,0)\}$ by an
algebraic action of $\C^*$.
\end{example}

\begin{example}
Consider the complete fan $\Delta$ 
generated
by
$$
w_1=e_1, \qquad w_2=e_2, \qquad w_3=-e_1, \qquad w_4=-e_1-e_2, \qquad w_5=-e_2
$$
of the toric del Pezzo surface $X$ obtained as the equivariant blowing-up of $\CP^2$ in two
points. Define $Q$ as the pentagon associated to the anti-canonical divisor of $X$ (which is ample for $X$ is del Pezzo)
$$
Q=\{u\in\Bbb R^2 \quad\vert\quad \langle w_1,u\rangle \geq -1,\hdots,\langle w_5,u\rangle \geq
-1\}\ .
$$
We want to construct a generalized Calabi-Eckmann fibration with elliptic curves as fibers. This implies that we must take $m=1$ and $n=5$, so we cannot add any 
indispensable point. As the sum of the $w_i$ is not zero, we cannot take at the same time $p_1=\hdots=p_5=1$. In other words, there does not exist any {\it non singular}
generalized Calabi-Eckmann fibration over $X$ with elliptic curves as fibers. However, if we allow exceptional fibers, the construction is possible keeping $m=1$. For
example, take
$$
p_1=2 \qquad p_2=2 \qquad p_3=1 \qquad p_4=1 \qquad p_5=1\ 
.
$$
This gives

\[
v_1=2e_1, \quad v_2=2e_2, \quad v_3=-e_1, \quad v_4=-e_1-e_2,\quad v_5=-e_2,  \\
\quad\epsilon_1=\epsilon_2=\dfrac{2}{7}, \qquad \epsilon_3=\epsilon_4=\epsilon_5=\dfrac{1}{7}
\]

Notice that $\mathcal L$ is $\Z^2$.
Taking a linear Gale transform of this, we obtain
$$
\lambda_1=1, \qquad \lambda_2=i, \qquad \lambda_3=-2-4i, \qquad \lambda_4=4+4i \qquad
\lambda_5=-4-2i\ .
$$

This gives a fibration in elliptic curves $N\to X$ where $N$ is the quotient of 
the
differentiable manifold $\gato(5) \s^3\times \s^4$ by a non trivial action of 
$\s^1$. The orbifold structure on $X$ has two codimension one singular sets of index $2$ and the form $7\omega$ projects onto a representant of the Chern class of $D$.
\end{example}

\begin{example}
We consider the same toric variety $X$ and the same polytope $Q$ as in Example 8, but this time we want $p_1=\hdots=p_5=1$, i.e. we want a non singular fibration.
We are thus forced to increase $m$ by one and take $m=2$ and $n=7$, which gives us two additional indispensable points.
We take

\[
v_1=e_1, \, v_2=e_2, \, v_3=-e_1, \, v_4=-e_1-e_2, 
\]
\[
v_5=-e_2 ,\, v_6=-v_1-v_2-v_3-v_4-v_5=e_1+e_2 ,\quad v_7=0
\]

and 
$$
\epsilon_1=\hdots=\epsilon_5=\epsilon_7=\dfrac{1}{9},\qquad \epsilon_6=\dfrac{1}{3}
$$ 
Notice that $\epsilon_6$ is chosen so that $v_6/\epsilon_6$ lies in the interior of ${\mathcal H}(v_1/\epsilon_1,\hdots, v_5/\epsilon_5)$.
Making all the computations, we find that
$$
(\Lambda_1,\hdots,\Lambda_7)=\begin{pmatrix}
1 &i &0 &0 &-1+i &-1 &3-2i \\
0 &0 &1 &i &1 &1+i &-5-4i
\end{pmatrix}
$$
defines a LVM manifold diffeomorphic to $(\gato(5) \s^3\times \s^4)\times \s^1$ 
(by application of Theorem 12 of \cite{Me}) which is the total space of a non singular principal
holomorphic fibration over $X$ with complex tori of dimension $2$ as fibers. The form $9\omega$ projects onto 
the anti-canonical divisor of $X$.
\end{example}

\begin{example}
Let $a\in{\mathbb N}$ and consider the complete fan $\Delta$ 
generated
by
$$
w_1=e_1 \qquad w_2=e_2 \qquad w_3=-e_2 \qquad w_4=-e_1+ae_2 
$$
of the Hirzebruch surface $F_a$.

Let $D= D_1+D_2+D_3+(a+1)D_4$. The divisor $D$ is ample (see \cite{Fu}, p.70). We take $v_i=w_i$ for $1\leq i \leq 4$ and add the vertex $v_5=-v_1-\hdots-v_4=-ae_2$. We have
$m=1$ and $n=5$. We choose
$$
\epsilon_1=\dfrac{1}{2a+5}\qquad \epsilon_2=\dfrac{1}{2a+5} \qquad \epsilon_3=\dfrac{1}{2a+5} \qquad \epsilon_4=\dfrac{a+1}{2a+5}\qquad \epsilon_5=\dfrac{a+1}{2a+5}
$$
Taking a linear Gale transform of this, we obtain
$$
\lambda_1=\lambda_4=1 \qquad \lambda_2=i \qquad \lambda_3=(2a^2+3a)+i(2a+1) \qquad \lambda_5=-2(a+1)-2i
$$
with only one indispensable point $\lambda_5$. We thus have the following proposition.
\end{example}

\begin{proposition}
Let $a\in\mathbb N$. Consider the admissible configuration
$$
\lambda_1=\lambda_4=1 \qquad \lambda_2=i \qquad \lambda_3=(2a^2+3a)+i(2a+1) \qquad \lambda_5=-2(a+1)-2i
$$
\noi Then, the corresponding LVM manifold $N_a$ is diffeomorphic to
$\s^3\times \s^3$ and is a principal fiber bundle in elliptic curves over the Hirzebruch
surface $F_a$ (where $F_0$ is $\Bbb CP^1\times\CP^1$). The scaling of the canonical Euler form of the bundle $M_1\to N_a$ by $2a+5$ projects onto a representant of
the Chern class of the ample divisor $D= D_1+D_2+D_3+(a+1)D_4$ on $F_a$.
\end{proposition}

\begin{remark}
The preceding example shows that, for special complex structures of Calabi-Eckmann type on $\s^3\times\s^3$, there
exists holomorphic principal actions of an elliptic curve whose quotient may be topologically different (since the Hirzebruch
surfaces $F_{2n}$ are all diffeomorphic to $\s^2\times\s^2$ whereas the Hirzebruch surfaces $F_{2n+1}$ are all diffeomorphic
to the non-trivial $\s^2$-bundle over $\s^2$, and these two manifolds have different intersection form).
\end{remark}

\section{From polytopes to quadrics}

This section and the following rely on the paper by T. Panov \cite{PanovKAIST} and we also recommend \cite{BP} for this section. Let $\R^n$ be given the standard inner product $\langle\cdot,\cdot\rangle$ and 
consider convex polyhedrons $P$ defined as intersections of  $m$ closed half-spaces:
$$
\Pi_{(a_i,b_i)}=\{x\in\R^n \,\,\,|\,\,\, \langle a_i, x\rangle+b_i\geq0\}, \quad for \quad i = 1, \dots ,m
$$	
with $a_i\in\R^n$,  $b_i\in\R$. Assume that the hyperplanes defined by the equations 
 $\langle a_i, x\rangle+b_i=0$ are in general position, i. e. at least $n$ of them meet at a single point.
Assume further that $dim\, P = n$ and $P$ is bounded (which implies that $m > n$).
Then $P$ is an $n$-dimensional  compact simple polytope. Set
$$
F_i =\{x \in P : \langle a_i, x\rangle+b_i=0\}\quad\quad (F \,\,\rm{for\,\,\, facet}).
$$
Since the hyperplanes are in general position $F_i$ is either empty or a facet of $P$. If it is empty the linear equation is redundant and we can remove the corresponding inequality without changing $P$.

Let $A_P$ be the $m\times n$ matrix of row vectors $a_i$, and $b_P$ be the column $m$-vector of
scalars $b_i\in\R$ ($i\in\{1,\cdots,m\}$). Then we can write:
$$
P = \{x \in \R^n\, : A_P x + b_P \geq 0\}
$$
and consider the affine map
$$
i_P : \R^n \to \R^m, 
$$
$$ 
i_P (x) = A_P x + b_P.
$$
It embeds $P$ into the first orthant       
$$
\R^m_{\geq0}=\{(y_1,\cdots,y_m) \in \R^m\,\,\,\,|\,\,\,\, y_i > 0, \quad i\in\{1,\hdots,m\}\}.
$$
We identify $\C^m$ (as a real vector space) with $\R^{2m}$ as usual using the map

\noi $ z = (z_1, \hdots, z_m)\mapsto (x_1, y_1, \hdots , x_m, y_m)$,
where $z_k = x_k + iy_k$ for $k = 1, \hdots ,m.$

\noi Consider the following commutative diagram where $\mathcal Z_P$ its obtained by pull-back and 

\noi $\mu:\C^m\to\R^m_{\geq0}$ is given by $\mu(z_1,\hdots,z_m)=(|z_1|,\hdots,|z_m|)$:
\begin{center}
\[
\begin{tikzcd}
\mathcal Z_P\arrow{r}{i_P^*}\arrow{d}[swap]{\pi} & \C^m\arrow{d}{\mu} \\
P\arrow{r}[swap]{i_P} & \R^m_{\geq0} \\
\end{tikzcd}
\]
\end{center}

The map $\mu$ may be thought of as the
quotient map for the coordinatewise action of the standard torus
$$
\mathbb T^m = \{(z_1,\hdots,z_m) \in \C^m: |z_i| = 1 \,\,for\,\, 1\leq{i}\leq{m}\}
$$ on $\C^m$. 

\medskip
    
\noi Therefore, $\mathbb T^m$ acts on $\mathcal Z_P$ with quotient $P$, and $i_{P}^*$ is a $\mathbb T^m$-equivariant
embedding.

\noi The image of $\R^n$ under $i_P$ is an $n$-dimensional affine plane in $\R^m$, which can be
written as
\[
i_P(\R^n) = \{y \in \R^m\,: y = A_P(x) + b_P \, \,for \,\,some\,\, x \in \R^n\}= 
\]

\[
=\{y \in\R^m \, : \,\Gamma y = \Gamma b_P\},
\]
where $\Gamma = ((\gamma_{jk}))$ is an $(m-n)\times m$ matrix whose rows form a basis of linear relations
between the vectors $a_i$. That is, $\Gamma$ is of full rank and satisfies the identity $\Gamma A_P = 0$.

Then we obtain that $\mathcal Z_P$ embeds into $\C^m$ as the set of common zeros of
$m-n$ real quadratic equations:

\[
i_P^*(\mathcal Z_P ) =
\left\{z \in \C^m\,\quad|\quad
\overset{m}{\underset{k=1}\sum}
\gamma_{jk}|z_k|^2 =
\overset{m}{\underset{k=1}\sum}
\gamma_{jk}b_k,\, for\, 1\leq{j}\leq{m-n}\right\} \tag{\bf Quadratic $\star$}
\]
The following properties of $\mathcal Z_P$ easily follow from its construction.
\begin{enumerate}
\item Given a point $z\in\mathcal Z_P$ , the $i^{th}$ coordinate of $i_P^*(z)\in \C^m$
vanishes if and only if $z$ projects onto a point $x\in P$ such that $x\in F_i$ for some facet $F_i$.

\item Adding a redundant inequality to  results in multiplying $\mathcal Z_P$ by a circle.

\item  $\mathcal Z_P$ is a smooth manifold of dimension $m + n$.
The embedding $i_P^* : \mathcal Z_P\to \C^m$ has $\mathbb T^m$-equivariantly trivial normal bundle.
\end{enumerate}

\section{From quadrics to polytopes (Associated Polytope of LVM manifolds)}

Let $N$ and
$$
M_{1}=\{z\in{\mathbb C}^n\quad\vert\quad \sum_{i=1}^{n}\Lambda_{i}\vert z_{i}\vert^{2}=0,\ \sum_{i=1}^{n}\vert z_{i}\vert ^{2}=1\}
$$
be as before in definition  \ref{moment_angle}.

\noi Let us remark that the standard action of the torus $({\mathbb S}^{1})^{n}$ on ${\mathbb C}^{n}$

\[
((\exp i\theta_1,\cdots, \exp i\theta_n) ,z)\longmapsto (\exp i\theta_{1}\cdot z_{1},\hdots,\exp i\theta_{n}\cdot z_{n})\quad
\tag{\textcolor{blue}{$\star\star$}}   \label{torus-action}
\]

leaves $M_{1}$ invariant. The quotient of $M_{1}$ by this action can be identified, via the diffeomorphism 
$r\in\mathbb{R}_{>0}^{+}\to r^{2}\in\mathbb{R}_{>0}^+$, to
\[\label{associated-polytope}
K=\{r\in (\mathbb R^{+})^{n}\quad\vert\quad \sum_{i=1}^{n}r_{i}\Lambda_{i}=0,\ \sum_{i=1}^{n}r_{i}=1\} 
\tag{\textcolor{blue}{P}}
\]

\begin{lemma}
The quotient $K$ is a convex polytope of dimension $n-2m-1$ with $n-k$ facets.
\end{lemma}

\begin{proof}
By definition $K$ is the intersection of the space $A$ of solutions of an affine system   with the closed sets  $r_{i}\geq 0$. Each one of these closed sets 
defines an affine half-space $A\cap \{r_{i}\geq 0\}$ in the affine space $A$. In other words, $K$ is the intersection of a finite number of affine half-spaces. 
Since this intersection is bounded (since $M_{1}$ is compact), one obtains indeed a convex polytope. The weak hyperbolicity  condition implies that the  
affine system  that defines $K$ is of maximal rank. Hence, $K$ is of dimension $n-2m-1$.

Let us consider in more detail the definition of $K$. The points $r\in K$ verifying $r_{i}>0$ for all  $i$ are the points which belong to the interior of the convex polytope. They correspond to the points $z$ de $M_{1}$ which also belong to 
$(\C^{*})^{n}$, i.e. to the points of $M_{1}$ such that the orbit under the action (\ref{torus-action})
is isomorphic to  $(\s^{1})^{n}$. The points which belong to a hyperface are exactly the points $r$ of 
$K$ having all of its coordinates
{\it except one} equal to zero. They correspond to the points $z$ de $M_{1}$ which have a unique coordinate equal to zero, i.e. such that its orbit under the action (\ref{torus-action})
 is isomorphic to $(\s^{1})^{n-1}$. One obtains from the definition of $K$ that there exist points of $K$ having all coordinates different from zero
except the i$^{th}$ coordinate if and only $0$ belongs to the convex envelope of the configuration formed by the $\Lambda_{j}$ with $j$ different from $i$; hence
if and only if  $\Lambda_{i}$ is a point which can be eliminated keeping the conditions of Siegel and weak hyperbolicity. 
therefore one has $n-k$ hyperfaces. \end{proof}

 \begin{definition}\label{associated-polytope}
 One calls the convex polytope $K=K_{\Lam}$ corresponding to the admissible configuration $\Lam$ the 
 \emph{associated polytope}. The polytopes ${\mathcal H}(\Lam)$ and $K_{\Lam}$ are related by the Gale transform.
 \end{definition}
One central idea is that the topology of the manifolds $M_{1}$, and therefore of the manifolds $N$, is codified by the
combinatorial type of the polytope $K$. 
To make this idea more precise, it is  interesting  to push to the end the reasoning  involved  in the proof of the preceding lemma. One had seen that  
$$
K_i=K\cap\{r_{i}=0,\ r_{j}>0 \text{ for } \, j\not = i\} 
$$
is nonempty, and therefore is a hyperface de $K$, if and only if 
$$
0\in\mathcal H((\Lambda_{j})_{j\not =i}).
$$
Analogously,  given $I$ a subset of $\{1,\hdots,n\}$, the set
$$
K_I=K\cap\{r_{i}=0 \text{ for}\,  i\in I,\ r_{j}>0\text{ for}  \, j\not \in I\}
$$
is nonempty, and therefore it is a facet of  $K$ of codimension  equal the cardinality of $I$, if and only if 
$$
0\in\mathcal H((\Lambda_{j})_{j\not\in I})
$$

One has therefore stablished a very important correspondence between two convex polytopes: 
the polytope $K$ on one hand and the convex hull of the $\Lambda_{i}$'s on the other hand. 

This correspondence allows us to
 to prove the following result:

\begin{remark}
It follows from {\rm \cite[Lemma 1.1]{MV}} that 
\begin{equation*}
\text{\rm rank}_{\mathbb C}
\begin{pmatrix}
\Lambda_1 &\hdots &\Lambda_n\cr
1 &\hdots &1
\end{pmatrix}
=m+1.
\end{equation*}
Hence, up to a permutation, condition \eqref{firstmcondition} is always fulfilled.
\end{remark}
\begin{definition} We say that $N_{\Lam}$ and $N_{\Lam'}$ are {\it $G$-biholomorphic} if they are $(G(\Lam),G(\Lam')$-equivariantly biholomorphic.
\end{definition}

Recall that by definition \ref{LVM-definition} the manifold $N_{\Lam}$ embeds in $\mathbb P^{n-1}$ as the $C^\infty$ submanifold
\begin{equation}
\label{Nsmooth}
N =\{[z]\in\mathbb P^{n-1}\quad\vert\quad \sum_{i=1}^n\Lam\vert z_i\vert ^2=0\}.
\end{equation}

\noi It is crucial to notice that this embedding is not arbitrary but has a clear geometric meaning. Indeed, it is proven in  that action \eqref{actionLVM} induces a foliation of $\mathcal{S}_{\Lam}$; that every leaf admits a unique point closest to the origin (for the euclidean metric); and finally that $N $ is the projectivization of the set of all these minima.
This is a sort of non-algebraic Kempf-Ness Theorem. So we may say that this embedding is canonical.

\noi The maximal compact subgroup $(\s^1)^n\subset (\mathbb C^*)^n$ acts on $\mathcal{S}_{\Lam}$, and thus on $N_{\Lam}$. This action is clear on the smooth model \eqref{Nsmooth}. Notice that it reduces to a $(\s^1)^{n-1}$ since we projectivized everything.
\vspace{5pt}\\
The quotient of $N_{\Lam}$ by this action is easily seen to be a simple convex polytope of dimension $n-2m-1$, cf.  Up to scaling, it is canonically identified to 
\begin{equation}
\label{KLambda}
K_{\Lam}:=\{r\in(\mathbb R^+)^n\quad\vert\quad \sum_{i=1}^n\Lambda r_i=0,\ \sum_{i=1}^nr_i=1\}.
\end{equation}
It is important to have a description of $K_{\Lam}$ as a convex polytope in $\mathbb R^{n-2m-1}$. This can be done as follows. Take a Gale diagram of $\Lam$, that is a basis of solutions $(v_1,\hdots, v_n)$ over $\mathbb R$ of the system \ref{System(S)}:

\[
\left\{
\begin{aligned}
\sum_{i=1}^n\Lambda_ix_i&=0\cr
\sum_{i=1}^nx_i&=0
\end{aligned}
\right.\tag{\textcolor{red}{S}}
\]

Take also a point $\epsilon$ in $K_{\Lam}$. This gives a presentation of $K_{\Lam}$ as
\begin{equation}
\label{KLambdaproj}
\{x\in\mathbb R^{n-2m-1}\quad\vert\quad \langle x,v_i\rangle\geq -\epsilon_i\text{ for }i=1,\hdots, n\}
\end{equation}
This presentation is not unique. Indeed, taking into account that $K_{\Lam}$ is unique only up to scaling, we have
\begin{lemma}
\label{Kupto}
The projection \eqref{KLambdaproj} is unique up to action of the affine group of $\mathbb R^{n-2m-1}$.
\end{lemma}
On the combinatorial side, $K_{\Lam}$ has the following property. A point $r\in K_{\Lam}$ is a vertex if and only if the set $I$ of indices $i$ for which $r_i$ is zero is maximal, that is has $n-2m-1$ elements. Moreover, we have
\begin{equation}
\label{Kcombinatoire}
r\text{ is a vertex }\iff \mathcal{S}_{\Lam}\cap \{z_i=0\text{ for }i\in I\}\not =\emptyset\iff 0\in\mathcal H(\Lam_{I^c})
\end{equation}
for $I^c$ the complementary subset to $I$ in $\{1,\hdots,n\}$.  This gives a numbering of the faces of $K_{\Lam}$ by the corresponding set of indices of zero coordinates.

More precisely, we have
\begin{equation}
\label{Knumbering}
\begin{aligned}
&J\subset\{1,\hdots,n\}\text{ is a face of codimension Card }J\cr
&\iff \mathcal{S}_{\Lam}\cap \{z_i=0\text{ for }i\in J\}\not =\emptyset\iff 0\in\mathcal H(\Lam_{J^c})
\end{aligned}
\end{equation}

In particular, $K_{\Lam}$ has $n-k$ facets. Observe moreover that the action \eqref{actionLVM} fixes 

\noi $\mathcal{S}_{\Lam}\cap \{z_i=0\text{ for }i\in J\}$, hence its quotient defines a submanifold $N_J$ of $N_{\Lam}$ of codimension $\text{Card J}$.

\section{Moment-angle manifolds}

\noi We will explain the link between the moment-angle manifolds in definition \ref{moment_angle} and the manifolds studied by
 V. Buchstaber et T. Panov in \cite{BP}. Let $P$ be a  simple convex polytope with the set $\mathcal F=\{F_1,\hdots, F_n\}$ of hyperfaces (\ie codimension one faces). 
\noi Let $T_i\simeq \s^1$ for $1\leq{i}\leq{n}$ and let
 $T_{\mathcal F}=T_1\times\cdots\times{T_n}\simeq(\s^1)^n=\T^n$ be the $n$ torus with its standard group structure.
For each hyperface 
$F_i$ associate $T_i$, the circle corresponding to the $i^{\rm{th}}$ coordinate of $\T^n$.

\noi If $G$ is a face of the polytope $P$ let 
\[
T_G=\prod_{F_i \supset G}T_i\subset T_{\mathcal F}
\]

\noi For each point $q\in P$, let $G(q)$  be the unique face of $P$ which contains $q$
 in its relative interior

\begin{definition} The \emph{moment-angle complex} $\mathcal Z_P$ associated to $P$ is defined as
$$
\mathcal Z_P=(T_{\mathcal F}\times P)/\sim
$$
where the equivalence relation is:  $(t_1,p)\sim (t_2,q)$ if and only if $p=q$ and $t_1t_2^{-1}\in T_{G(q)}$.

\end{definition}

The \emph{moment-angle complex} $\mathcal Z_P$ depends only upon the combinatorial type of $P$ and
it admits a natural continuous action of  the $n$-torus $T_{\mathcal F}$ having as quotient $P$. The fact that $P$
is simple implies that  $\mathcal Z_P$ is a topological manifold (see \cite{BP}  Lemma 6.2).

\noi Consider now the moment-angle manifolds $M_1(\Lam)$ defined by formula (\ref{moment_angle})
and let $K_{\Lam}$ the associated polytope (\ref{associated-polytope}). One has the natural projection
$\Pi:M_1(\Lam)\to{K_{\Lam}}$ The faces of codimension $q$ of
 $K(\Lam)$ correspond to the orbits of the of the points of $\mathcal V$ which have some precise $q$ coordinates fixed.
 In other words the orbits above the relative interior of a codimension $q$ face are isomorphic to
 $(\s^{1})^{n-q}$. En poussant un peu plus loin cette
description, on montre le lemme suivant.

\begin{lemma}  {\bf(\cite{BM} Lemma 0.15)}.
Let $N_{\Lam}$ be an LVM manifold without indispensable points. Let $K_{\Lam}$ be its associated polytope.
Then there exists an equivariant homeomorphism between $M_1(\Lam)$ and the moment-angle variety $Z_{K_{\Lam}}$.
\end{lemma}

\emph{Equivariant homeomorphism} means that the homeomorphism conjugates the action (\ref{torus-action}) of
on $M_1{(\Lam)}$ to the action of $T_{\mathcal F}$ on $\mathcal Z_P$.

Hence:

\begin{corollary}
Let $N_{\Lam}$ et $N_{\Lam'}$  two LVM manifolds without indispensable points. Then there exists an equivariant
homeomorphism of the associated moment-angle varieties 
$M_1{(\Lam)}$ and  $M_1{(\Lam')}$ if and only if the associated polytopes
$K_{\Lam}$ et $K_{\Lam'}$ are combinatorially equivalent.
More generally, there exists an equivariant homeomorphism between $M_1{(\Lam)}$ and  $M_1{(\Lam')}$ 
if and only $K_{\Lam}$ and $K_{\Lam'}$ are combinatorially equivalent the number of indispensable points 
 $k$ and $k'$, respectively, are equal.
\end{corollary}  

\begin{proof}
The combinatorial equivalence between $K_{\Lam}$ and $K_{\Lam'}$  implies the existence
of an equivariant homeomorphism between $\mathcal Z_{K_{\Lam}}$ and $\mathcal Z_{K_{\Lam'}}$, and hence 
by lemma 9 between  entre $M_1{(\Lam)}$ and $M_{1}{(\Lam')}$. 
The proof for any number of indispensable points follows from the first result and lemma 9.
\end{proof}
    
It is more delicate to have the same result up to \emph{equivariant diffeomorphism}, however one has the following theorem:

\begin{theorem} {\bf (\cite{BM} Theorem 4.1)}.
There is an equivalence between the following assertions: 

\medskip
\noindent (i) The manifolds  $M_1{(\Lam)}$ and $M_{1}{(\Lam')}$ are equivariantly diffeomorphic

\noindent (ii) The corresponding associated polytopes 
 $K_{\Lam}$ and $K_{\Lam'}$ are combinatorially equivalent and the number of indispensable points 
 $k$ and $k'$ are equal.
\end{theorem}

\section{Flips of simple polytopes and elementary surgeries on LVM manifolds}

The motivation of this section is to generalize the following result of Mac Gavran \cite{McG} adapted to our case.

\begin{theorem}\label{Mac Gavran} {\bf (Mac Gavran \cite{McG}).}
Let $\Lam$ be an admissible configuration. 
Suppose that the associated polytope $K_{\Lam}$ is a polygon with $p$ vertices.
Then the moment-angle manifold$M_1{(\Lam)}$ is diffeomorphic via an equivariant diffeomorphism to the connected
sum of products of spheres;

$$
(\#_{j=1}^{p-3} (jC^{j+1}_{p-2})\s^{2+j}\times\s^{p-j})\times (\s^1)^k
$$
\end{theorem}

There are many cases of configurations for higher-dimensional polytopes where the manifolds $M_1(\Lam)$ are
similar to those  of Mac Gavran \ie
the manifolds  are products of manifolds of the type:

\medskip

\noindent (i) Odd dimensional spheres.

\noindent (ii) Connected sums of products of spheres

\medskip

 F. Bosio and L. Meersseman \cite{BM} showed that some, but not all, moment-angle manifolds
$M_{\Lam}$ are connected
sums of products of spheres, and they conjectured that if the dual to the polytope
is neighborly, then the manifold is such a connected sum. This conjecture was proven by 
Samuel Gitler and Santiago L\'opez de Medrano in \cite{GLdM}.

Let us remember once more the results by S. L\'opez de Medrano (\cite{LdM1} et \cite{LdM2}) on the classification of manifolds $M_1(\Lam)$ when $m=1$ given above in subsection 
\ref{Connected sums1} given by Theorem \ref{polygon_connected_sum}

When $m=1$ the vectors
are vectors $\Lambda_{i}$ in $\C\simeq \R^{2}$ and S. L\'opez de Medrano shows that one can modify the configuration
$\Lam\in\C$ through a smooth homotopy $\Lam_t$ (just moving the vectors) that \emph{satisfies the admissibility conditions 
of Siegel and weak hyperbolicity for all $t\in[0,1]$} such that $\Lam_0=\Lam$ and $\Lam_1$ is a regular polygon with an odd number of vertices $2l+1$ and with multiplicities $n_{1}$, ..., $n_{2l+1}$. Thus, for instance, in 
figure 5, one can move from the pentagon at the left to the pentagon at the right to configurations with different multiplicities,
for instance configurations with 4 vectors with multiplicities
$n_{1}=n_{2}=1$ and $n_{3}=3$, then $3$ vectors of multiplicities $n_{1}=1$, $n_{2}=n_{3}=2$, and finally $5$ vectors of multiplicity $1$.  Ehresmann lemma implies, that all manifolds belonging to the homotopy are diffeomorphic.

\begin{figure}[h]  
\begin{center}
\includegraphics[height=3cm]{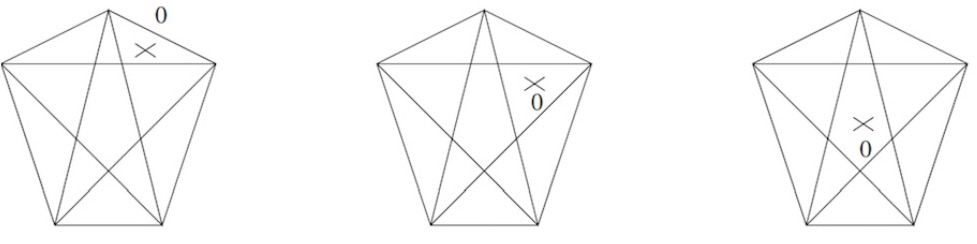}
\end{center}
\caption{\sl Chambers of a pentagon in $\C$ (the small cross is the origin in $\C$)} 
\label{Figure 3}
\end{figure}

With these notations we recall theorem \ref{polygon_connected_sum} which  was seen before:
\begin{theorem} \cite{LdM1}, \cite{LdM2}
Let $N$ be an LVM manifold $m=1$ then $M_{1}$ is diffeomorphic to
\medskip

\noindent (i) The product of spheres $\s^{2n_{1}-1}\times \s^{2n_{2}-1}\times\s^{2n_{3}-1}$ if $l=1$.

\noindent (ii) The connected sum
$$
\#_{i=1}^{2l+1} \s^{2d_{i}-1}\times\s^{2n-2d_{i}-2}
$$
if $l>1$. Where $d_i=n_{[i]}+\hdots+n_{[i+l-1]}$ and $[a]$ is the residue of the euclidean division of $a$ by $2l+1$.
\end{theorem}
If $m>1$ one has a higher dimensional polytope of $n$ elements in $\C^m$ ($n>2m$) and there is not a way to
have a canonical model. One could consider a homotopy that takes the configuration to one with minimal number of vertices,
but that is not enough to determine the polytope  which is the convex hull of the points in the configuration.
For this reason it is better to adapt the approach used by Mac Gavran in \cite{McG}.
 He considers simply connected manifolds 
of dimension $p+2$ which admit a smooth action of a torus $(\s^1)^p$ which satisfies certain conditions,
in particular one requires that the quotient under the action can be identified with a 2-dimensional convex polygon
$K$ with $p$ vertices. If we write  $M_1\simeq (\s^1)^k\times M_0$, where $\simeq$ means ``up to an equivariant diffeomorphism'', then as in lemma \ref{indispensable-lemma} one shows that the factor $M_0$ verifies the hypotheses
of Mac Gavran. The proof of Mac Gavran theorem is done by induction on the number $p$ of vertices of $K$.
If $p=3$ one has a triangle and we know that $M_1$ is $\s^5\times (\s^1)^k$, where $M_0$ is the sphere
$\s^5$. To go from a polygon with $p$  vertices to a polygon with $p+1$ vertices one can do the following ``surgery'':
remove an open neighborhood of a vertex and glue an interval. The reciprocal operation consists in collapsing to a 
point an edge. Now we recall that the faces  of the associated polytope corresponding to the admissible 
sub-configurations of $\Lam$ (\ie subsets of $\Lam$) determine equivariant subvarieties of $M_1$ or $M_0$ where the
quotient space identifies in a natural way with the given face. In other words to remove a neighborhood of a face means
to remove an invariant (under the action of the torus) tubular neighborhood of the subvariety associated to the face in question in $M_0$. The invariant subvarieties have trivial tubular neighborhoods (by the slice theorem). Since we know that
the subvarieties associated to a vertex is a torus and the subvarieties associated to an edge are the product of a torus with
$\s^3$, one sees that if $M_p$ denotes the manifold corresponding to a polygon 
$K$ with $p$ vertices, then to pass from $M_p$  a $M_{p+1}$ consists of applying an equivariant surgery

\[
M_{p+1}=(M_p\times \s^1)\setminus ((\s^1)^{p-2}\times \mathbb D^4\times \s^1)\cup ((\s^1)^{p-2}\times\s^3\times \mathbb D^2)
\]

\noindent Where $\mathbb D^s$ denotes the closed disk of dimension $s$. The work of Mac Gavran consists of understanding the meaning of these surgeries up to equivariant diffeomorphisms.
To generalize this approach to higher dimensional polytopes $K$ we need to generalize the notion of ``surgery''.
and understand what is the construction one has to perform on the moment angle manifold $M_1$ associated to $K$.
This is done using the following notion of cobordism between polytopes inspired by \cite{McM} et \cite{Ti}.

\begin{definition} Let $P$ and $Q$ be two simple convex polytopes of the same dimension $p$. One says that
$P$ and $Q$ are obtained from each other by an \emph{elementary cobordism} if there exists a simple convex polytope
$W$ of dimension $p+1$ such that:

\medskip
\noindent (i) $P$ and $Q$ are disjoint hyperfaces of $W$.

\noindent (ii) There exists a unique vertex $v$ of $W$ that does not belong neither to $P$ or $Q$.

\end{definition}

Let us recall that everything related to polytopes is up to combinatorial equivalence for instance figure 6 illustrates an
elementary cobordism between a square and a pentagon.

\medskip

\begin{figure}[h]  
\begin{center}
\includegraphics[height=5cm]{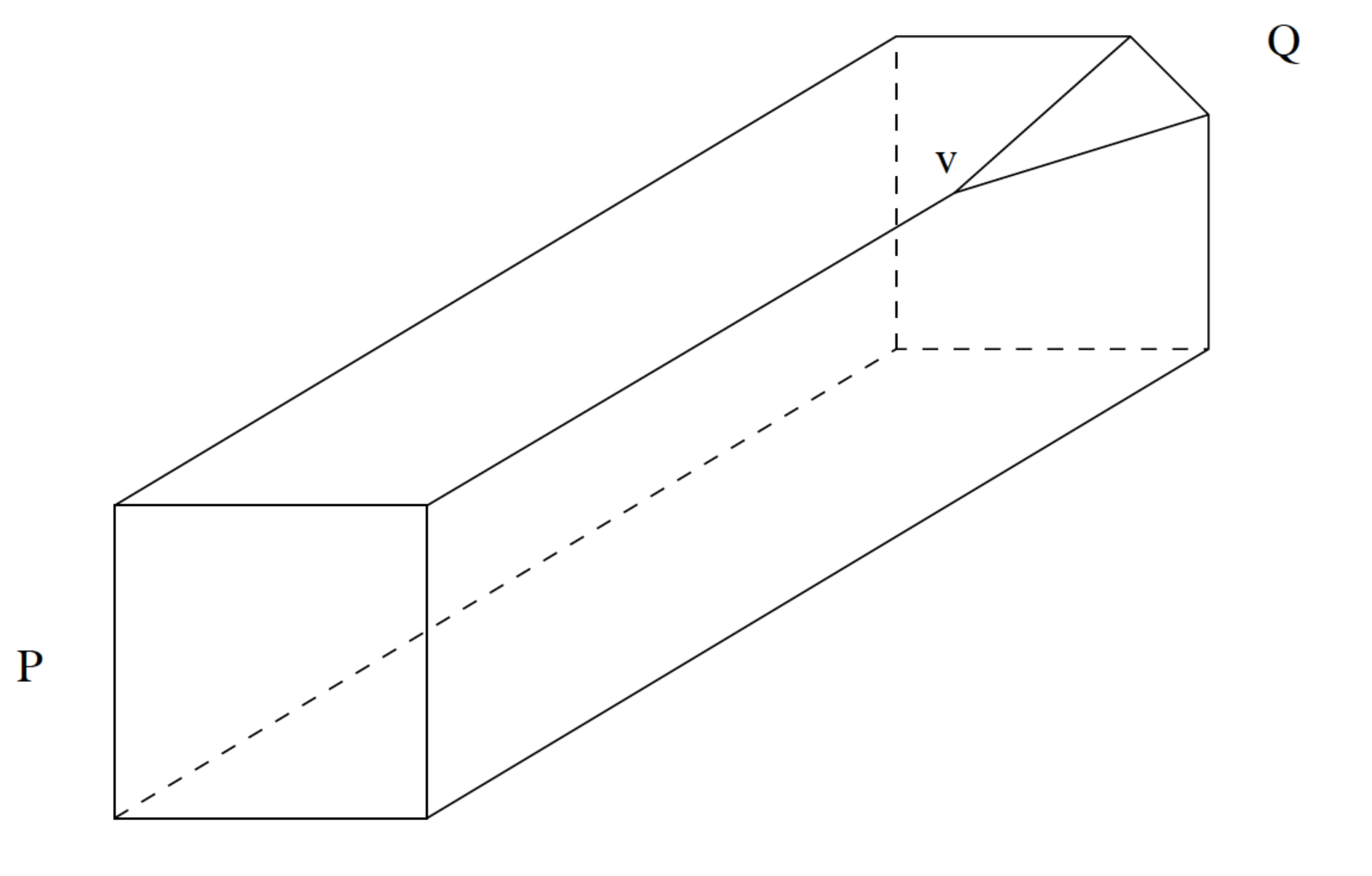}
\end{center}
\caption{\sl Elementary Cobordism between a square and a pentagon} 
\label{Figure 3}
\end{figure}


Given a vertex $v$, since $W$ is simple, there are exactly $q+1$ edges that have $v$ as an end point. Then,
hypothesis (ii) these edges have the second end point either in $P$ or in $Q$, then the \emph{type} of the elementary
cobordism is the pair $(a,b)$ where $a$  (respectively $b$) is the number of edges joining $v$ to $P$ 
(respectively $Q$). Of course  $a+b=q+1$,

\begin{definition}
One says that $Q$ is obtained from $P$ by a flip of type $(a,b)$ 
if there exists an elementary cobordism of type $(a,b)$ between $P$ and $Q$.
\end{definition}

\noi Figure 7 shows an example of type $(2,2)$.

\begin{figure}[h]  
\begin{center}
\includegraphics[height=5cm]{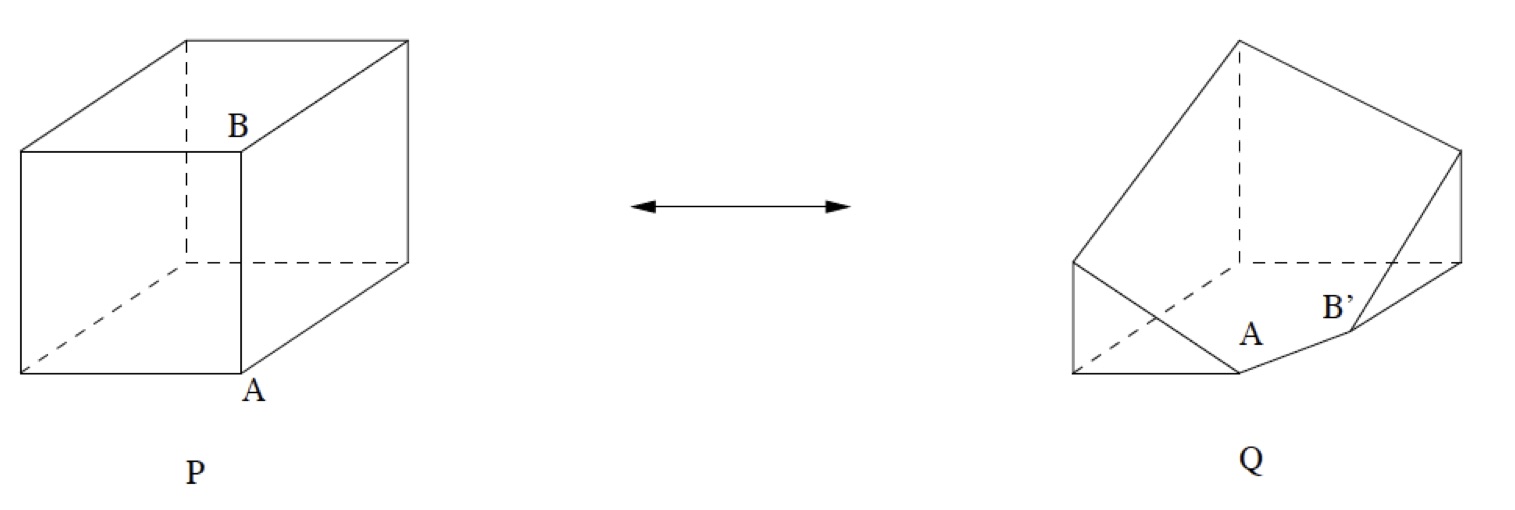}
\end{center}
\caption{\sl Flip de type $(2,2)$} 
\label{Figure 4}
\end{figure}
\noi Let us consider the elementary cobordism $W$ of dimension 4 between the 3-dimensional polyhedra $P$ and $Q$ and let us ``cut'' $W$ with 3-dimensional hyperplanes parallel to $P$. Starting from $P$ one sees that the edge  $[AB]$ is contracted as one moves the cuts up to the point when the edge
collapses to the vertex $v=A$ when the cut meets the vertex $A$. On the other hand if one makes cuts by hyperplanes parallel to $Q$ the edge $[AB]$ is contracted to $v=A$. In some sense $W$ is the {\it trace} of the cobordism. In other words
$Q$ is obtained $P$ by removing a neighborhood of the edge
 $[AB]$ and gluing the ``transverse'' edge $[AB']$.

\medskip

This description can be generalized for higher dimensional polytopes. A flip of type
 $(a,b)$ is obtained by removing a simplicial face of dimension $a$ and gluing the neighborhood of a simplicial face 
 of ``complementary'' dimension $b$. Since the simplicial faces of dimension $a$ correspond to products of a sphere
 of dimension $2a-1$ by a torus (\cite{BM} Proposition 3.6) an argument similar to that of Mac Gavran shows that
 if $K'$ is obtained from  $K$ (of dimension $q$) by a flip of type
$(a,b)$, then the manifold $(M_0)'$ is obtained from $M_0$ (de dimension $p$) by an elementary {\it surgery of type
$(a,b)$} then:
\[
(M_0)'=(M_0\times \s^1)\setminus ((\s^1)^{p-2b}\times \mathbb D^{2b}\times \s^1)\cup ((\s^1)^{p-2b}\times\s^{2b-1}\times \mathbb D^2)
\]
if $a=1$  
and
\[
(M_0)'=M_0\setminus ((\s^1)^{p-2b-2a+1}\times \mathbb D^{2b}\times \s^{2a-1})\cup 
((\s^1)^{p-2b-2a+1}\times\s^{2b-1}\times \mathbb D^{2a})
\]
if $a>1$.

\medskip
\noi The proof of this fact is very delicate and technical since we must prove the equivariance of the constructions.
All the details can be found in \cite{BM}.

\medskip

The essential difference with the case of polygons of Mac Gavran is that starting with an odd-dimensional sphere
as $M_0$ after a finite number of elementary surgeries one does not end up with a manifold of the type connected sum
of products of spheres or product of spheres. In fact in the next section one will describe the homology. However the previous
considerations prove again that if two moment-angle manifolds of type $M_1$ of dimension $p$ are combinatorially equivalent they are obtained from the sphere $\s^{2p-1}$ by the same sequence of elementary surgeries and therefore
they are equivariantly diffeomorphic.

\section{The homology of LVM manifolds}

Recall that from lemma \label{indispensable-lemma} we have that for a configuration $\Lam$ the associated moment angle manifold $M_1(\Lam)$ factorizes as $M_1(\Lam)=\simeq (\s^1)^k\times M_0(\Lam)$ where $k$ is the number of indispensable points and $M_0(\Lam)$ is 2-connected. Since $M_0$ is a moment-angle manifold one can use
the results of V. Buchstaber and T. Panov \cite{BP} to compute the homology and cohomology  of these manifolds

\begin{theorem} \label{homology-M_0}{\bf (\cite{BM}, Theorem 10.1.)}
Soit $N_{\Lam}$ be an LVM manifold and $M_0(\Lam)$ the 2-connected factor as in \ref{indispensable-lemma}. Let  $K$ 
be the associated polytope (quotient under the action of the torus). Let
 $K^*$ be its dual which is therefore a convex simplicial polytope. Let 
$\mathbf F$ be its set of vertices. Then the homology of $M_0(\Lam)$ with coefficients in $\Z$ is given by the formula:

\[
H_i(M_0(\Lam),\mathbb Z)=\bigoplus_{\mathcal I\subset\mathbf F}\tilde H_{i-|     \mathcal I|     -1}(K_{\mathcal I}^*,\mathbb Z)
\]
where $\tilde H_i$ denotes the reduced homology, $| \mathcal{I}|$ is the cardinality of $\mathcal{I}$
 and $K_{\mathcal I}^*$ is the maximal simplicial subcomplex of $K^*$ with vertices $\mathcal{I}$.
 \rm{(}We remark that $H_i(M_0(\Lam),\mathbb Z)=0\,\, \rm{if}\,\,i<0$\rm{)}.
\end{theorem}

\newpage

Let us explain the meaning of \emph{maximal simplicial subcomplex of $K^*$ with vertices $\mathcal{I}$}. 
Given a $q$-tuple $(i_1,\hdots, i_q)$ in $\mathcal I$ it is a face of the simplicial subcomplex $K^*_{\mathcal I}$ 
if and only if a $q$-face of the simplicial complex $K^*$. 
For instance in figure 8 ($K^{*}$ is the octahedral  which the dual of the cube  $K$) the subcomplex which corresponds to
the vertices $\{1,2,3,4\}$ is indicated in boldface. 

\medskip
\begin{figure}[h]  
\begin{center}
\includegraphics[height=5cm]{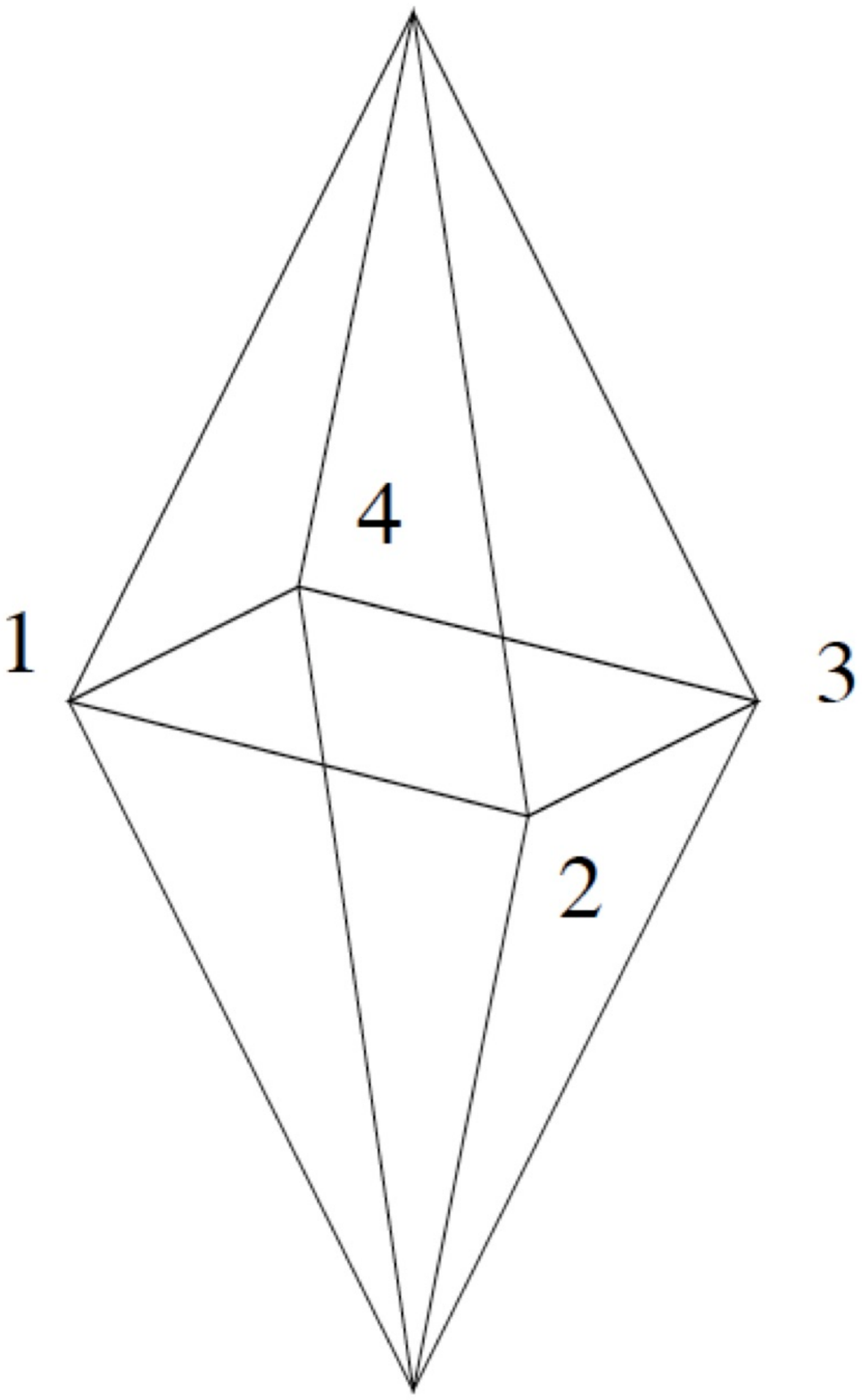}
\end{center}
\caption{\sl } 
\label{Figure 6}
\end{figure}

Let us consider now the question of the level of complexity of the homology of the manifolds $M_0(\Lam)$.
By theorem \ref{homology-M_0} the dual polytope $K^*$ can be an arbitrary simplicial complex and the question of complexity becomes to ask which simplicial complexes can be {\it maximal} subcomplexes of a simplical convex polytope.
We claim that any finite simplicial complex can be a {\it maximal} sub-complexes of a simplical convex polytope. In effect,
let $K_0$ be any finite simplicial complex, we can always embed $K_0$ in a simplex $\mathfrak{S}^d$ of dimension
$d$ equal to the number of vertices of $K_0$ minus one. In general is not embedded in a maximal subcomplex
For instance in figure 8,  $K_0$ is the one-dimensional complex with is a circuit of four edges with vertices in boldface. It can be embedded
in a tetrahedron $\mathfrak{S}^3$ as a circuit with 4 vertices but the maximal associated subcomplex is the tetrahedron itself so this embedded copy is not maximal but we can fix this by choosing a barycentric subdivision of the tetrahedron. In general
it is enough to make barycentric subdivision of all the faces that belong to the maximal simplicial complex generated by $K_0$ to obtain an embedding which is maximal. \emph{This is always possible} (see \cite{BM}).
 
\noi Therefore one has:
\begin{theorem} {\bf(\cite{BM}, Th\'eor\`em 14.1)}. Let $K_0$  be any finite simplicial complex. Then there exists a 2-connected LVM manifold $N$ such that its homology verifies: 

$$
H_{i+q+1}(N,\mathbb Z)=\tilde H_i(K_0,\mathbb Z)\oplus\hdots 
$$
for all $i$ between $0$ and the dimension of $K_0$.
\end{theorem}

Hence there exist an LVM manifold such that its homology has as a direct summand the homology of
a given simplicial complex, in particular its homology can be as complex as one wishes. For instance, given a
finite abelian group $\mathfrak{G}$ there exists a configuration $\Lam$ such that $N_{\Lam}$
 has as subgroup $\mathfrak{G}$ in its group of torsion.

\begin{remark} In {\bf (\cite{BM}, Theorem 10.1)} one finds a formula describing the ring structure via the cup product of the cohomology of these manifolds.
\end{remark}

\begin{remark} \textcolor{blue}{More details and results about the homology of moment-angle manifolds using the fact that they have in many cases an open-book structure will be found in section \ref{maiq}  subsection \ref{I 2}}
\end{remark}

\section{Wall-crossing}

Let us consider again figure 5 before now considered as figure 9 with the purpose of illustrating
the process of \emph{wall-crossing}

\begin{figure}[h]  
\begin{center}
\includegraphics[height=3cm]{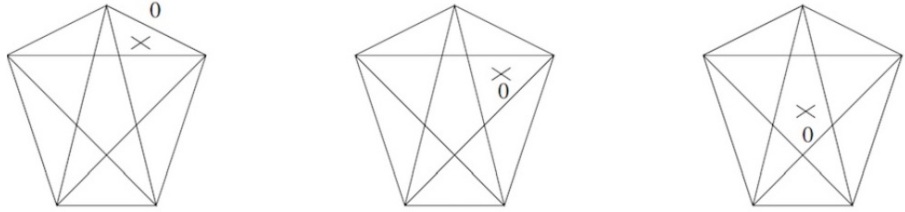}
\end{center}
\caption{\sl Wall-crossing from one chamber to another} 
\label{Figure 3}
\end{figure}

Consider in figure 9 different positions of the origin (marked as a cross) with respect to a configuration which is a regular pentagon and in the three positions the pentagon has been translated so that the origin is in different ``chambers'' 
bounded by the diagonals of the pentagon. 
We see that if the point marked with a cross moves from the figure on the left to the figure on the right then the figure in the left has two
indispensable points, in the second there is one indispensable point and in the figure at right there are not indispensable
points. The the manifolds from left to right are, respectively, $\s^{5}\times\s^{1}\times\s^{1}$, then 
$\s^{3}\times\s^{3}\times\s^{1}$ and finally $\#5(\s^{3}\times\s^{4})$. 

\medskip

As we mentioned before, these configurations are similar: one passes from one to the other translating $\Lam$  by a family
$\Lam_t$ or if one wishes
translating the origin. If we take the latter perspective and if we regard the translation of the origin as a homotopy along 
which $0$ moves, we see that there is a moment in which $0$ crosses at certain moment  a 
``wall'' $[\Lambda_{i}\Lambda_{j}]$ (in fact one crosses first a wall to go from left to the middle and the another to go from the middle to the right). The topology changes exactly after crossing the wall. In effect, if $0$ does not encounters the wall 
the configuration is admissible and $M_1(\Lam)$ does not change differentiably (again using Ehresmann lemma).
After crossing the mure the topology of $M_{1}$ changes drastically and after crossing the wall, by the same argument using Ehresmann lemma, nothing happens for the rest of the homotopy.

\medskip  
This situation generalizes to every dimension 

\medskip 
\noindent {\bf Question.}
\emph{How does the topology of $M_{1}$ changes  when we cross a wall?}

\medskip 

Let us see what happens in our example in figure 9 at the level of the associated polygons.
At  the left one has a triangle, in the middle a square and finally at right  a pentagon. In other words
one passes from the configuration at the left to the configuration in the middle by a surgery of type
$(1,2)$, then from the configuration in the middle to that in the right to a second surgery of type 
$(1,2)$. This solves completely this particular case.

\medskip

Some simple arguments of convex geometry allows us to see that everything is analogous in the general case.
When one crosses a wall in a configuration $(\Lambda_{1},\hdots,\Lambda_{n})$, 
 the wall is supported by  $2m$ vectors $\Lambda_{i}$. This wall separates the 
 convex envelope of $\Lam$ in two connected components one contains $0$ {\it before} the other contains $0$  {\it after}.
The $\Lambda_{j}$'s which do not belong to the wall divide in two parts:
 $a$ belong to the part that contains $0$ before crossing the wall and $b$ to the part that contains $0$
 after crossing the wall.  One of course has $a+b=n-2m$, namely $a+b$ is equal to the dimension of the 
 associated polytope plus one. We say that it is a wall-crossing of type $(a,b)$. 
 
 \noi With this notation we have:

\begin{theorem} {\bf (\cite{BM}, Theorem 5.4.)}
Let $\Lam$ et $\Lam'$ be two admissible configurations. 
 Suppose that $\Lam'$ is obtained from $\Lam$ through a wall crossing of type
$(a,b)$. Then

\medskip
\noindent (i) The polytope associated to $K'$ is obtained from $K$ by a flip of type $(a,b)$.

\noindent (ii) The manifold $M_{1}(\Lam')$ is obtained from $M_{1}(\Lam)$ by an elementary 
surgery of type $(a,b)$.

\end{theorem}

One can be more precise and characterize the face of the polytope where the ``flip'' happens
(or equivalently the subvariety along which we perform an equivariant surgery in function of the wall)

\section{LVMB manifolds}
\noi Recall that
\label{LVMB}
\begin{equation}
\label{SLambda}
\mathcal S:=\{z\in\mathbb C^n \quad\vert\quad 0\in\mathcal H(\Lambda_{I_z})\}
\end{equation}
where 
\begin{equation}
\label{Iz}
i\in I_z\iff z_i\not =0.
\end{equation}
Then $N_{\Lam}$ is the quotient of the projectivization $\mathbb P(\mathcal S)$ by the holomorphic action
\begin{equation}
\label{actionLVM}
(T,[z])\in\mathbb C^m\times \mathbb P(\mathcal S)\longmapsto \big [z_i\exp \langle \Lambda_i,T\rangle\big ]_{i=1,\hdots,n}
\end{equation}
where $\langle -,-\rangle$ denotes the inner product of $\mathbb C^m$, and not the hermitian one. It is a compact complex manifold of dimension $n-m-1$, which is either a $m$-dimensional compact complex torus (for $n=2m+1$) or a non k\"ahler manifold (for $n>2m+1$).

\subsection{Bosio manifolds} In \cite{Bosio}, Fr\'ed\'eric Bosio gave a generalization of the previous construction. His idea was to relax the weak hyperbolicity and Siegel conditions for $\Lam$ and to look for all the subsets $\mathcal{S}$ of $\C^n$ such that action \eqref{actionLVM} is free and proper. 

To be more precise, let $n\geq 2m+1$ and let $\Lam=(\Lambda_1,\hdots,\Lambda_n)$ be a configuration of $n$ vectors in $\C^m$. 
Let also $\mathcal E$ be a non-empty set of subsets of $\{1,\hdots,n\}$ of cardinal $2m+1$ and set
\begin{equation}
\label{SB}
\mathcal{S}=\{z\in\C^n\mid  I_z\supset E\text{\, for some\, }E\in\mathcal E\}
\end{equation}
Assume that 
\begin{enumerate}
\item[(i)] For all $E\in\mathcal E$, the affine hull of $(\Lambda_i)_{i\in E}$ is the whole $\C^m$.
\item[(ii)] For all couples $(E,E')\in\mathcal E\times\mathcal E$, the convex hulls $\mathcal H((\Lambda_i)_{i\in E})$ and $\mathcal H((\Lambda_i)_{i\in E'})$ have non-empty interior.
\item[(iii)] For all $E\in\mathcal E$ and for every $k\in\{1,\hdots,n\}$, there exists some $k'\in E$ such that $E\setminus\{k'\}\cup\{k'\}$ belongs to $\mathcal E$.
\end{enumerate}
Then, action \eqref{actionLVM} is free and proper \cite[Th\'eor\`eme 1.4]{Bosio}. We still denote it by $N_\Lambda$ althought it also depends on the choice of $\mathcal{S}$. As in the LVM case, it is a compact complex manifold of dimension $n-m-1$, which is either a $m$-dimensional compact complex torus (for $n=2m+1$) or a non K\"ahler manifold (for $n>2m+1$).

Assume now that $(\Lambda_1,\hdots,\Lambda_n)$ is an admissible configuration. Let
\begin{equation}
\label{ELVM}
\mathcal E=\{I\subset\{1,\hdots, n\}\mid 0\in\mathcal H((\Lambda_i)_{i\in I}\}
\end{equation}
Then \eqref{SB} and \eqref{SLambda} are equal, the previous three properties are satisfied and the LVMB manifold is exactly the corresponding LVM.

We say that $\Lambda_i$, or simply $i$, is an {\it indispensable point} if every point $z$ of $\mathcal{S}$ satisfies $z_i\not =0$. We denote by $k$ the number of indispensable points.

\subsubsection{The associated polytope of a LVM manifold}
\label{asspol}

In this section, $N_\Lambda$ is a LVM manifold.
Recall that the manifold $N_\Lambda$ embeds in $\mathbb P^{n-1}$ as the $C^\infty$ submanifold
\begin{equation}
\label{Nsmooth}
N =\{[z]\in\mathbb P^{n-1}\quad\vert\quad \sum_{i=1}^n\Lambda\vert z_i\vert ^2=0\}.
\end{equation}
It is crucial to notice that this embedding is not arbitrary but has a clear geometric meaning. Indeed, it is proven in \cite{MeThesis} that action \eqref{actionLVM} induces a foliation of $\mathcal{S}$; that every leaf admits a unique point closest to the origin (for the euclidean metric); and finally that \eqref{Nsmooth} is the projectivization of the set of all these minima\footnote{This is a sort of non-algebraic Kempf-Ness Theorem.}. So we may say that this embedding is canonical. 
\vspace{5pt}\\
The maximal compact subgroup $(\mathbb S^1)^n\subset (\mathbb C^*)^n$ acts on $\mathcal{S}$, and thus on $N_\Lambda$. This action is clear on the smooth model \eqref{Nsmooth}. Notice that it reduces to a $(\mathbb S^1)^{n-1}$ since we projectivized everything.
\vspace{5pt}\\
The quotient of $N_\Lambda$ by this action is easily seen to be a simple convex polytope of dimension $n-2m-1$, cf. \cite{MeThesis} and \cite{MV}. Up to scaling, it is canonically identified to 
\begin{equation}
\label{KLambda}
K_{\Lam}:=\{r\in(\mathbb R^+)^n\quad\vert\quad \sum_{i=1}^n\Lambda r_i=0,\ \sum_{i=1}^nr_i=1\}.
\end{equation}
It is important to have a description of $K_{\Lam}$ as a convex polytope in $\mathbb R^{n-2m-1}$. This can be done as follows. Take a Gale diagram of $\Lam$, that is a basis of solutions $(v_1,\hdots, v_n)$ over $\mathbb R$ of the system
\begin{equation}
\label{systemS}
\left\{
\begin{aligned}
\sum_{i=1}^n\Lambda_ix_i&=0\cr
\sum_{i=1}^nx_i&=0
\end{aligned}
\right.
\end{equation}
Take also a point $\epsilon$ in $K_{\Lam}$. This gives a presentation of $K_\Lambda$ as
\begin{equation}
\label{KLambdaproj}
\{x\in\mathbb R^{n-2m-1}\quad\vert\quad \langle x,v_i\rangle\geq -\epsilon_i\text{ for }i=1,\hdots, n\}
\end{equation}
This presentation is not unique. Indeed, taking into account that $K_\Lambda$ is unique only up to scaling, we have
\begin{lemma}
\label{Kupto}
The projection \eqref{KLambdaproj} is unique up to action of the affine group of $\mathbb R^{n-2m-1}$.
\end{lemma}
On the combinatorial side, $K_\Lambda$ has the following property. A point $r\in K_\Lambda$ is a vertex if and only if the set $I$ of indices $i$ for which $r_i$ is zero is maximal, that is has $n-2m-1$ elements. Moreover, we have
\begin{equation}
\label{Kcombinatoire}
r\text{ is a vertex }\iff \mathcal{S}\cap \{z_i=0\text{ for }i\in I\}\not =\emptyset\iff 0\in\mathcal H(\Lambda_{I^c})
\end{equation}
for $I^c$ the complementary subset to $I$ in $\{1,\hdots,n\}$.  This gives a numbering of the faces of $K_\Lambda$ by the corresponding set of indices of zero coordinates. To be more precise, we have
\begin{equation}
\label{Knumbering}
\begin{aligned}
&J\subset\{1,\hdots,n\}\text{ is a face of codimension Card }J\cr
&\iff \mathcal{S}\cap \{z_i=0\text{ for }i\in J\}\not =\emptyset\iff 0\in\mathcal H(\Lambda_{J^c})
\end{aligned}
\end{equation}
In particular, $K_\Lambda$ has $n-k$ facets. Observe moreover that the action \eqref{actionLVM} fixes $\mathcal{S}\cap \{z_i=0\text{ for }i\in J\}$, hence its quotient defines a submanifold $N_J$ of $N_\Lambda$ of codimension $\text{Card J}$.

Also, \eqref{Knumbering} implies that
\begin{equation}
\label{Sdelta}
\mathcal{S}=\{z\in\C^n\mid I_z^c\text{ is a face of } K_\Lambda\}
\end{equation}

\medskip

\section{\textcolor{blue}{Moment-angle manifolds and intersection of quadrics}}\label{maiq}

\begin{remark}\textcolor{brown}{This section is based on the papers
\cite{BLdMV, BLV, GLdM}  and it borrows freely a lot from them.
In order to be compatible with the notation in these papers, we use in this section sometimes different notations that the ones used in the previous
sections, for instance the moment-angle manifolds $M_1(\Lam)$ are called here $\ZLam$ and the corresponding
LVM manifolds $N_{\Lam}$ are denoted here $\vdu(\Lam)=\ZLam/\s^{^1}$}
\end{remark}

The topology of generic intersections of quadrics in $\R^n$ of the form:
\[
\sum_{i=1}^n\lam_ix_i^2=0,\quad\quad\sum_{i=1}^n x_i^2=1, \,\,\, \rm{where}\,\, \lam_i\in\R^k,\, i=1,\dots,n
\label{intersection_quadrics} 
\]

\noi appears naturally in many instances and has been studied for many years. If $k=2$ they are  diffeomorphic to a triple product of spheres or to the connected sum of sphere products (\cite {GoLdM, LdM2}); for $k>2$ this is no longer the case (\cite{BBCG}, \cite{BM}) but there are large families of them which are again connected sums of spheres products (\cite{GLdM}).\\

\noi The generic condition, known as \emph{weak hyperbolicity} and equivalent to regularity of the manifold, is the following:\\

\noi\emph{If $J\subset{1,\dots,m}$ has $k$ or fewer elements then the origin is not in the convex hull of the $\lam_i$ with $i\in J$.}\\

\noi A crucial feature of these manifolds is that they admit natural group actions: all of them admit $\Z^{^n}_2$ actions by changing the signs of the coordinates.\\ 

\noi Their complex versions in $\C^n$, which we denote by $Z^{\C}$ or $Z^{\C}(\Lam)$  
(\textcolor{red}{denoted by $M_1(\Lam)$ in the previous sections}),

$$
\sum_{i=1}^n\lam_i|z_i|^2=0,\quad\quad \sum_{i=1}^n|z_i|^2=1 , \,\,\, \rm{where}\,\, \lam_i\in\C^k,\, i=1,\dots,n
\label{intersection_quadrics} 
$$
\noi (now known as \emph{moment-angle manifolds}) admit natural actions of the $n$-torus $\T^n$
\[
\left((u_i,\hdots,u_n),(z_i,\hdots,z_n)\right)\mapsto(u_1z_1,\hdots.u_nz_n)
\]

The quotient can be identified in both cases with the polytope $\p$ given by

$$
\sum_{i=1}^n\lam_ir_i=0,\quad\quad\sum_{i=1}^nr_i=1,\quad\quad r_i\geqslant 0.
$$
that determines completely the varieties (so we can use the notations $Z(\p)$ and $Z^{^\C}(\p)$ for them) as well as the actions. The weak hyperbolicity condition implies that $\p$ is a simple polytope and any simple polytope  can be realized as such a quotient. 

\noi The facets of $\p$ are its non-empty intersections with the coordinate hyperplanes. If all such intersections are non-empty $Z$ and $Z^{^\C}$ fall under the general concept of \emph{generalized moment-angle complexes} (\cite{BBCG}).\\

\noi If we take the quotient of $\ZLam$ by the scalar action of $\s^{^1}$: 
$$
\vdu(\Lam)=\ZLam/\s^{^1},
$$

\noi we obtain a compact, smooth LVM manifold  $\vdu(\Lam)\subset \CP^{^{n-1}}$.\\

\noi When $k$ is even, $\vdu(\Lam)$ and $\ZLam\times\s^{^1}$ have natural complex structures and so does $\ZLam$ itself when $k$ is odd, but admit symplectic structures only in a few well-known cases (\cite{LdMV}, \cite{Me}). \\

\noi An open book construction was used to describe the topology of $Z$ for $k=2$ in some cases not covered by Theorem 2 in \cite{LdM2}. In \cite{GLdM} it is a principal technique for studying the case $k>2$. In section I-1 \ref{I-1} we recall this construction, underlining the case of moment-angle manifolds:\\

\noi \textit{If $\p$ is a simple convex polytope and $F$ one of its facets, $Z^{\C}(\p)$ admits an open book decomposition with binding $Z^{^\C}(F)$ and  trivial mo\-no\-dro\-my.}\\

\noi When $k=2$, the varieties $Z$ and $\ZLam$ can be put in a normal form given by an \emph{odd cyclic partition} (see section I-1 \ref{I-1}) and they are  diffeomorphic to a triple product of spheres or to the connected sum of sphere products (see \cite{LdM2, GLdM}). Using the same normal form, we give a topological description of the leaves 
of their open book decompositions which is complete in the case of moment-angle manifolds:\\

\noi \textit{The leaf of the open book decomposition of $\ZLam$ is the interior of:}
\begin{itemize}
\item [a)]\noi  \textit{a product $\s^{2n_2-1}\times\s^{2n_3-1}\times\disc^{2n_1-2}$,}

\item [b)]\noi  \textit{a connected sum along the boundary of products of the form $\s^p\times\disc^{2n-p-4}$,}

\item [c)]\noi  \textit{in some cases, there may appear summands of the form:}
\vspace{0.025in}

\noi \textit{a punctured product of spheres $\s^{2p-1}\times\s^{2n-2p-3}\backslash\disc^{2n-4}$ or}
\vspace{0.025in}

\noi \textit{the exterior of an embedding $\,\,\,\,\,\,\,\s^{2q-1}\times\s^{2r-1} \subset \,\,\,\s^{2n-4}.$}
\end{itemize}

\noi The precise result (Theorem \ref{pageC} in section \ref{I-1}) follows from Theorem \ref{3} in section I-4, a general theorem that gives the topological description of the \emph{half} real varieties $Z_{_+}=Z\cap\{x_1\geqslant 0\}$,  and requires additional dimensional and connectivity hypotheses that should be avoidable using the methods of \cite{GoLdM}. Some of the proofs follow directly from the result in \cite{LdM2}, while other ones require the use of some parts of its proof. All these manifolds with boundary are also generalized moment-angle complexes.\\

\noi In part II \ref{S3}, using a recent deep result about contact forms due to Borman, Eliashberg and Murphy \cite{BEM}, we show that every odd dimensional moment-angle manifold admits a contact structure. This is surprising since even dimensional moment-angle manifolds admit symplectic structures only in a few well-known cases. We also show this for large families of more general odd-dimensional intersections of quadrics by a different argument.\\

\newpage

\section*{Part I. Open book structures.}\label{open-book-structure}

\subsection*{I-1. Construction of the open books.}\label{I-1}

\noi Let $\Lam'$ be obtained from $\Lam$ by adding an extra $\lambda_1$ which we interpret as the coefficient of a new extra variable $x_0$, so we get the variety $Z'$:
$$
\lambda_1\left(x^2_0+x^2_1\right)+\sum_{i>1}\lambda_ix^2_i=0,\,\,\,\,\,\,\,\,\,\,(x^2_0+x^2_1)+\sum_{i>1}x^2_i=1.
$$

\noi Let $Z_+$ be the intersection of $Z$ with the half space $x_1\geqslant 0$. $Z_+$ admits an action of $\Z^{^{n-1}}_2$ with quotient the same $\p$: $Z_+$ can be obtained by reflecting $\p$ on all the coordinate hyperplanes except $x_1=0$. $Z_+$ is a manifold with boundary $Z_0$ which is the intersection of $Z$ with the subspace $x_1=0$.\\

\noi Consider the action of $\s^1$ on $Z'$ by rotation of the coordinates  $\left(x_0,x_1\right)$ . This action fixes the points of $Z_0$ and all its other orbits cut $Z_+$ transversely in exactly one point. So $Z'$ is the open book with binding $Z_0$, page $Z_+$ and trivial monodromy:

\begin{theorem}\label{bookR}
i) Every manifold $Z'$ is an open book with trivial monodromy, bin\-ding $Z_0$ and page $Z_+$.\\

\noi ii) If $\p$ is a simple convex polytope and $F$ is one of its facets, there is an open book decomposition of $Z^{^\C}(\p)$ with binding $Z^{^\C}(F)$ and trivial monodromy.
\end{theorem}

\noi (ii) follows because if we write the equations of $Z^{^\C}(\p)$ in real coordinates, we get terms $\lambda_i (x_i^2+y_i^2)$ so each $\lambda_i$ appears twice as a coefficient and $Z^{^\C}(\p)$ is a variety of the type $Z'$ in several ways. It is then an open book with binding the manifold $Z^{^\C}_0(\p)$ obtained by taking $z_i=0$.\\

\noi When $k=2$ it can be assumed $\Lam$ is one of the following normal forms (see \cite{LdM2}): Take $n=n_1+\dots+n_{2\ell+1}$ a partition of $n$ into an odd number of positive integers. Consider the configuration $\Lam$ consisting of the vertices of a regular polygon with $(2\ell+1)$ vertices, where the $i$-th vertex in the cyclic order appears with multiplicity $n_{i}$. \\
\begin{center}
\begin{figure}[h]
\vspace{-0.3in}\centerline{\includegraphics[height=3.5cm]{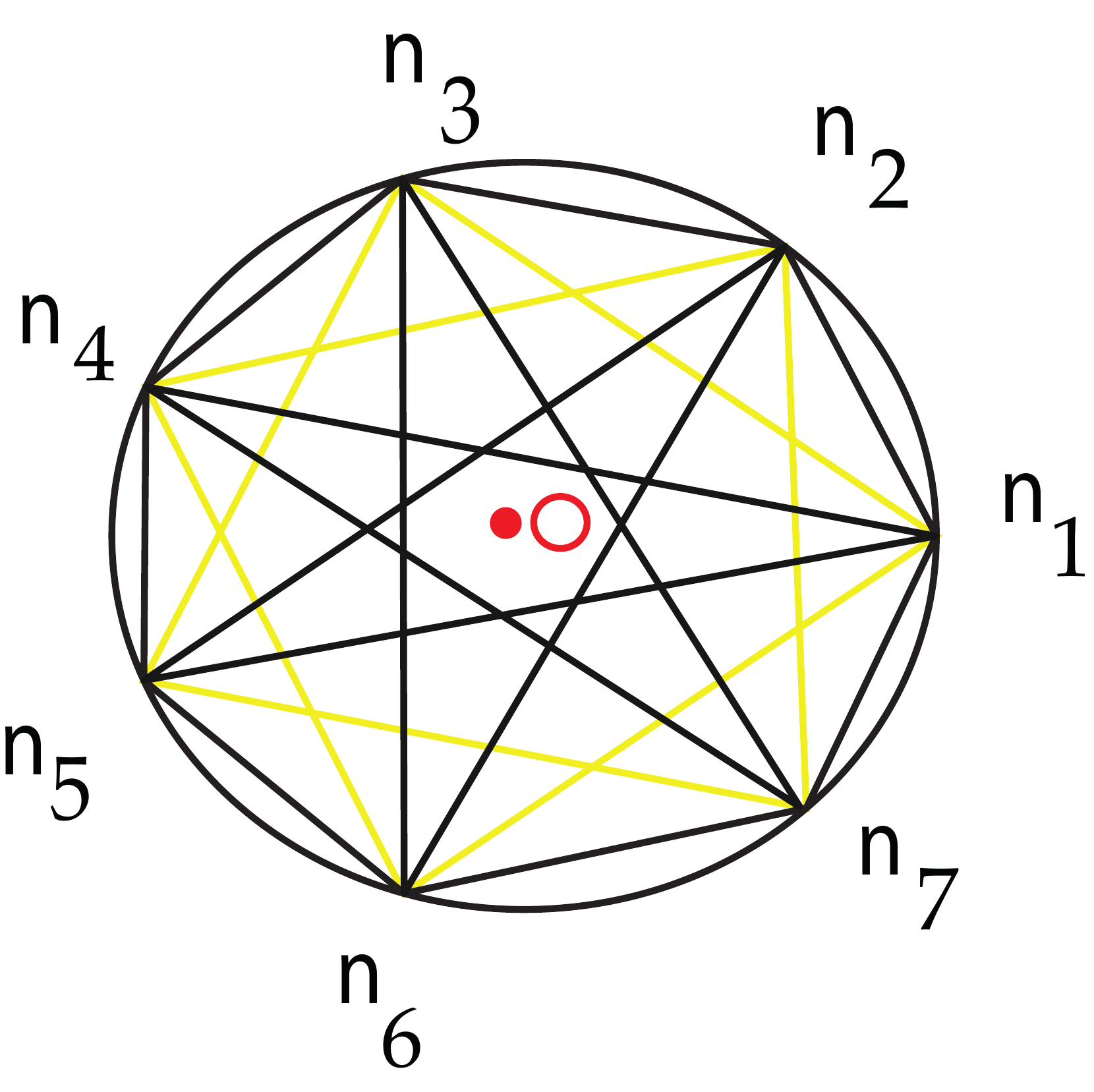}}
\end{figure}
\end{center}   
\vspace{-0.35in}
\noi The topology of $Z$ and $\ZLam$ can be completely described in terms of the numbers $d_i=n_i+\dots+n_{i+\ell-1}$, i.e., the sums of $\ell$ consecutive $n_i$ in the cyclic order of the partition (see \cite{LdM2}, \cite{GoLdM} 
and section I-1 \ref{I-1}):
\begin{center}
\begin{figure}[h]
\centerline{\includegraphics[height=3.5cm]{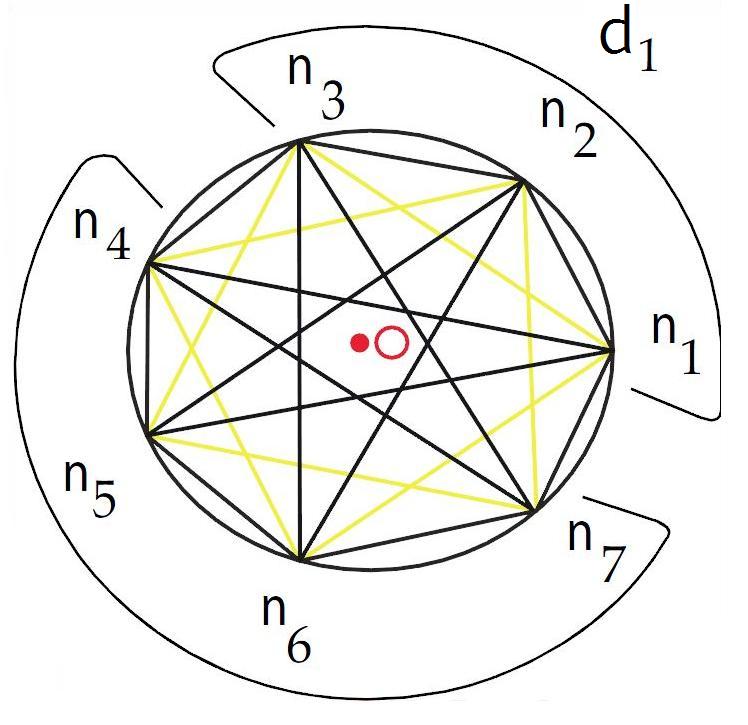}}
\end{figure}
\end{center}  
\vspace{-0.3in}
\noi For $\ell=1$: $Z=\s^{n_1-1}\times\s^{n_2-1}\times\s^{n_3-1},\,\,\,Z^\C=\s^{2n_1-1}\times\s^{2n_2-1}\times\s^{2n_3-1}$.\\

\noi For $\ell>1$: $Z=\gato_{j=1}^{2\ell+1}\left(\s^{d_i-1}\times\s^{n-d_i-2}\right),\,\,\, Z^\C=\gato_{j=1}^{2\ell+1}\left(\s^{2d_i-1}\times\s^{2n-2d_i-2}\right)$.\\

\noi Now we have a similar description of the topology of the leaves in all moment-angle manifolds, where $\coprod$ denotes connected sum along the boundary and $\te_{2n_2-1,2n_4-1}^{2n-4}$ is the exterior of $\s^{2n_2-1}\times\s^{2n_5-1}$ in $\s^{2n-4}$ (see section I-3 \ref{I-3}):

\begin{theorem}\label{pageC}
Let $k=2$, and consider the manifold $Z^{^\C}$ corresponding to the cyclic partition $n=n_1+\dots+n_{2\ell+1}$. Consider the open book decomposition of $Z^{^\C}$ corresponding to the binding at $z_1=0$, as given by Theorem 1. Then the leaf of this decomposition is diffeomorphic to the interior of:

\begin{itemize}
\item [a)] If $\ell=1$, the product 
$$
\s^{2n_2-1}\times\s^{2n_3-1}\times\disc^{2n_1-2}.
$$

\item [b)] If $\ell>1$ and $n_1>1$, the connected sum along the boundary of $2\ell+1$ manifolds:
$$
\coprod_{i=2}^{\ell+2}\left(\s^{2d_i-1}\times\disc^{2n-2d_i-3}\right)\coprod\coprod_{i=\ell+3}^1\left(\disc^{2d_i-2}\times\s^{2n-2d_i-2}\right).
$$

\item [c)] If $n_1=1$ and $\ell>2$, the connected sum along the boundary of $2\ell$ manifolds:
$$
\coprod_{i=3}^{\ell+1}\left(\s^{2d_i-1}\times\disc^{2n-2d_i-3}\right)\coprod\coprod_{i=\ell+3}^1\left(\disc^{2d_i-2}\times\s^{2n-2d_i-2}\right)
$$
$$
\coprod\left(\s^{2d_2-1}\times\s^{2d_{\ell+2}-1}\backslash\disc^{2n-4}\right).
$$

\item [d)] If $n_1=1$ and $\ell=2$, the connected sum along the boundary of two manifolds:
$$
\left(\s^{2d_2-1}\times\s^{2d_4-1}\backslash \disc^{2n-4}\right)\coprod\te_{2n_2-1,2n_5-1}^{2n-4}.
$$

\end{itemize}
\end{theorem} 

\noi Theorem \ref{pageC}
 will follow from its real version (see Theorem \ref{3}). It follows also that in cases c) and d) the product of the leaf with an open interval is diffeomorphic to the interior of a connected sum along the boundary of the type of case b).\\

\noi For $k>2$, the topology of moment-angle manifolds and their leaves is much more complicated and it seems hopeless to give a complete description of them: they may have non-trivial triple Massey products (\cite{Ba}), any amount of torsion in their homology (\cite{BM}) or may be a different kind of open books (\cite{GLdM}). Nevertheless, it is plausible that a description of their leaves as above may be possible for large families of them in the spirit of \cite{GLdM}.\\

\noi The manifold $\vdu(\Lam)$, defined in the introduction, also admits an open book decomposition, since the $\s^1$ action on the first coordinate commutes with the diagonal one.\\

\noi Let
$$
\pi_{_{\Lambda}}:\ZLam\to\vdu(\Lam),
$$
denote the canonical projection.\\

\noi Consider now the open book decomposition of $Z^{^\C}$ described above, co\-rres\-pon\-ding to the variable $z_{_1}$. If $\Lam_{_0}$ is obtained from $\Lam$ by removing $\lambda_{_1}$ it is clear that the diagonal $\s^{^1}$-action on $Z^{^\C}$ has the property that each orbit intersects each page in a unique point and at all of its points this page is intersected tranversally by the orbits. This implies that the restriction of the canonical projection $\pi_{_{\Lambda}}$ to each page is a diffeomorphism onto its image $\vdu(\Lam)-\vdu(\Lam_{_0})$.\\

\noi For $k$ even we therefore obtain, since $\vdu(\Lam)-\vdu(\Lam_{_0})$ has a complex structure:\\

\noi \emph{For $k$ even, the page of the open book decomposition of $Z^{^\C}(\pmb{\Lambda})$ in Theorem 2 with binding $Z^{^\C}_{_0}(\pmb{\Lambda}_{_0})$ admits a natural complex structure which makes it biholomorphic to $\vdu(\pmb{\Lambda})-\vdu(\pmb{\Lambda}_{_0})$.}\\

\noi \emph{For $k$ odd, both the whole manifold and the binding admit natural complex structures.}\\

\noi So we have a very nice open book decomposition of every moment-angle ma\-nifold. Unfortunately, it does not have the right properties to help in the cons\-truction of contact structures on them.\\

\noi The topology of these manifolds and of the leaves of their foliations is more complicated, even for $k=2$, and only some cases have been described (see \cite{LdMV} for the simpler ones).

\subsection*{I-2. Homology of intersections of quadrics and their halves.}\label{I 2}

We recall here the results of \cite{LdM2}, whose proofs are equally valid for any intersection of quadrics and not only  for $k=2$.\\

\noi Let $Z=Z(\Lam)\subset \R^n$ as before, $\p$ its associated polytope  and $F_1,\dots,F_n$ the intersections of $\p$ with the coordinate hyperplanes $x_i=0$ (some of which might be empty).\\

\noi Let $g_i$ be the reflection on the $i$-th coordinate hyperplane and for $J\subset\{1,\dots,n\}$ let $g_J$ be the composition of the $g_i$ with $i\in J$. Let also $F_J$ be the intersection of the $F_i$ for $i\in  J$.\\

\noi The polytope $\p$, all its faces (the non-empty $F_J$) and all their combined reflections on the coordinate hyperplanes form a cell decomposition of $Z$. Then the elements $g_J(F_L)$ with non-empty $F_L$ generate the chain groups $C_\ast(Z)$, where to avoid repetitions one has to ask $J\cap L=\emptyset$ (since $g_i$ acts trivially on $F_i$).\\

\noi A more useful basis is given as follows: let $h_i=1-g_i$ and $h_J$ be  the product of the $h_i$ with $i\in J$. The elements $h_J(F_L)$ with $J\cap L=\emptyset$ are also a basis, with the advantage that $h_JC_\ast(Z)$ is a chain subcomplex of $C_\ast(Z)$ for every $J$ and, since $h_i$ annihilates $F_i$ and all its subfaces, this subcomplex can be identified with the chain complex $C_\ast(\p,\p_J)$, where $\p_J$ is the union of all the $F_i$ with $i\in J$. It follows that
$$
H_\ast(Z)\approx\oplus_JH_\ast(\p,\p_J).
$$

\noi For the manifold $Z_+$ we start also with the faces of $\p$, but we cannot reflect them in the subspace $x_1=0$. This means we miss the classes $h_J(F_L)$ where $1\in J$ and we get\footnote{The retraction $Z\to Z_{_+}$ induces an epimorphism in homology and fundamental group.}
$$H_\ast(Z_{_+})\approx\oplus_{1\notin J} H_\ast(\p,\p_J).$$

\medskip
\noi To compute the homology of $\ZLam$ one can just take that of its real version (with each $\lambda_i$ duplicated) or directly using instead of the basis $h_J(F_L)$ with $J\cap L=\emptyset$ the basis of (singular) cells $F_L \times T_J$ (with $J\cap L=\emptyset$) where $T_J$ is the subtorus of $T^n$ which is the product of its $i-$th factors with $i\in J$. This gives the splitting 

$$
H_i(\ZLam)\approx\oplus_J H_{i-\vert J \vert}(\p,\p_J).
$$

(See \cite{BM}).\\

\noi These splittings have the same summands as the ones in \cite{BBCG} derived from the homotopy splitting of $\Sigma Z$. Even if it is not clear that they are the \emph{same} splitting, having two such with different geometric interpretations is most valuable.

\subsection*{I-3. The space $\te_{p,q}^m$.}\label{I-3}

\bigskip

\noi Consider the standard embedding of $\s^{p}\times \s^q$ in $\s^m$, $m>p+q$ given by 

$$\s^{p}\times \s^{q}\subset \R^{p+1}\times \R^{q+1}=\R^{p+q+2}\subset \R^{m+1}.$$
whose image lies in the $m$-sphere of radius $\sqrt{2}$.\\

\noi We will denote by $\te_{p,q}^m$ the exterior of this embedding, i.e., the complement in $S^m$ of the open tubular neighborhood $U=int(\s^{p}\times \s^q\times \disc^{m-p-q})\subset \s^{m}$. Observe that the boundary of $\te_{p,q}^m$ is $\s^{p}\times \s^q\times \s^{m-p-q-1}$ and that the classes $[\s^{m-p-q-1}]$, $[\s^{p}\times \s^{m-p-q-1}]$ and $[\s^{q}\times \s^{m-p-q-1}]$ are the ones bellow the top dimension that go to zero in the homology of $U$. By Alexander duality, the images of these classes freely generate the homology of $\te_{p,q}^m$.\\

\noi Theorem A2.2 of \cite{GLdM} tells that, under adequate hypotheses (and probably always) $\te_{p,q}^m \times \disc^1$ is diffeomorphic to a connected sum along the boundary of products of the type $\s^a \times \disc^{m+1-a}$.\\

\noi Under some conditions (and probably always), $\te_{p,q}^m$ is characterized by its boundary and its homology properties: Let $X$ be a smooth compact manifold with boundary $\s^{p}\times \s^q\times \s^{m-p-q-1}$ and $\iota$ the inclusion $\partial X\subset X$.\\

\noi \textbf{Lemma.} \textit{Assume that}
\vspace{0.1in}

(i) \textit{$X$ and $\partial X$ are simply connected.}
\vspace{0.1in}

(ii) \textit{the classes $\iota_*[\s^{m-p-q-1}]$, $\iota_*[\s^{p}\times \s^{m-p-q-1}]$ and $\iota_*[\s^{q}\times \s^{m-p-q-1}]$ freely generate the homology of $X$.}
\vspace{0.1in}

\textit{Then $X$ is diffeomorphic to $\te_{p,q}^m$.}\\

\noi \textbf{Proof:} Observe that condition (i) implies that $p,q,m-p-q-1 \ge 2$ so $dim(X)=m \ge 7$. Consider the following subset of $\partial X$:
$$K=\s^{p}\times * \times \s^{m-p-q-1} \cup * \times \s^{q}\times \s^{m-p-q-1}$$
and embed $K$ into the interior of $X$ as $K\times \{1/2\}$ with respect to a collar neighborhood $\partial X \times [0,1)$ of $\partial X$. Finally, let $V$ be a smooth regular neighborhood (\cite{H}) of $K\times \{1/2\}$ in $\partial X \times [0,1)$.\\

\noi Now, the inclusion $V \subset X$ induces an isomorphism in homology. Since the codimension of $K$ in $X$ is equal to $1 + min(p,q)\ge 3$, $X\setminus int(V)$ is simply connected and therefore an h-cobordism, so $X$ is diffeomorphic to $V$.\\

\noi Since $\te_{p,q}^m$ verifies the same properties as $X$, the above construction \textit{with the same $V$} shows that $\te_{p,q}^m$ is also diffeomorphic to $V$ and the Lemma is proved.

\subsection*{I-4 Topology of $Z$ and $Z_+$ when $k=2$}\label{I 4}

{\noi For $k=2$ and $\ell=1$ a simple computation shows that
$$
Z_+=\disc^{n_1-1}\times\s^{n_2-1}\times\s^{n_3-1}.
$$}

\noi For the case $\ell>1$ we recall here the main steps in the proof of the result about the topology of $Z$ in \cite{LdM2}, underlining those that are needed to determine the topology of $Z_+$. For the cyclic partition $n=n_1+\dots+n_{2\ell+1}$ we will denote by $J_i$ the set of indices corresponding to the $n_i$ copies of the $i$-th vertex of the polygon in the normal form. Let also $D_i=J_i\cup\dots\cup J_{i+\ell-1}$ and $\tilde{D}_i$ its complement.

\noi It is shown in \cite{LdM2} that for $k=2$, the pairs $(\p,\p_J)$ with non-trivial homology are those where $J$ consists of $\ell$ or $\ell+1$ consecutive classes, that is, those where $J$ is some $D_i$ or $\tilde{D}_i$. In those cases there is just one dimension where the homology is non-trivial and it is infinite cyclic.\\

\noi  In the case of $D_i$ that homology group is in dimension $d_i-1$ where $d_i=n_i+\dots+n_{i+\ell-1}$ is the length of $D_i$. A generator is given by the face $F_{L_i}$ where
$$L_i=\tilde{D}_i\backslash\left(\{j_{i-1}\}\cup\{j_{i+\ell}\}\right)$$
and $j_{i-1}\in J_{i-1}$, $j_{i+\ell}\in J_{i+\ell}$ are any two indices in the extreme classes of $\tilde{D}_i$ (in other words, those conti\-guous to $D_i$).\\

\begin{center}
\begin{figure}[h]
\vspace{-0.25in}\centerline{\includegraphics[height=3.5cm]{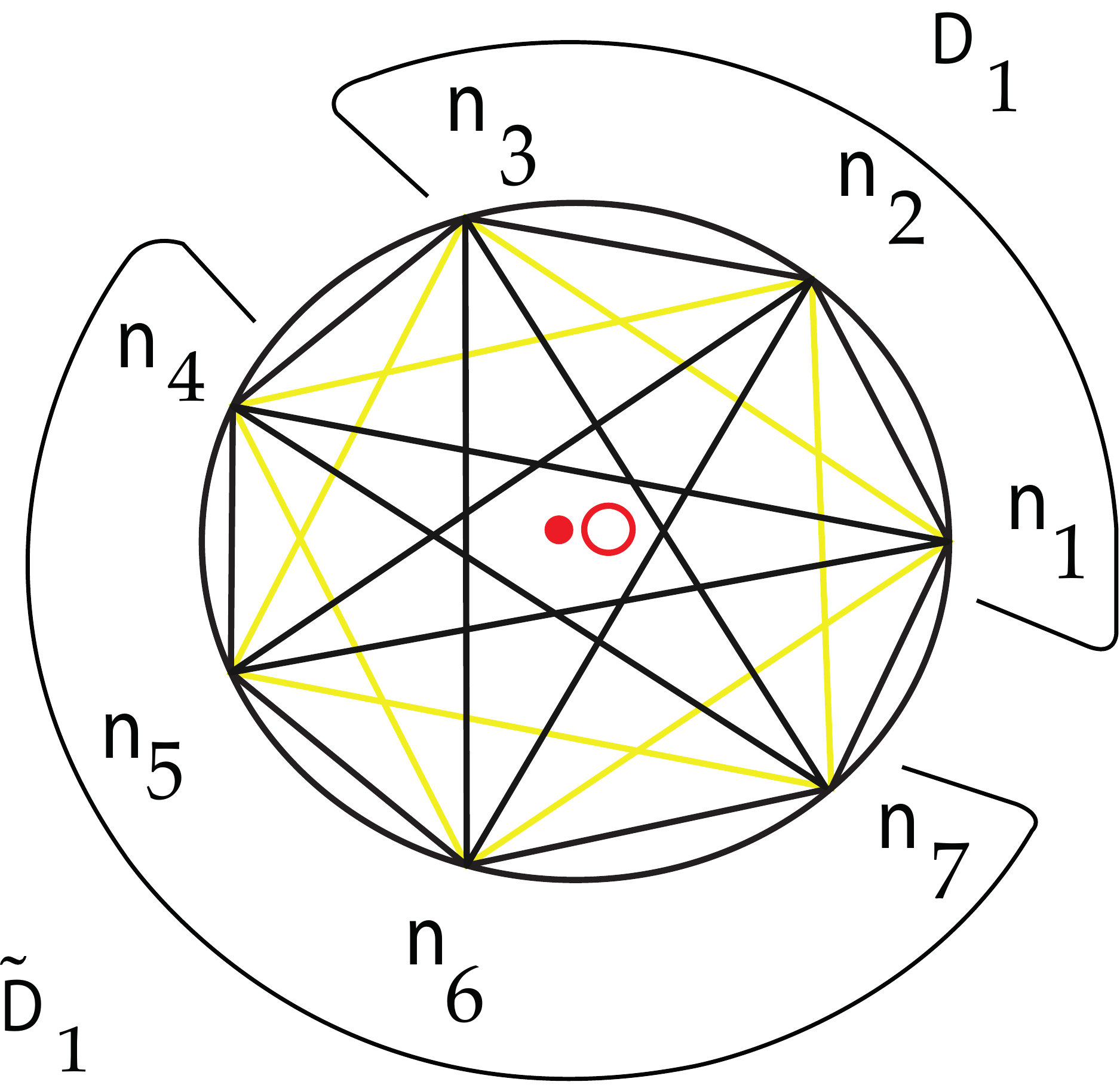}}
\end{figure}
\end{center}   
\vspace{-0.4in}
 
\noi $F_{L_i}$ is non empty of dimension $d_i-1$. It is not in $\p_{D_i}$, but its boundary is. Therefore it represents a homology class in $H_{d_i-1}(\p,\p_{D_i})$, which defines a generator $h_{D_i}F_{L_i}$ of $H_{d_i-1}(Z)$. Since $F_{L_i}$ has exactly $d_i$ facets it is a $(d_i-1)$-simplex so when reflected in all the coordinate subspaces containing those facets we obtain a sphere, which clearly represents $h_{D_i}F_{L_i}\in H_{d_i-1}(Z)$.\\
 
 \noi The class corresponding to $\tilde{D}_i$ is in dimension $n-d_i-2$ and is represented by the face $F_{\tilde{L}_i}$, where $\tilde{L}_i=D_i\backslash\{j\}$ and $j$ is any index in one of the extreme classes of $D_i$. It represents a generator of $H_{n-d_i-2}(Z)$, but now it is a product of spheres. For $\ell=1$ this cannot be avoided, but for $\ell>1$, with a good choice of $j$ and a surgery, it can be represented by a sphere (this also follows from \cite {GLdM}). We will not make use of this class in what follows.\\
 
 \noi The final result is that, if $\ell>1$, all the homology of $Z$ below the top dimension can be represented by embedded spheres with trivial normal bundle.\\
 
\noi Let $Z_+'$ be the manifold with boundary obtained by setting $x_0\ge0$ in $Z'$ (as defined in section I-1 \ref{I-1}). Then $Z_+'$ can be deformed down to $Z_+$ by folding gradually the half-plane $x_0\ge 0, x_1$ onto the ray $x_1\ge 0$. This shows that the inclusion $Z \subset Z'_+$ induces an epimorphism in homology so one can represent all the classes in a basis of $H_\ast(Z_{_+}')$ by embedded spheres with trivial normal bundle. Those spheres can be assumed to be disjoint since they all come from the boundary $Z$ and can be placed at different levels of a collar neighborhood. Finally, one forms a manifold $Q$ with boundary by joining disjoint tubular neighborhoods of those spheres by a minimal set of tubes and then the inclusion $Q\subset Z'_+$ induces an isomorphism in homology. If $Z$ is simply connected and of dimension at least 5, then $Z'_+$ minus the interior of $Q$ is an $h$-cobordism and therefore $Z$ is diffeomorphic to the boundary of $Q$ which is a connected sum of spheres products. Knowing its homology we can tell the dimensions of those spheres:\\

\noi \emph{If $\ell>1$ and $Z$ is simply connected of dimension at least $5$, then:}
$$
Z=\gato_{j=1}^{2\ell+1}\left(\s^{d_j-1}\times\s^{n-d_j-2}\right).
$$
For the moment-angle manifold $Z^\C$ this formula gives, without any restrictions
$$
Z^\C=\gato_{j=1}^{2\ell+1}\left(\s^{2d_j-1}\times\s^{2n-2d_j-2}\right).
$$
\noi (In \cite{GL} this has recently been proved without any restrictions also on $Z$).\\

\noi The topology of $Z_+'$ is implicit in the above proof: $Z_+'$ is diffeomorphic to $Q$ and therefore it is a connected sum along the boundary of manifolds of the form $\s^p\times\disc^{n-3-p}$. Since any $Z$  with $n_1>1$ is such a $Z'$ we have:\\

\noi \emph{If $Z_0$ is simply connected of dimension at least $5$, and $\ell>1$, $n_1>1$ then:}
$$
Z_+=\coprod_{i=2}^{\ell+2}\left(\s^{d_i-1}\times\disc^{n-d_i-2}\right)\coprod\coprod_{i=\ell+3}^1\left(\disc^{d_i-1}\times\s^{n-d_i-2}\right).
$$
The clases $D_i$ and $\tilde{D}_i$ that now give no homology are the ones that contain $n_{1}$.\\

\noi The case $n_1=1$ is different. When $n_1>1$ the inclusion $Z_0\subset Z_+$ induces an epimorphism in homology (since it is of the type $Z\subset Z'_+$). This is not the case for $n_1=1$: for the partition $5=1+1+1+1+1$, the polytope $\p$ is a pentagon and an Euler characteristic computation (from a cell decomposition formed by $\p$ and its reflections) shows that $Z$ is the surface of genus $5$. Now $Z_0$ has partition $4=1+2+1$ and consists of four copies of $\s^1$. From this, $Z_+$ must be a torus minus four disks that can be seen as the connected sum of a sphere minus four disks (all whose homology comes from the boundary) and a torus that carries the homology not coming from the boundary.\\

\noi In general, when $n_1=1$ $Z_0$ is given by a normal form with $2\ell-1$ classes, has $4\ell-2$ homology generators below the top dimension, only half of which survive in $Z_+$. But $Z_+$ has $2\ell+1$ homology generators, so two of them do not come from its boundary and actually form a handle. 

\noi To be more precise, the removal of the element $1\in J_1$ allows the opposite classes $J_{\ell+1}$ and $J_{\ell+2}$ to be joined into one without breaking the weak hyperbolicity condition.
\begin{center}
\begin{figure}[h]
\vspace{-0.25in}\centerline{\includegraphics[height=3.5cm]{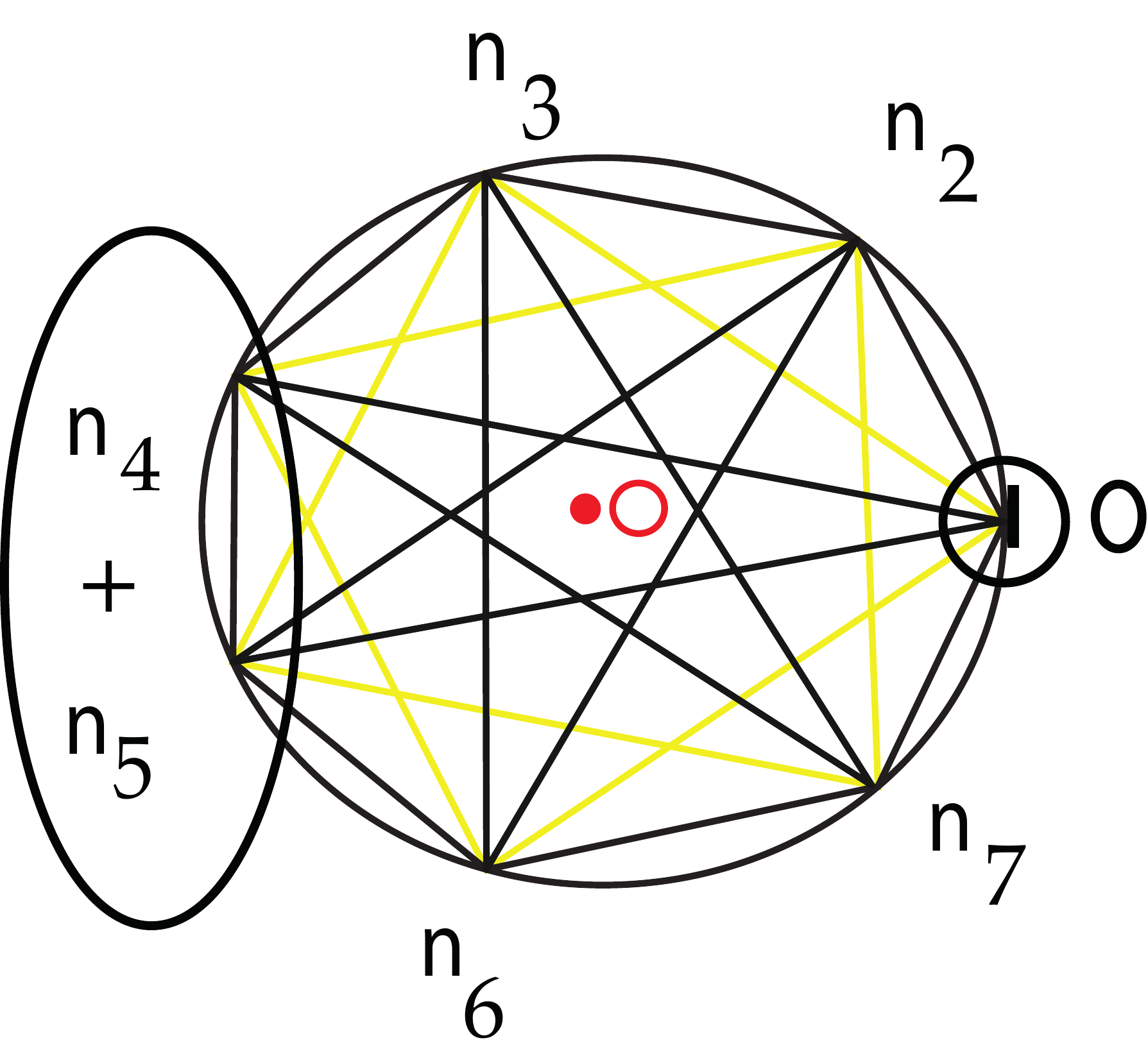}}
\end{figure}
\end{center}   
\vspace{-0.35in}
\noi Therefore $Z_0$ has fewer such classes and $D_2=J_2\cup\dots\cup J_{\ell+1}$, which gives a generator of $H_\ast(Z_{_+})$, does not give anything in $H_\ast(Z_{_0})$ because there \emph{it is not a union of classes}  (it lacks the elements of $J_{\ell+2}$ to be so).\\

\noi The two classes in $H_\ast(Z_{_+})$ missing in $H_\ast(Z_{_0})$ are thus those corresponding to $J=D_2$ and $J=D_{\ell+2}$; all the others contain both $J_{\ell+1}$ and $J_{\ell+2}$ and thus live in $H_\ast(Z_{_0})$.\\

\noi As shown above, these two classes are represented by embedded spheres in $Z_+$ with trivial normal bundle built from the cells $F_{L_2}$ and $F_{L_{\ell+2}}$ by reflection. Now $F_{L_2}\cap F_{L_{\ell+2}}$ is a single vertex $v$, all coordinates except $x_1$, $x_{j_{\ell+1}}$, $x_{j_{\ell+2}}$  being $0$.
\begin{center}
\begin{figure}[h]
\vspace{-0.25in}\centerline{\includegraphics[height=3.5cm]{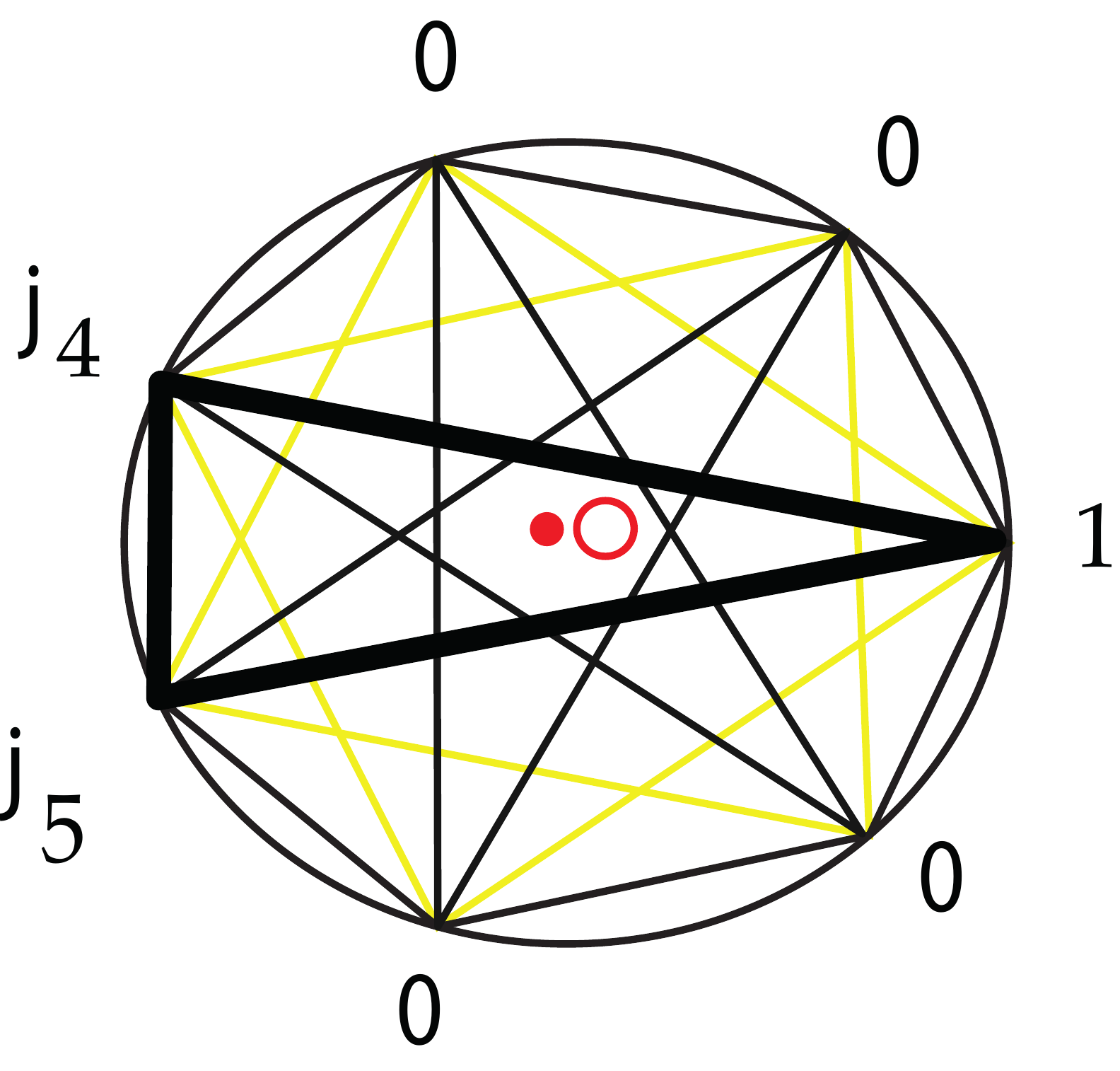}}
\end{figure}
\end{center}   
\vspace{-0.25in}
\noi The corresponding spheres are obtained by reflecting in the hyperplanes co\-rres\-pon\-ding to elements in $D_2$ and $D_{\ell+2}$, respectively. Since these sets are disjoint, the only point of intersection is the point $v$.\\

\noi Now, a neighborhood of the vertex $v$ in $\p$ looks like the first orthant of $\R^{n-3}$ where the faces $F_{L_2}$ and $F_{L_{\ell+2}}$ correspond to complementary subspaces. When reflected in all the coordinates hyperplanes of $\R^{n-3}$, one obtains a neighborhood of $v$ in $Z_+$ where those subspaces produce neighborhoods   of the two spheres. Therefore the spheres intersect transversely in that point.\\

\noi A regular neighborhood of the union of those spheres is diffeomorphic to their product minus a disk:
$$
\s^{d_2-1}\times\s^{d_{\ell+2}-1}\backslash \disc^{n-3}.
$$

\noi Joining its boundary with the boundary of $Z_+$ we see that $Z_+$ is the connected sum along the boundary of two manifolds:
$$Z_+= \s^{d_2-1}\times\s^{d_{\ell+2}-1}\backslash \disc^{n-3}\coprod X$$
where $\partial X =Z_0$ and $X$ is simply connected. 

\noi Now, all the homology of $X$ comes from its boundary which again is $Z_{0}$ , since all those classes actually live in the homology of $Z$ and are the ones corresponding to the clases $D_i$ and $\tilde{D}_i$ that do not contain $n_{1}$. Those classes also exist in the homology of $Z_{0}$ and are given by the same generators, so this part of the homology of $Z_{0}$ embeds isomorphically onto the homology of $X$.\\
  
\noi If $\ell>2$, $Z_{0}$ is a connected sum of sphere products, so the homology classes of $X$ can be represented again by disjoint products $\s^p\times\disc^{n-p-3}$ and finally we construct the analog of the manifold with boundary $Q$ and the $h$-cobordism theorem gives:\\

\noi \emph{If $Z$ is simply connected of dimension at least $6$, and $n_1=1$, $\ell>2$ then}
$$
Z_{_+}=\coprod_{i=3}^{\ell+1}\left(\s^{d_i-1}\times\disc^{n-d_i-2}\right)\coprod\coprod_{i=\ell+3}^1\left (\disc^{d_i-1}\times\s^{n-d_i-2}\right)
$$
$$
\coprod \left(\s^{d_2-1}\times\s^{d_{\ell+2}-1}\backslash \disc^{n-3}\right).
$$

\noi The homology classes of $Z_{+}$ are those corresponding to $D_{2}$, $D_{4}$ (not coming from the boundary) and to $D_{3}$, $\tilde{D}_{1}$, $\tilde{D}_{5}$. Clearly the last ones come from the classes $[\s^{n_{3}+n{4}-1}]$, $[\s^{n_{2}-1}\times \s^{n_{3}+n{4}-1}]$ and $[\s^{n_{5}-1}\times \s^{n_{3}+n{4}-1}]$ in the boundary. This means that $X$ satisfies the hypotheses of the Lemma in section I.3 with $p=n_{2}-1$, $q=n_{5}-1$ and $m=n-3$, so we can conclude that $X$ is diffeomorphic to $\te_{n_{2}-1,n_{5}-1}^{n-3}$. We have proved all the cases of the

\begin{theorem}\label{3}
Let $k=2$, and consider the manifold $Z$ corresponding to the cyclic decomposition $n=n_1+\dots+n_{2\ell+1}$ and the half manifold $Z_{_+}=Z\cap\{x_1\geqslant 0\}$. When $\ell>1$ assume $Z$ and $Z_{_0}=Z\cap\{x_1=0\}$ are simply connected and the dimension of $Z$ is at least $6$. Then $Z_+$ diffeomorphic to:
\begin{itemize}
\item [a)] If $\ell=1$, the product 
$$
\s^{n_2-1}\times\s^{n_3-1}\times\disc^{n_1-1}.
$$
\item [b)] If $\ell>1$ and $n_1>1$, the connected sum along the boundary of $2\ell+1$ manifolds:
$$
\coprod_{i=2}^{\ell+2}\left(\s^{d_i-1}\times\disc^{n-d_i-2}\right)\coprod
\coprod_{i=\ell+3}^1\left(\disc^{d_i-1}\times\s^{n-d_i-2}\right).
$$
\item [c)] If $n_1=1$ and $\ell>2$, the connected sum along the boundary of $2\ell$ manifolds:
$$
\coprod_{i=3}^{\ell+1}\left(\s^{d_i-1}\times\disc^{n-d_i-2}\right)\coprod\coprod_{i=\ell+3}^1\left(\disc^{d_i-1}\times\s^{n-d_i-2}\right)
$$
$$
\coprod\left(\s^{d_2-1}\times\s^{d_{\ell+2}-1}\backslash \disc^{n-3}\right).
$$
\item [d)] If $n_1=1$ and $\ell=2$, the connected sum along the boundary of two manifolds:
$$
\left(\s^{d_2-1}\times\s^{d_4-1}\backslash \disc^{n-3}\right) \coprod \te_{n_{2}-1,n_{5}-1}^{n-3}.
$$ 
\end{itemize}

\noi When $n_1=1$ and $\ell=2$ we have the additional complication that restricting to $x_1=0$ takes us from the \emph{pentagonal} $Z_+$ to the \emph{triangular} $Z_0$ , which is not a connected sum but a product of three spheres and not all of its homology below the middle dimension is spherical.
\begin{center}
\begin{figure}[h]
\vspace{-0.25in}\centerline{\includegraphics[height=3.5cm]{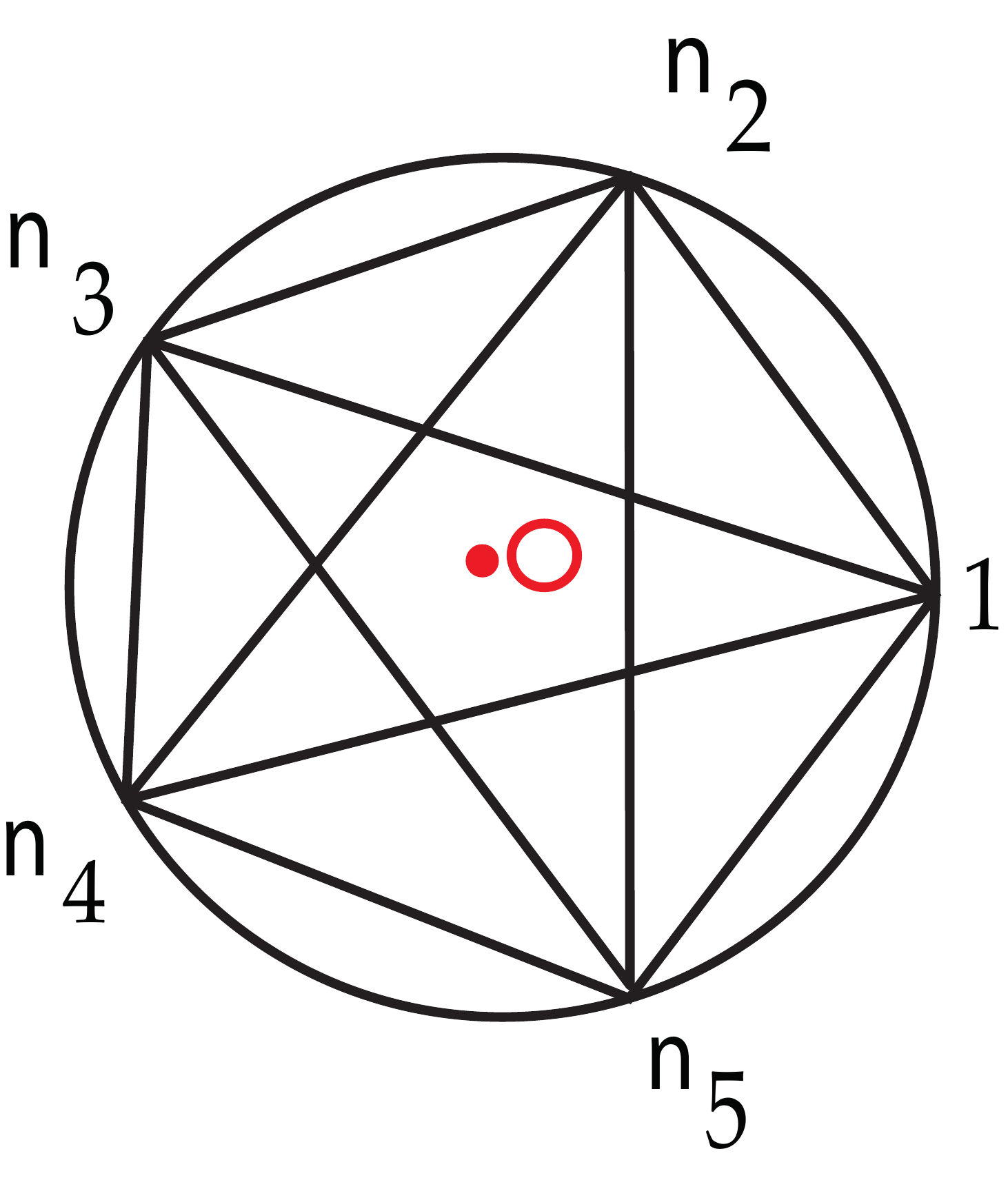}}
\end{figure}
\end{center}   
\vspace{-0.25in}
\end{theorem}

\noi Theorem \ref{3} immediately describes, under the same hypotheses, the topology of the page of the open book decomposition of $Z'$ given by Theorem 1, since this page is precisely the interior of $Z_{+}$.\\

\noi Theorem \ref{pageC} about the page of the open book decomposition of the moment-angle manifold $Z^\C$ follows also, since this page is $Z_{+}$ for $Z$ the (real) intersection of quadrics corresponding to the partition $2n-1=(2n_{1}-1) + (2n_{2}) + \dots + (2n_{2\ell+1})$. In this case all the extra hypotheses of Theorem \ref{3} hold automatically.\\

\noi Theorem \ref{3} applies also to the topological description of some \textit{smoothings} of the cones on our intersections of quadrics. In this case the normal form is not sufficient to describe all possibilities as it was in (\cite{LdM1}) where actually only the sums $d_{i}$ were needed to describe the topology or in the present work where additional information about $n_1$ only is required.\\

\section*{Part II. Some contact structures on moment-angle mani\-folds }\label{S3}

\noi The even dimensional moment-angle manifolds and the $\lv$-manifolds have the characteristic that, except for a few, well-determined cases, do not admit symplectic structures.
We will show that the odd-dimensional moment-angle manifolds (and large families of intersections of quadrics) admit contact structures.

 \begin{theorem}\label{contact} If $k$ is even, $\ZLam$ is a contact manifold.
  \end{theorem}
 
 \noi First we show that $\ZLam$ is an almost-contact manifold. Recall that a $(2n+1)$-dimensional manifold $\vd$ is called \emph{almost contact} if its tangent bundle admits a reduction to $\su(n)\times \R$. This is seen as follows: consider the fibration $\pi: \ZLam\to \vdu(\Lam)$ with fibre the circle, given by taking the quotient by the diagonal action. Since $\vdu(\Lam)$ is a complex manifold, the foliation defined by the diagonal circle action is transversally holomorphic. Therefore,  $\ZLam$ has an atlas modeled on $\C^{n-2}\times\R$ with changes of coordinates of the charts of the form
 $$
 \left(\left(z_1,\cdots,z_{n-2}\right), t\right)  \mapsto \left(h\left(z_1,\cdots,z_{n-2},t\right), g\left(z_1,\cdots,z_{n-2},t\right)\right),
 $$ 
 \noi where $h:U\to\C^{n-2}$  and $g:U\to\R$ where $U$ is an open set in $\C^{n-2}\times\R$ and, for each
 fixed $t$ the function $\left(z_1,\cdots,z_{n-2}\right)\mapsto h\left(z_1,\cdots,z_{n-2},t\right)$ is a biholomorphism onto an open set of $\C^{n-2}\times\{t\}$. This means that the differential, in the given coordinates, is represented by a matrix of the form
  $$
 \left[\begin{array}{ccc|c}
  &  &  &  \\
   & A &  & \ast \\
  &  &  & \\ \hline
  0& \dots & 0 & r
  \end{array}\right]
$$
 \noi where $\ast$ denotes a column $(n-2)$-real vector and $A\in{\gl(n-2,\C)}$. 
 The set of matrices of the above type form a subgroup
 of $\gl(2n-3,\R)$. By Gram-Schmidt this group retracts onto $\su(n-2)\times\R$.\\
 
 \noi Now it follows from \cite{BEM} that $\ZLam$ is a contact manifold and the Theorem is proved.\\
 
 \noi  In \cite{BV} it is given a different construction, in some sense more explicit, of contact structures, not on moment-angle manifolds but on certain non-diagonal generalizations of moment-angle manifolds of the type that has been studied by  G\'omez Guti\'errez and Santiago L\'opez de Medrano in \cite{GL}. It consists in the construction of a positive confoliation which is constructive and uses the heat flow method described in \cite{AltWu}.\\

The argument used there applies however for many other intersections of quadrics that are not moment angle manifolds, for which the proof of the previous Theorem need not apply:

\begin{theorem}
There are infinitely many infinite families of odd-dimensional generic intersections of quadrics that admit contact structures.
\end{theorem}

\noi First consider the odd-dimensional intersections of quadrics that are connected sums of spheres products: \\

\noi An odd dimensional product $\s^m\times\s^n$ of two spheres admits a contact structure by the following argument: let $n$ even and $m$ odd, and $n,m>2$. Without loss of generality, we suppose that $m>n$ (the other case is analogous) then $\s^m$ is an open book with binding $\s^{m-2}$ and page $\R^{m-1}$. Hence $\s^n\times\s^m$ is an open book with binding $\s^{m-2}\times\s^n$ and page $\R^{m-1}\times\s^{n}$. This page is parallelizable since 
$\R\times\s^n$ already is so. Then, since $m+n-1$ is even the page has an almost complex structure. Furthermore, by hypo\-thesis, $2n\leq{n+m}$ hence by a theorem of Eliashberg \cite{Elias} the page is Stein and is the interior of an compact manifold with contact boundary $\s^{m-2}\times\s^n$. Hence by  a theorem of E. Giroux \cite{Giroux} $\s^n\times\s^m$ is a contact manifold. 

\noi Now, it was shown by C. Meckert \cite{Meckert} and more generally by Weinstein \cite{Weinstein} (see also \cite{Elias}) that the connected sum of contact manifolds of the same dimension is a contact manifold. Therefore all odd dimensional connected sums of sphere products admit contact structures.

\noi Additionally, it was proved by F. Bourgeois in \cite{Bour} (see also Theorem 10 in \cite{Giroux}) that if a closed manifold $\vd$ admits a contact structure, then so does $\vd\times \T$. Therefore, all moment-angle manifolds of the form $Z\times\T^{2m}$, where $Z$ is a connected sum of sphere products, admit contact structures.
For every case where $Z$ is a connected sum of sphere products we have an infinite family obtained by applying construction $Z\mapsto Z'$ an infinite number of times and in the different coordinates (as well as other operations). The basic cases from which to start these infinite families constitute also a large set and the estimates on their number in each dimension keep growing. Adding to those varieties their products with tori we obtain an even larger set of cases where an odd-dimensional $Z$ admits a contact structure.

\noi Another interesting fact is that most of them (including moment-angle mani\-folds) also have an open book decomposition. However, for these open book decompositions there does not exist a contact form which is supported in the open book decomposition  like in Giroux's theorem because the pages are not Weinstein manifolds (i.e manifolds of dimension $2n$ with a Morse function with indices of critical points lesser or equal to $n$). It is possible however that the pages of the book decomposition admit Liouville structures in which case one could apply the techniques of D. McDuff (\cite{McDuff}) and P. Seidel (\cite{Seidel}) to obtain contact structures.

\section*{Acknowledgements} I would like to thank Yadira Barreto, Santiago L\'opez de Medrano, Ernesto Lupercio
and Laurent Meersseman for their suggestions, lecture notes,  and sharing their thoughts about the different aspects of the subject of these notes with me during several years.

\noi {\bf This work was partially supported by project IN106817, PAPIIT, DGAPA,

\noi Universidad Nacional Aut\'onoma de M\'exico}

\bibliographystyle{abbrve}
\bibliography{main}

\begin{thebibliography}{10}

\bibitem{AK} Abe Yu., Kopfermann K., \emph{Toroidal Groups: Line Bundles, Cohomology and Quasi-Abelian Varieties}, Lecture Notes in Mathematics, 1759.
Springer-Verlag, Berlin, 2001. viii+133 pp.

  \bibitem{AltWu} Altschuler, S. J.; Wu, L. F., \emph{On Deforming Confoliations}, J. Differential Geometry, Vol. 54, pp. 75--97, 2000.


\bibitem{Ba}  Baskakov, I. V., \emph{Massey triple products in the cohomology of moment-angle complexes}, Russian Math. Surveys 58 (2003), 1039--1041.

\bibitem{BaZ} Battaglia, F.; Zaffran, D., \emph{Foliations modeling nonrational simplicial toric varieties}.
 Int. Math. Res. Not. IMRN 2015, no. 22, 11785--11815.


  \bibitem{BBCG} Bahri, A.; Bendersky, M.; Cohen, F. R.; Gitler, S.,  \emph{The polyhedral product functor: a method of decomposition for moment-angle complexes, arrangements and related spaces}, Adv. Math., Vol. 225, no. 3, pp. 1634--1668, 2010.


\bibitem{BCS} Borisov, L.A.; Chen, L. ; Smith, G.G.,
\emph{The Orbifold Chow Ring of Toric Deligne-Mumford Stacks}. J. Amer. Math. Soc. \textbf{18} (2005), no. 1, 193--215. 

\bibitem{BEM} Borman, M.S.; Eliashberg, Y.; Murphy, E., \emph{Existence and classification of overtwisted contact structures in all dimensions}, Acta Mathematica, 215 (2015), 281-361.
 

\bibitem{BHPV} Barth, W.; 
Hulek,  K.; Peters,   C.;  Van de Ven , A., \emph{Compact complex surfaces}, Springer, Berlin, 2004 



\bibitem{BLdMV}  Barreto, Y.; López de Medrano, S.; Verjovsky, A.
\emph{Some open book and contact structures on moment-angle manifolds.} Bol. Soc. Mat. Mex. (3) 23 (2017), no. 1, 423--437 
	

\bibitem{BLV} Barreto, Y.; L\'opez de Medrano, S.; Verjovsky, A., \emph{Open Book Structures on Moment-Angle Manifolds $Z^{\mathbb C}(\Lambda)$ and Higher Dimensional Contact Manifolds}. arXiv:1303.2671.



\bibitem{BM} Bosio F.; Meersseman, L., \emph{Real quadrics in $\C^n$, complex manifolds and convex polytopes}, Acta Math., Vol. 197, pp.  53--127, 2006.

		
  
 \bibitem{Bosio} Bosio, F., \emph{Vari\'et\'es complexes compactes: une g\'en\'eralisation de la construction de Meersseman et L\'opez de Medrano-Verjovsky}. Ann. Inst. Fourier \textbf{51} (2001), no. 5, 1259--1297.
	
 
\bibitem{Bour} Bourgeois, F., \emph{Odd dimensional tori are contact manifolds}, International Mathematics Research Notices, 
no. 30, pp. 1571--1574, 2002.


\bibitem{BP} Buchstaber V. M.; Panov T. E.
\emph{Torus actions and their applications in topology and combinatorics}
AMS,  Providence R.I. USA, 2002


\bibitem{BV} Barreto, Y.; Verjovsky, A.,  \emph{Moment-angle manifolds, intersection of quadrics and higher dimensional contact manifolds}, Moscow Mathematical Journal, Vol. 14, No. 4, pp. 669--696, 2014.


\bibitem{BP1}
	Battaglia, F.; Prato, E., \emph{Simple nonrational convex polytopes via symplectic geometry}.
	Topology \textbf{40} (2001), 961--975.
	
\bibitem{BP2} Battaglia F. ; Prato E., \emph{Generalized Toric Varieties for Simple Nonrational Convex Polytopes}, Internat. Math. Res. Notices 24 (2001), 1315-1337

\bibitem{CE} Calabi, E.; Eckmann, B., \emph{A class of compact, complex manifolds which are not algebraic}, Ann. Math., Vol. 58, pp. 494--500, 1953.


\bibitem{CFZ} Cupit-Foutu, S.; Zaffran, D. \emph{Non-K\"ahler manifolds and GIT quotients}.
	Math. Z. \textbf{257} (2007), 783--797.

 
\bibitem{Co} Cox D., 
\emph{The homogeneous coordinate ring of a toric variety}
J. Algebraic Geometry, 4, 1995, 17--50
	

\bibitem{Connes} Connes, A., \emph{Non-commutative differential geometry}. Publications math\'ematiques de l'IHES, 1985, vol. \textbf{62}, no 1, p. 41--144.

\bibitem{Connes2}   Connes, A, \emph{Noncommutative geometry}. Academic Press, Inc., San Diego, CA, 1994. xiv+661 pp


\bibitem{DeG} Demailly, J. P.; Gaussier, H.
\emph{Algebraic embeddings of smooth almost complex structures.} 
J. Eur. Math. Soc. (JEMS) 19 (2017), no. 11, 3391--3419. 


\bibitem{DJ} Davis, M.; Januszkiewicz T. \emph{Convex polytopes, Coxeter orbifolds and torus actions}, Duke
Math. Journal 62 no. 2, 417--451.

\bibitem{Elias} Eliashberg, Y., \emph{Topological characterization of Stein manifolds of dimension $>2$ }, International Journal of  Mathematicas, Vol. 1, No. 1, pp. 29-46,  1990.

\bibitem{Elias3} Eliashberg, Y., \emph{Contact $3$-manifolds twenty years since J. Martinet's work}, Ann. Inst. Fourier, Vol. 42, no. 1-2,  pp. 165--192,  1992.

\bibitem{ET} Eliashberg, Y., M. and Thurston, W. P., \emph{Confoliations}, \emph{American Mathematical Society. Lectures Series}, Vol. 13, 1998.

\bibitem{Fu} Fulton, W., \emph{Introduction to toric varieties}, Princeton University Press, Princeton, 1993

\bibitem{Geiges1} Geiges, H., \emph{A brief history of contact geometry and topology},  Expositiones Mathematicae, Vol. 19, pp. 25--53, 2001.

\bibitem{GH} Griffiths, P.; Harris, J. \emph{Principles of algebraic geometry.} Pure and Applied Mathematics. Wiley-Interscience, John Wiley \& Sons, New York, 1978. xii+813 pp

\bibitem{GHS} Girbau, J.; Haefliger, A.; Sundararaman D.,
\emph{On deformations of transversely holomorphic foliations}
J. Reine Angew. Math. vol. 345, 1983, pages 122--147


\bibitem{GiMo} Giroux; E.; Mohsen, J. P., \emph{Structures de contact et fibrations symplectiques sur le cercle}, \emph {in  process}.

\bibitem{Giroux} Giroux, E., \emph{Geometrie de contact: de la dimension trois vers les dimensions superieures}, \emph{ICM}, Vol. II, pp. 405--414, 2002.



\bibitem{GL} Goldstein, R.; Lininger, L.,  \emph{A Classification of 6-Manifolds with free
$S^1$-Action,} Springer-Verlag Lecture Notes in Mathematics
298 (1971), p.316--323	



\bibitem{GLdM} Gitler, S; L\'opez de Medrano, S, \emph{Intersections of quadrics, moment-angle manifolds and connected sums}. Geom. Topol. 17 (2013), no. 3, 1497--1534
	
\bibitem{GoLdM}
G\'omez, G. V.;  L\'opez de Medrano, S,,
\emph{Stably parallelizable manifolds are complete intersections of quadrics}
pr\'epublication, 2005
	

\bibitem{GoLdM2} G\'omez, G. V; L\'opez de Medrano, S.,\emph{Topology of the intersections of quadrics II.} Bol. Soc. Mat. Mex. (3) 20 (2014), no. 2, 237--255.


\bibitem{Gray} Gray, J. W., \emph{Some global properties of contact structures}, Annals of Math, Vol. 69, no. 2, pp. 421--450, 1959.


\bibitem{Gr} Gr\"unbaum, B.,\emph{Convex polytopes,}
 Interscience,
New York, 1967

\bibitem{H} Hirsch, M.W., \emph{Smooth Regular Neighborhoods}, Annals of Mathematics, Vol. 76, No. 3 (Nov., 1962), pp. 524--530.


\bibitem{Ha} Haefliger, A., 
\emph{Deformations of transversely holomorphic flows on spheres and deformations of Hopf manifolds}
Compo. Math. vol, 55, 1985. pages 241--251


 
\bibitem{Ho} Hopf, H.,
\emph{Zur Topologie der komplexen Mannigfaltigkeiten}, Studies and essays presented to R. Courant,
New York, 1948



\bibitem{K} Kopfermann, K., \emph{Maximale Untergruppen Abelscher komplexer Liescher
Gruppen}, Schr. Math. Inst. Univ. M\"unster No. 29 (1964) iii+72 pp.

\bibitem{Kato} Kato, M, Yamada, A., \emph{Examples of Simply Connected Compact Complex
3-folds II}, Tokyo J.Math. 9(1986), p.1--28.

\bibitem{Klaus} Niederkruger, K., \emph{The plastikstufe-a generalization of the overtwisted disk to higher dimensions}, Algebraic and Geometric Topology, Vol. 6, pp. 2473--2508, 2006.



\bibitem{KLMV}Katzarkov, Ludmil; Lupercio, Ernesto; Meersseman, Laurent; Verjovsky, Alberto., \emph{The definition of a non-commutative toric variety.} Algebraic topology: applications and new directions, 223--250, Contemp. Math., 620, Amer. Math. Soc., Providence, RI, 2014.


\bibitem{KO} Niederkruger, K.; Van Koert, O., \emph{Every contact manifold can be given a non-fillable contact structure},  \emph{IMRN}, no. 23, Art. ID rnm115, 22 pp, 2007.



\bibitem{LdM1} L\'opez de Medrano S., \emph{The Space of Siegel Leaves of a
Holomorphic Vector Field,} Springer-Verlag Lecture Notes in
Mathematics 1345(1988), p.233--245.

\bibitem{LdM2} L\'opez de Medrano S., {\it The Topology of the Intersection of
Quadrics in $\R^n$,} Springer--Verlag Lecture Notes in Mathematics
1370 (1989), p.280--292.

\bibitem{LdM4} L\'opez de Medrano, S. \emph{Singularities of homogeneous quadratic mappings}. Rev. R. Acad. Cienc. Exactas Fís. Nat. Ser. A Math. RACSAM 108 (2014), no. 1, 95--112
	

\bibitem{LdM5} L\'opez de Medrano, S.  \emph{Samuel Gitler and the topology of intersections of quadrics}. Bol. Soc. Mat. Mex. (3) 23 (2017), no. 1, 5--21.

\bibitem{LdMV} L{\'o}pez de~Medrano, S.; Verjovsky, A. \emph{A new family of
		complex, compact, non-symplectic manifolds}. Bulletin of the Brazilian
	Mathematical Society \textbf{28} (1997), no.~2, 253--269.
	
\bibitem{LN} Loeb J. J.; Nicolau, M.,
\emph{Holomorphic flows and complex structures on products of
odd-dimensional spheres.} Math. Ann. vol. 306, 1996, pages 781--817

\bibitem{LN2} Loeb, J. J.; Nicolau, M. \emph{On the complex geometry of a class of non-Kählerian manifolds.} Israel J. Math. 110 (1999), 371--379.


\bibitem{LM} Lescure, F.; Meersseman L.,
\emph{Compactifications \'equivariantes non k\"ahl\'eriennes d'un groupe alg\'ebrique multiplicatif}.
 Ann. Inst. Fourier, vol, 52, 2002. pages 255--273


\bibitem{LuM} Lutz, R.; Meckert, C., \emph{Structures de contact sur certaines sph\`eres exotiques}, C. R. Acad. Sci. Paris S\'er. A-B , Vol. 282, pp. A591--A593, 1976.



\bibitem{Ma} Maeda, H.,
\emph{Some complex structures on the product of spheres},
J. Fac. Sci. Univ. Tokyo, vol. 21, pages 161--165, 1974




\bibitem{McDuff} McDuff, D., \emph{Symplectic manifolds with contact type boundaries}, Invent. Math., Vol. 103, no. 3, pp. 651--671, 1991.


\bibitem{McG} McGavran, D., \emph{Adjacent connected sums and torus actions}, Trans.
Amer. Math. Soc. 251 (1979), 235--254.


\bibitem{McM} MacMullen, P.,
\emph{On simple polytopes}
 Invent. Math. 113, 1993, 419--444


\bibitem{Meckert} Meckert, C., \emph{Forme de contact sur la somme connexe de deux vari\'et\'es de contact de dimension impare}, Annales De L'Institut Fourier, Vol. 32, no. 3 pp. 251--260, 1982.



\bibitem{Me} Meersseman, L., \emph{A new geometric construction of compact complex
		manifolds in any dimension}. Mathematische Annalen \textbf{317} (2000), 79--115.
	
\bibitem{MeThesis} Meersseman, L., \emph{Un nouveau proc\'ed\'e de construction g\'eom\'etrique de vari\'et\'es compactes, complexes, non alg\'ebriques, en dimension quelconque}, PhD Thesis, Lille, 1998.


\bibitem{Mi} Milnor J., \emph{Lectures on the h-cobordism theorem,} Princeton University Press,
1965. 



\bibitem{MM} Moerdijk, I.; Mr\v cun, J.
\emph{Introduction to Foliations and Lie Groupoids}.
	Cambridge University Press, Cambridge, 2003.


\bibitem{MMP} Mart\'inez, D.; Mu\~noz, V.; Presas, F., \emph{Open book decompositions for almost contact manifolds}, \emph{Proceedings of the XI Fall Workshop on Geometry and Physics, Publicaciones de la RSME}, Vol. 6, pp. 131--149, 2004.


\bibitem{Mo} Morimoto, A.,
\emph{On the classification of non compact complex abelian Lie groups},
Trans. Amer. Math. Soc. vol 123, 1966, pages 200--228



\bibitem{MV} Meersseman, L.; Verjovsky, A. \emph{Holomorphic principal bundles over projective toric varieties}. 
Journal f\"ur die Reine und Angewandte Mathematik \textbf{572} (2004), 57--96.

\bibitem{MV2} Meersseman, L.; Verjovsky, A. \emph{Sur les variétés LV-M}. (French) [On LV-M manifolds] Singularities II, 111--134, Contemp. Math., 475, Amer. Math. Soc., Providence, RI, 2008.

\bibitem{Or} Orlik P.,
\emph{Seifert manifolds}
LNM, vol 291, Springer,
Berlin, 1972


\bibitem{PanovKAIST} Panov T., \emph{Moment-angle manifolds and complexes.} Lecture notes KAIST'2010. Taras Panov. Trends in Mathematics - New Series. ICMS, KAIST. Vol.12 (2010), no.1, pp.43--69


\bibitem{Pr} Prato, E.,
\emph{Simple non-rational convex polytopes via symplectic geometry.} Topology 40 (2001), no. 5, 961--975


\bibitem{Presas} Presas, F., \emph{A class of non-fillable contact structures},
  \emph{Geom. Topol.}, no. 11, pp. 2203--2225, 2007.
  

\bibitem{Sc} Scheja G., \emph{Riemannsche
Hebbarkeitss\"atze f\"ur Cohomologieklassen,} Math. Ann. 144
p.345--360, 1961.

\bibitem{S} Steenrod N.,
\emph{The Topology of Fibre Bundles,} Princeton University Press,
1951.


\bibitem{St}  Sternberg, S.,  \emph{Lectures on differential geometry'}. Second edition. With an appendix by Sternberg and Victor W. Guillemin. Chelsea Publishing Co., New York, 1983. 


\bibitem{Sta} Stanley, P. R.,\emph{Combinatorics and commutative algebra}, 2nd ed., Progress in Mathematics,
vol. 41, Birkh\"auser Boston Inc., Boston, MA, 1996.


\bibitem{Seidel}  Seidel, P., \emph{Simple examples of distinct Liouville type symplectic structures},  J. Topol. Anal. , no. 1, pp,  1--5,  2011.


\bibitem{Ti}  Timorin, V.A.,
\emph{An analogue of the Hodge-Riemann relations for simple convex polytopes},
Russian Math. Surveys,
vol 54, 1999, pages 381--426


\bibitem{Wa} Wall  C.T.C., \emph{Stability, Pencils and Polytopes,} Bull.London Math.Soc. 12(1980),
p.401--421. 

\bibitem{We} Wells  R.O. Jr., \emph{Differential Analysis on Complex Manifolds,} Prentice Hall, 1973. 


\bibitem{Weinstein} Weinstein, A., \emph{Contact surgery and symplectic handlebodies}, Hokkaido Math. J., Vol. 620, no. 2,  pp. 241--251, 1991.


\bibitem{Win} Winkelnkemper, H. E., \emph{Manifolds as open books}, Bull. Am. Math. Soc., Vol. 79, pp. 45--51, 1973.

  
   
  
  
  \end{thebibliography}

\end{document}